\documentclass[a4paper,reqno,12pt]{amsart}
\usepackage[margin=1.1in]{geometry}
\usepackage{xspace,mathtools} 
\usepackage{indentfirst}
\usepackage{amsfonts}
\usepackage{amscd}
\usepackage{stmaryrd}
\usepackage{amsthm}
\usepackage{amssymb}
\usepackage{amsmath}
\usepackage{thmtools}
\usepackage{todonotes}
\usetikzlibrary{decorations.markings}
\usepackage{enumitem}
\usepackage{setspace}
\usepackage{hyperref}
\usetikzlibrary{matrix,arrows,positioning}
\usepackage{tikz}
\usepackage{tikz-cd}
\usetikzlibrary{decorations.markings,shapes.geometric}
\usepackage{xypic}

\makeatletter
\providecommand\@dotsep{5}
\renewcommand{\listoftodos}[1][\@todonotes@todolistname]{%
  \@starttoc{tdo}{#1}}
\makeatother

\newtheorem{Lem}{Lemma}[section]
\newtheorem{Prop}{Proposition}[section]
\newtheorem{Def}{Definition}[section]
\newtheorem{teorema}{Theorem}

\theoremstyle{plain}
\newtheorem{Thm}{Theorem}[section]
\newtheorem{Cor}{Corollary}[section]
\newtheorem{Prob}{Problem}[section]

\newtheorem{Ex}{Example}[section]
\newtheorem{Rem}{Remark}[section]
{\unskip\nobreak\hskip 1em plus 1fil\nobreak$\Box$
\parfillskip=0pt%
\endtrivlist}

\newcommand{\Hom}{\text{\textnormal{Hom}}}
\newcommand{\End}{\text{\textnormal{End}}}
\newcommand{\coker}{\text{\textnormal{coker}}}
\newcommand{\Aut}{\text{\textnormal{Aut}}}
\newcommand{\AUT}{\text{\textnormal{AUT}}}

\newcommand{\git}{/\!\!/}
\newcommand{\Tr}{\mathbb{T}\text{\textnormal{r}}}

\DeclareMathOperator{\im}{im}
\mathchardef\mhyphen="2D

\newcommand{\Ind}{\text{\textnormal{Ind}}}

\newcommand{\HInd}{\text{\textnormal{HInd}}}
\newcommand{\RInd}{\text{\textnormal{RInd}}}
\newcommand{\Res}{\text{\textnormal{Res}}}
\newcommand{\Vect}{\mathsf{Vect}}
\newcommand{\gen}{\text{\textnormal{gen}}}
\newcommand{\opp}{\text{\textnormal{op}}}
\newcommand{\co}{\text{\textnormal{co}}}
\newcommand{\ev}{\text{\textnormal{ev}}}

\newcommand{\Rep}{\mathsf{Rep}}

\newcommand{\AALG}{\mathcal{A}\mathsf{LG}}

\newcommand{\RRep}{\mathsf{RRep}}

\newcommand{\id}{\text{\textnormal{id}}}

\newcommand{\MMod}{\mathcal{M}\mathsf{od}}
\newcommand{\AAlg}{\mathcal{A}\mathsf{lg}}
\newcommand{\fd}{\text{\textnormal{fd}}}
\newcommand{\refl}{\text{\textnormal{ref}}}
\newcommand{\Stab}{\text{\textnormal{Stab}}}
\newcommand{\rep}{\text{\textnormal{rep}}}

\DeclareSymbolFont{extraup}{U}{zavm}{m}{n}
\DeclareMathSymbol{\varheart}{\mathalpha}{extraup}{86}
\DeclareMathSymbol{\vardiamond}{\mathalpha}{extraup}{87}
\newcommand{\Ad}{\text{\textnormal{Ad}}}

\newcommand{\bK}{{\mathbb K}}

\newcommand{\bZ}{{\mathbb Z}}

\newcommand{\bS}{{\mathbb S}}

\newcommand{\sT}{{\mathcal T}}
\newcommand{\sX}{{\mathcal X}}

\newcommand{\cG}{{\mathfrak G}} 
\renewcommand{\cH}{{\mathfrak H}}
\newcommand{\cK}{{\mathfrak K}}
\newcommand{\tcG}{{\widetilde{\mathfrak G}}}
\newcommand{\tcH}{{\widetilde{\mathfrak H}}}
\newcommand{\tcK}{{\widetilde{\mathfrak K}}}
\newcommand{\tG}{{\mathcal G}}

\newcommand{\ZTwo}{{\bZ_2}}
\newcommand{\Sym}{{\bS_t}} 
\newcommand{\Fld}{{\bK}} 
\newcommand{\spcl}{{\mathfrak{h}}} 

\newcommand{\vx}{\ensuremath{\mathbf{x}}\xspace}

\newcommand{\Mod}{\mbox{-}{\mathrm M}{\mathrm o}{\mathrm d}}

\newcommand{\Mfd}{\mbox{-}{\mathrm M}{\mathrm o}{\mathrm d}_{fd}}

\newcommand\righttwoarrow{%
        \mathrel{\vcenter{\mathsurround0pt
                \ialign{##\crcr
                        \noalign{\nointerlineskip}$\rightarrow$\crcr
                        \noalign{\nointerlineskip}$\rightarrow$\crcr
                }%
        }}%
}

\makeatletter
\newcommand{\xRrightarrow}[2][]{\ext@arrow 0359\Rrightarrowfill@{#1}{#2}}
\newcommand{\Rrightarrowfill@}{\arrowfill@\equiv\equiv\Rrightarrow}
\makeatother

\usepackage[xindy,sort=def]{glossaries}
\makeglossaries
\newglossaryentry{glALG}{sort=use,
name={\ensuremath{\mathsf{ALG}_{\Fld}}},
description={the category of unital $\Fld$-algebras and their unital algebra homomorphisms}
}
\newglossaryentry{glAALG}{sort=use,
name={\ensuremath{\AALG_{\Fld}}},
description={the Morita bicategory of $\Fld$-algebras, bimodules and intertwiners}
}
\newglossaryentry{glAAlg}{sort=use,
name={\ensuremath{\AAlg_{\Fld}}},
description={the subbicategory of $\AALG_{\Fld}$ of finite dimensional $\Fld$-algebras, finite dimensional bimodules and intertwiners}
}
\newglossaryentry{glAAlgfd}{sort=use,
name={\ensuremath{\AAlg_{\Fld}^{\fd}}},
description={the fully dualizable subbicategory of $\AALG_{\Fld}$}
}
\newglossaryentry{glAut}{sort=use,
name={\ensuremath{\Aut(A)}},
description={the group of automorphisms of an algebra $A$}
}
\newglossaryentry{glAutg}{sort=use,
name={\ensuremath{\Aut^{\gen}(A)}},
description={the group of automorphisms and antiautomorphisms of an algebra $A$}
}
\newglossaryentry{glAUT}{sort=use,
name={\ensuremath{\AUT(A)}},
description={the crossed module of automorphisms of an algebra $A$}
}
\newglossaryentry{glAUTg}{sort=use,
name={\ensuremath{\AUT^{\gen}(A)}},
description={the crossed module of generalized automorphisms of an algebra $A$}
}
\newglossaryentry{glcG}{sort=use,
name={\ensuremath{\cG= (G_2 \xrightarrow[]{\partial} G_1)}},
description={(or $\cH$, $\cK$) a crossed module with a $\ZTwo$-grading}
}
\newglossaryentry{gltcG}{sort=use,
name={\ensuremath{\tcG}},
description={(or $\tcH$, $\tcK$) the strict small 2-group associated to a crossed module $\cG$}
}
\newglossaryentry{gltG}{sort=use,
name={\ensuremath{\tG}},
description={an abstract weak 2-group}
}
\newglossaryentry{glGAAlgfd}{sort=use,
name={\ensuremath{\cG \mhyphen N\AAlg_{\Fld}^{\fd}}},
description={the Morita bicategory of  separable N-weak $\cG$-algebras}
}
\newglossaryentry{glInd}{sort=use,
  name={\ensuremath{\Ind_{\cH}^{\cG}}}, 
  description={induction pseudofunctor from $\cH \mhyphen N\AAlg_{\Fld}$ to $\cG \mhyphen N\AAlg_{\Fld}$}
  }
\newglossaryentry{glFld}{sort=use,
name={\ensuremath{\Fld}},
description={a field, usually assumed to be separably closed}
}
\newglossaryentry{glLG}{sort=use,
name={\ensuremath{\mathcal{L} \cG}},
description={a crossed module in groupoids which models the homotopy $2$-type of the free loop space of a crossed module $\cG$}
}
\newglossaryentry{glLGr}{sort=use,
name={\ensuremath{\mathcal{L}^{\refl}  \cG}},
description={a $\mathbb{Z}_2$-graded crossed module in groupoids modelling the homotopy $2$-type of a $\ZTwo$-quotient of the free loop space of the ungraded crossed module of $\cG$}
}
\newglossaryentry{glRep}{sort=use,
name={\ensuremath{\Rep_{\mathcal{V}}(\tG)}},
description={the bicategory of $2$-representations of a $2$-group $\tG$ on a bicategory $\mathcal{V}$}
}
\newglossaryentry{glRRep}{sort=use,
name={\ensuremath{\RRep_{\mathcal{V}}(\tG)}},
description={the bicategory of Real $2$-representations of a $\ZTwo$-graded $2$-group $\tG$ on a bicategory $\mathcal{V}$ with a weak duality}}
\newglossaryentry{glTr}{sort=use,
name={\ensuremath{\Tr_{\rho}}},
description={the Real categorical character of a Real $2$-representation $\rho$}
}
\newglossaryentry{gl2Ve}{sort=use,
name={\ensuremath{2 \Vect_\Fld}},
description={the bicategory of Kapranov--Voevodsky $2$-vector spaces over $\Fld$}
}
\newglossaryentry{glZTwo}{sort=use,
name={\ensuremath{\ZTwo}},
description={the multiplicative group $\{+1, -1\}$}
}
\newglossaryentry{glTh}{sort=use,
name={\ensuremath{\Theta_{A}}},
description={the Real $2$-module associated to an N-weak $\cG$-algebra $A$}
}
\newglossaryentry{glChi}{sort=use,
name={\ensuremath{\chi_{\rho}}},
description={the Real $2$-character of a Real $2$-representation $\rho$}
}
\newglossaryentry{glCi}{sort=use,
name={\ensuremath{(-)^{\circ}}},
description={a weak duality pseudofunctor on a bicategory $\mathcal{V}$}
}

\begin{document}

\title[Real Burnside rings]{Burnside rings for Real $2$-representation theory: The linear theory}

\author[D. Rumynin]{Dmitriy Rumynin}
\address{Department of Mathematics\\
University of Warwick\\
Coventry, CV4 7AL, U.K.
\newline
\hspace*{0.31cm}  Associated member of Laboratory of Algebraic Geometry\\
National Research University Higher School of Economics\\
Russia}
\email{d.rumynin@warwick.ac.uk}

\author[M.\,B. Young]{Matthew B. Young}
\address{Max Planck Institute for Mathematics\\
Vivatsgasse 7\\
53111 Bonn, Germany}
\email{myoung@mpim-bonn.mpg.de}

\date{\today}


\begin{abstract}
This paper is a fundamental study of the Real $2$-representation theory of $2$-groups. It also contains many new results in the ordinary (non-Real) case. Our framework relies on a $2$-equivariant Morita bicategory, where a novel construction of induction is introduced. We identify the Grothendieck ring of Real $2$-representations as a Real variant of the Burnside ring of the fundamental group of the $2$-group and study the Real categorical character theory. This paper unifies two previous lines of inquiry, the approach to $2$-representation theory via Morita theory and Burnside rings, initiated by the first author and Wendland, and the Real $2$-representation theory of $2$-groups, as studied by the second author.
\end{abstract}

\maketitle

\tableofcontents

\section{Introduction}

The notion of a $2$-representation of a $2$-group is introduced and studied by Barrett and Mackaay \cite{barrett2006}, Crane and Yetter \cite{crane2005} and Elgueta \cite{elgueta2007}, amongst many others. In this setting, a $2$-group acts coherently by autoequivalences of a category or, more generally, of an object of a target bicategory. An important target bicategory is
\gls{gl2Ve},
the bicategory of Kapranov--Voevodsky $2$-vector spaces over a field
\gls{glFld},
a natural categorification of the category of finite dimensional vector spaces. In this example, or in more general linear settings, there is a character theory of $2$-representations, discovered independently by Bartlett \cite{bartlett2011} and Ganter and Kapranov \cite{ganter2008}. This character theory can be seen as a concrete instance of the theory of secondary traces, as studied by T\"{o}en and Vezzosi \cite{toen2009, toen2015} and Ben-Zvi and Nadler \cite{benzvi2013}. The theory of $2$-representations, with its character theory, appears naturally in many areas of mathematics, including topological gauge theory \cite{bartlett2011, willerton2008} and equivariant elliptic cohomology \cite{hopkins2000, ganter2008, lurie2009b}. It is indispensable in traditional representation theory, for example, through its relation to conjectures of Lusztig \cite{bezrukavnikov2012} and McKay \cite{Isaacs2007}, as explained in \cite{rumynin2018}, or via the topological field theoretic approach to representations of algebraic groups \cite{benzvi2009}.

One weakness of $2$-character theory is that, in general, it cannot distinguish equivalence classes of $2$-representations. This issue is resolved by Rumynin and Wendland \cite{rumynin2018} who, under mild assumptions, describe the Grothendieck ring of $2$-representations of an essentially finite $2$-group on $2 \Vect_\Fld$ in terms of a generalized Burnside ring \cite{gunnells2012}. The mark homomorphisms of this Burnside ring not only distinguish equivalence classes of $2$-representations but also recover the $2$-characters. The perspective taken in \cite{rumynin2018} is that of Morita theory, so that $2$-groups are represented on the Morita bicategory of algebras, bimodules and intertwiners, instead of on $2 \Vect_\Fld$. This perspective is amenable to explicit calculations and constructions.

In this paper we develop a Morita-theoretic approach to Real $2$-representation theory. A number of our results are new and interesting already for ordinary $2$-representations. The word {\em Real} is capitalized, following Atiyah \cite{atiyah1968,atiyah1966}, where he distinguishes ``real'' (objects defined over ${\mathbb R}$) and ``Real'' (objects with an involution). For instance, a Real vector space is a complex vector space together with an anti-linear involution.

The Real $2$-representation theory of $2$-groups is introduced and investigated in \cite{mbyoung2018c} as a categorification of the Real representation theory of groups, as studied by Atiyah and Segal \cite{atiyah1969} and Karoubi \cite{karoubi1970} in the form of equivariant $KR$-theory. There are two distinct notions of a Real $2$-representation. In this paper we focus on {\em linear} Real $2$-representations, in which the target bicategory is endowed with an involution which is contravariant on $2$-morphisms. A second notion of an {\em anti-linear} Real $2$-representation, related to Hermitian Morita theory \cite{hahn1985}, requires the target bicategory to be linear and endowed with an involution which is fully covariant but anti-linear on $2$-morphisms. We hope to treat the anti-linear theory in consequent work. The character theory of (projective) Real $2$-representations of finite groups is also studied in \cite{mbyoung2018c}. Real $2$-representations, and their characters, appear naturally in unoriented topological field theory and orientifold string and $M$-theory \cite{mbyoung2019} and, conjecturally, in Real variants of equivariant elliptic cohomology \cite{mbyoung2018c}.

Let us now describe assiduously the content of the present paper, stating concisely the main theorems.

In Chapter~ \ref{sec:background} we introduce notation and set out our vision of the subject. Let
\gls{glcG}
be a crossed module with a $\ZTwo$-grading, that is, a group homomorphism $\pi: G_1 \rightarrow \ZTwo$ which satisfies $\im(\partial) \leq \ker(\pi)$. We allow $\pi$ to be trivial: 
in this case our results belong to the ordinary (non-Real) theory.
Associated to $\cG$ is a $\ZTwo$-graded $2$-group
\gls{gltcG}
whose action on bicategories is our primary interest.
For the introduction we assume that $G_1$ is finite since our results are cleanest under this assumption.

We start Chapter~ \ref{sec:weakcGAlg} by defining weak $\cG$-algebras, the central notion of this paper. This is an associative ${\Fld}$-algebra $A$ together with a projective action of $G_1$ by algebra automorphisms and anti-automorphisms, according to the grading $\pi$, and a projective group homomorphism $G_2 \rightarrow A^{\times}$ which satisfy a number of coherence conditions. Compactly, it is an instance of Noohi's weak crossed module homomorphisms \cite{noohi2007}, special cases of which play a key role in \cite{rumynin2018}. A weak $\cG$-algebra in which the projective homomorphisms are, in fact, genuine homomorphisms is called strict. In Section~\ref{sec:moritaBicat} we construct various Morita bicategories which are $2$-equivariant for $\tcG$ and fit into the following diagram of subbicategories:
\[
\begin{tikzpicture}[baseline= (a).base]
\node[scale=1] (a) at (0,0){
\begin{tikzcd}[column sep=4.0em, row sep=1.5em]
\cG \mhyphen \AAlg_{\Fld} \arrow[hook]{r} & \cG \mhyphen RW\AAlg_{\Fld} \arrow[hook]{r} & \cG \mhyphen N\AAlg_{\Fld}\\
\cG \mhyphen \AAlg^{\fd}_{\Fld} \arrow[hook]{r} \arrow[hook]{u}& \cG \mhyphen RW\AAlg^{\fd}_{\Fld} \arrow[hook]{r} \arrow[hook]{u} & \cG \mhyphen N\AAlg^{\fd}_{\Fld} . \arrow[hook]{u}
\end{tikzcd}
};
\end{tikzpicture}
\]
The strictness of the $\cG$-algebras decreases from left to right and the superscript ``$\fd$'' indicated the fully dualizable subbicategory. A key technical notion, realizability of a twisted $2$-cocycle for the group $G_1$, is introduced in Section~\ref{sec:realizability}. The following {\em strictification} theorem is the main result of Chapter~\ref{sec:weakcGAlg}.
\begin{teorema}[{Theorem~\ref{thm:strictcGAlg} and Proposition~\ref{prop:strictification}}] Let $A\in\cG \mhyphen N\AAlg_{\Fld}$ be an N-weak $\cG$-algebra.
  \begin{enumerate}[label=(\roman*)]
  \item If $A$ is split semisimple, then there exists a strict (split semisimple) $\cG$-algebra $B\in\cG \mhyphen \AAlg^{\fd}_{\Fld}$ which is $\cG$-Morita equivalent to $A$.
  
      \item In general, there exist a $\ZTwo$-graded crossed module $\cH$ and a strict $\cH$-algebra $B\in\cH \mhyphen \AAlg_{\Fld}$ such that $\tcH \simeq \tcG$ as $\ZTwo$-graded $2$-groups and $B$ is $\cH$-Morita equivalent to $A$.
    \end{enumerate}
\end{teorema}

%
%

In Chapter~\ref{sec:induction} we define and study induction pseudofunctors. Given a crossed submodule $\cH$ of $\cG$, it is natural to expect the existence of an induction pseudofunctor
\[
\mbox{\gls{glInd}}
: \cH \mhyphen N\AAlg_{\Fld} \rightarrow \cG \mhyphen N\AAlg_{\Fld}.
\]
Using our strictification result, it suffices to construct $\Ind_{\cH}^{\cG}$ on the subbicategories $\cH \mhyphen \AAlg_{\Fld}$ and $\cG \mhyphen \AAlg_{\Fld}$ of strict algebras. There are three flavours of $\Ind_{\cH}^{\cG}$, depending on the $\ZTwo$-gradings of $\cH$ and $\cG$. We can now state the main result of Chapter ~\ref{sec:induction}.

\begin{teorema}[{Theorems \ref{thm:IndFunctor} and \ref{thm:2IndReal}}]
\leavevmode
  \begin{enumerate}[label=(\roman*)]
  \item There exists an induction pseudofunctor
    $\Ind_{\cH}^{\cG} : \cH \mhyphen \AAlg_{\Fld} \rightarrow \cG \mhyphen \AAlg_{\Fld}$.
  \item
    If the index $|G_2:H_2|$ is finite and the characteristic of $\Fld$ does not divide $|G_2:H_2|$, then the above pseudofunctor restricts to a pseudofunctor
    $\Ind_{\cH}^{\cG} : \cH \mhyphen \AAlg_{\Fld}^{\fd} \rightarrow \cG \mhyphen \AAlg_{\Fld}^{\fd}$.
    \end{enumerate}
\end{teorema}

Our construction of $\Ind_{\cH}^{\cG}$ is direct and explicit, illustrating the power of the Morita-theoretic approach to $2$-representation theory. An important technical result is a Maschke-type theorem for induced $\cG$-algebras, Proposition~\ref{thm:maschke}, asserting that $\Ind_{\cH}^{\cG}$ preserves separability. When $\cG$ is trivially graded, we prove that, under certain assumptions, $\Ind_{\cH}^{\cG}$ is both left and right biadjoint to the restriction pseudofunctor $\Res_{\cH}^{\cG}$ (Proposition ~\ref{thm:indBiadjunct}). The analogous question in the $\ZTwo$-graded setting is more subtle (Problem~\ref{prob:adjunction}).
In Section~\ref{sec:indRealKAlg} we describe the monoidal behaviour of $\Ind_{\cH}^{\cG}$.

In Chapter~\ref{sec:classReal2Reps} we return to Real $2$-representations on $2\Vect_{\Fld}$. We construct a local biequivalence $\cG \mhyphen \AAlg_{\Fld}^{\fd}\rightarrow \RRep_{2\Vect_{\Fld}}(\tcG)$ (Proposition ~\ref{thm:weakAlgTo2ModBicat}) that enables us to prove the first of the two main results of the paper.

\begin{teorema}[{Theorem~\ref{thm:moritaModel}}]
  If ${\Fld}$ is separably closed, then $\RRep_{2 \Vect_{\Fld}}(\tcG)$ is biequivalent to $\cG \mhyphen \AAlg^{\fd}_{\Fld}$.
\end{teorema}
  
As a consequence, we give a Morita-theoretic classification of Real $2$-representations of $\tcG$ (Corollary~\ref{thm:Real2RepIsoClass}). The resulting structure theorem for Real $2$-representations (Theorem~\ref{thm:ksDecomp}) yields a Real generalization of known results \cite{ostrik2003,elgueta2007,ganter2008}.

In Chapter~\ref{sec:GrothGrpRRep} we describe the Grothendieck ring of $\RRep_{2\Vect_{\Fld}}(\tcG)$, proving the second main result of the paper.

\begin{teorema}[{Theorem~\ref{thm:grothGroupDescr}}]
  The Grothendieck ring $K_0(\RRep_{2 \Vect_{\Fld}}(\tcG))$ is isomorphic to the generalized Burnside ring $\mathbb{B}^{\Phi}_{\mathbb{Z}}(\cG)$, where $\Phi$ is the functor of ``Real one dimensional $2$-representations''.
\end{teorema}
The isomorphism is explicit and compatible with the corresponding ordinary result \cite{rumynin2018}.

Finally, in Chapter~\ref{sec:RealChar} we turn to the character theory of Real $2$-representations. We define the categorical character and $2$-character of a Real $2$-representation of an arbitrary $\ZTwo$-graded $2$-group $\tcG$. Our approach is geometric, being formulated in terms of various kinds of loop spaces of $\cG$. One feature of our approach is that it works directly with the $2$-group (or its crossed module model) and applies uniformly to the ordinary and Real cases. Moreover, its form immediately suggests a generalization to the $n$-categorical and projective cases. In the finite case, we relate Real $2$-characters to mark homomorphisms of the generalized Burnside ring in Corollary~\ref{cor:Real2Char}. At its core, this is a result about $2$-characters of certain induced Real $2$-representations (cf. Theorem~\ref{thm:inducedCatChar} and Corollary~\ref{thm:induced2Char}). This provides a $2$-group generalization of the corresponding result for (Real projective) $2$-representations of finite groups \cite{ganter2008, mbyoung2018c}. We expect these results to be representation theoretic analogues of Hopkins--Kuhn--Ravenel character theory \cite{hopkins2000, lurie2019} of $2$-equivariant elliptic cohomology (see Problem~\ref{prob:HKR}).

\addtocontents{toc}{\protect\setcounter{tocdepth}{1}}

\subsection*{Acknowledgements}
Both authors are grateful to the Max Planck Institute for Mathematics where they met and started this project. The first author would like to thank the University of Zurich, whose hospitality he enjoyed.
The second author would like to thank Catharina Stroppel for discussions. Both authors would like to thank the referee for their careful reading and their many helpful comments and suggestions. The first author was partially supported by the Russian Academic Excellence Project `5--100'.



\section{Crossed modules and $2$-groups}
\label{sec:background}

\subsection{\texorpdfstring{$\ZTwo$}{}-graded crossed modules}

A crossed module $\cG = (G_2 \xrightarrow[]{\partial} G_1)$ consists of groups $G_1$ and $G_2$, an action $G_1 \times G_2 \rightarrow G_2$, $(g,x) \mapsto {^g}x$, of $G_1$ on $G_2$ by group automorphisms and a homomorphism $\partial: G_2 \rightarrow G_1$. This data is required to satisfy
\[
\partial ({^g}x) = g \partial(x) g^{-1},
\qquad \qquad
{^{\partial(y)}} x = y x y^{-1}
\]
for all $g,x,y$. Here, and for the rest of the paper, we use letters $f,g,h$ for elements of $G_1$, $x,y,z$ for elements of $G_2$ and $a,b,c$ for elements of some algebra $A$. 

A group $G$ defines a crossed module $( \{e\} \xrightarrow[]{e} G)$, which we continue to denote by $G$.

Denote by
\gls{glZTwo}
the multiplicative group $\{+1, -1\}$.

\begin{Def}
A $\ZTwo$-graded crossed module is a crossed module morphism $\pi: \cG \rightarrow \ZTwo$.
\end{Def}

Explicitly, $\pi$ is the data of a group homomorphism $\pi: G_1 \rightarrow \ZTwo$ which satisfies $\im (\partial) \leq \ker (\pi)$. 
The kernel of $\pi: \cG \rightarrow \ZTwo$ is the crossed submodule $\cG_0 = (G_2 \xrightarrow[]{\partial} G_0 )$, where $G_0 = \ker(\pi: G_1 \rightarrow \ZTwo)$. We call $\cG_0$ the ungraded crossed module of $\cG$. We say that $\cG$ is trivially graded if $\cG_0=\cG$.

\subsection{Crossed modules of generalized automorphisms}

Fix a ground field ${\Fld}$. Let $A$ be a ${\Fld}$-algebra, always assumed to be associative and unital. The group of units of $A$ is $A^{\times}$. The centre of $A$ is $Z(A)$. The assignment of an algebra to its opposite extends to an involution
$(-)^{\opp} :
\mbox{\gls{glALG}}
\rightarrow \mathsf{ALG}_{\Fld}$ of the category of ${\Fld}$-algebras and unital algebra morphisms. Given $\epsilon \in \ZTwo$, define ${^{\epsilon}}(-):\mathsf{ALG}_{\Fld} \rightarrow \mathsf{ALG}_{\Fld}$ so that ${^{+1}}(-) = \id_{\mathsf{ALG}_{\Fld}}$ and ${^{-1}}(-) = (-)^{\opp}$.

Let
\gls{glAutg}
be the set of all algebra isomorphisms of the form $A \rightarrow A$ or $A^{\opp} \rightarrow A$. Define $\pi: \Aut^{\gen}(A) \rightarrow \ZTwo$ so that $g \in \Aut^{\gen}(A)$ is an algebra homomorphism ${^{\pi(g)}}A \rightarrow A$, where we have used the functor ${^{\epsilon}}(-)$ defined in the previous paragraph. We consider $\Aut^{\gen}(A)$ as a group with multiplication
\[
g \cdot h = g \circ {^{\pi(g)}} h.
\]
Since $(-)^{\opp}$ acts trivially on morphisms (viewed as set maps), this is the usual composition of morphisms. The map $\pi$ makes $\Aut^{\gen}(A)$ into a $\ZTwo$-graded group with ungraded subgroup
\gls{glAut}.

\begin{Def}
The crossed module of generalized automorphisms of $A$ is
\[
\mbox{\gls{glAUTg}}
= \big( A^{\times} \xrightarrow[]{\partial} \Aut^{\gen}(A) \big),
\]
where $\Aut^{\gen}(A)$ acts on $A^{\times}$ by ${^{g}}x = g(x^{\pi(g)})$ and $\partial(x)$ is the inner automorphism $a \mapsto x a x^{-1}$.
\end{Def}

To see that, for example, the first axiom of a crossed module holds, note that $\partial({^{g}} x)(a) = g(x^{\pi(g)}) a g(x^{-\pi(g)})$ while
\[
(g \partial(x) g^{-1})(a)
=
g(x^{\pi(g)} g^{-1}(a)x^{-\pi(g)})
=
g(x^{\pi(g)}) a g(x^{-\pi(g)}),
\]
as required. The $\ZTwo$-grading of $\Aut^{\gen}(A)$ induces a $\ZTwo$-grading of $\AUT^{\gen}(A)$, the ungraded crossed module of which is
$\mbox{\gls{glAUT}} \coloneqq (A^{\times} \xrightarrow[]{\partial} \Aut(A))$.

\subsection{\texorpdfstring{$2$}{}-groups}
\label{sec:2Grps}

For categorical background, we refer the reader to B\'{e}nabou \cite{benabou1967}.
For a detailed introduction to $2$-groups, see Baez and Lauda \cite{baez2004}.

In this paper, we use the term $2$-group for what is called a weak $2$-group in \cite{baez2004}, namely, a bicategory
\gls{gltG}
$\tG$
with a single object in which all $1$-morphisms are equivalences and all $2$-morphisms are isomorphisms. Morphisms of $2$-groups are pseudofunctors. Note that $\tG$ can be seen as a monoidal groupoid in which each endofunctor $g \otimes - : \tG \rightarrow \tG$, $g \in \tG$, is an equivalence. We switch freely between these two perspectives. A $2$-group is called strict if its underlying bicategory is a (strict) $2$-category. Every $2$-group is equivalent to a strict $2$-group. A group can be thought of as a groupoid and, hence, as a $2$-group in a canonical way.

There is a well-known equivalence $\widetilde{(-)}$ from the category of crossed modules and strict crossed module morphisms to the category of strict, small $2$-groups and $2$-functors \cite{brown1976}, \cite[\S 3.3]{noohi2007}. This functor assigns to a crossed module $\cG$ the $2$-group $\tcG$ with object $\star$, with $1$-morphisms $1\End_{\tcG}(\star) = G_1$ and $2$-morphisms
\[
2\Hom_{{\tcG}}(g_1,g_2) = \{ x \in G_2 \mid \partial(x) g_1 = g_2\}.
\]
The vertical composition law in $\tcG$ is illustrated by the diagram
\[
(g_2 \xRightarrow[]{x_2} g_3) \circ (g_1 \xRightarrow[]{x_1} g_2)
=
\begin{tikzpicture}[baseline= (a).base]
\node[scale=0.9] (a) at (0,0){
\begin{tikzcd}[column sep=5em]
\star
  \arrow[bend left=65]{r}[name=U,below]{}{g_3} 
  \arrow[]{r}[name=M]{}[above left]{g_2}
  \arrow[bend right=65]{r}[name=D,]{}[swap]{g_1}
& 
\star
     \arrow[Rightarrow,to path={(M) -- node[label=right:\footnotesize \scriptsize $x_2$] {} (U)}]{}
     \arrow[shorten <=3pt,shorten >=3pt,Rightarrow,to path={(D) -- node[label=right:\footnotesize \scriptsize $x_1$] {} (M)}]{}
\end{tikzcd}
};
\end{tikzpicture}
=
\begin{tikzpicture}[baseline= (a).base]
\node[scale=0.9] (a) at (0,0){
\begin{tikzcd}[column sep=5em]
\star
  \arrow[bend left=35]{r}[name=U,below]{}{g_3} 
  \arrow[bend right=35]{r}[name=D]{}[swap]{g_1}
& 
\star.
     \arrow[Rightarrow,to path={(D) -- node[label=right:\footnotesize \scriptsize $x_2 x_1$] {} (U)}]{}
\end{tikzcd}
};
\end{tikzpicture}
\]
The horizontal composition law is illustrated by 
\[
(g_1 \xRightarrow[]{x} g_2) \diamond (g_1^{\prime} \xRightarrow[]{x^{\prime}} g_2^{\prime})
=
\begin{tikzpicture}[baseline= (a).base]
\node[scale=0.9] (a) at (0,0){
\begin{tikzcd}[column sep=5em]
\star
  \arrow[bend left=25]{r}[name=U,below]{}{g_2} 
  \arrow[bend right=25]{r}[name=D]{}[swap]{g_1}
& 
\star
  \arrow[bend left=25]{r}[name=U2,below]{}{g_2^{\prime}} 
   \arrow[bend right=25]{r}[name=D2]{}[swap]{g_1^{\prime}}
&
\star
     \arrow[Rightarrow,to path={(D) -- node[label=right:\footnotesize \scriptsize $x$] {} (U)}]{}
      \arrow[Rightarrow,to path={(D2) -- node[label=right:\footnotesize \scriptsize $x^{\prime}$] {} (U2)}]{}
\end{tikzcd}
};
\end{tikzpicture}
=
\begin{tikzpicture}[baseline= (a).base]
\node[scale=0.9] (a) at (0,0){
\begin{tikzcd}[column sep=6em]
\star
  \arrow[bend left=25]{r}[name=U,below]{}{g_2 g_2^{\prime}} 
  \arrow[bend right=25]{r}[name=D]{}[swap]{g_1 g_1^{\prime}}
& 
\star.
     \arrow[Rightarrow,to path={(D) -- node[label=right:\footnotesize \scriptsize $x \cdot {^{g_1}}x^{\prime}$] {} (U)}]{}
\end{tikzcd}
};
\end{tikzpicture}
\]

The equivalence class of the $2$-group $\tcG$ is determined by the quadruple $(\pi_1 (\cG), \pi_2 (\cG), \alpha, [\theta])$ where $\alpha$ is an action of $\pi_1 (\cG)\coloneqq \coker (\partial)$ on $\pi_2 (\cG)\coloneqq \ker (\partial)$ and $[\theta] \in H^3(\pi_1 (\cG), \pi_2 (\cG))$ is the Sinh cohomology class \cite[Theorem 8.3.7]{baez2004} (cf. \cite{elgueta2007,rumynin2019}).

A $\ZTwo$-graded $2$-group is a $2$-group morphism $\pi: \tG \rightarrow \ZTwo$. The ungraded $2$-group of $\tG$ is the locally full subbicategory $\tG_0$ on $1$-morphisms in $\ker (\pi)$. We also write $\widetilde{(-)}$ for the induced equivalence between the categories of $\ZTwo$-graded crossed modules and $\ZTwo$-graded groups.

\subsection{Generalized automorphism \texorpdfstring{$2$}{}-groups}
\label{sec:genAut2Grp}

Given a bicategory $\mathcal{V}$, denote by $\mathcal{V}^{\co}$ the bicategory obtained from $\mathcal{V}$ by reversing its $2$-cells.

We recall a construction of \cite[\S 3.3]{mbyoung2018c}. Let $\mathcal{V}$ be a bicategory with weak duality involution, in the sense of Shulman \cite[\S 2]{shulman2018}. This is the data of a pseudofunctor
$\mbox{\gls{glCi}}
: \mathcal{V}^{\co} \rightarrow \mathcal{V}$, a pseudonatural adjoint equivalence $\mu : 1_{\mathcal{V}} \Rightarrow  (-)^{\circ} \circ ((-)^{\circ})^{\co}$ and additional higher coherence data which we do not recall here. Let $V \in \mathcal{V}$. The collection of all equivalences of the form $V \rightarrow V$ or $V^{\circ} \rightarrow V$, together with the $2$-isomorphisms between them, assembles to a $\ZTwo$-graded $2$-group $1\Aut^{\gen}_{\mathcal{V}}(V)$, called the generalized automorphism $2$-group of $V$. The monoidal structure $\otimes$ is defined on objects by
\[
f_2 \otimes f_1 = f_2 \circ ({^{\pi(f_2)}}f_1 \circ \mu_V^{\delta_{\pi(f_2), \pi(f_1), -1}}),
\]
where $\pi(f) \in \ZTwo$ is such that $f: {^{\pi(f)}} V \rightarrow V$, the symbol ${^{\pi(f)}}(-)$ determines the application of $(-)^{\circ}$ and
$$
\delta_{\pi(f_2), \pi(f_1), -1} = 
\begin{cases}
  +1 & \mbox{ if } \pi(f_2) = \pi(f_1)=-1, \\
  \; 0 & \mbox{ otherwise,}
  \end{cases} 
$$
while $\mu_V^0$ means that the map $\mu_V$ is omitted. The definition of $\otimes$ on morphisms is similar. We omit the definition of the associator, which uses the higher coherence data. The ungraded $2$-group of $1\Aut^{\gen}_{\mathcal{V}}(V)$ is $1\Aut_{\mathcal{V}}(V)$, the weak automorphism $2$-group of $V$, as defined in \cite[\S 8]{baez2004}.

\subsection{Real $2$-representation theory}
\label{sec:Real2Rep}

Following \cite[\S 3]{mbyoung2018c}, we recall the basic definitions of the Real $2$-representation theory of $2$-groups.

The bicategory of $2$-representations of a $2$-group $\tG$ on a bicategory $\mathcal{V}$ is
\[
\mbox{\gls{glRep}}
= 1\Hom_{\mathsf{Bicat}}(\tG, \mathcal{V}),
\]
consisting of pseudofunctors, pseudonatural transformations and modifications. A $2$-representation of $\tG$ is, thus, the datum of an object $V \in \mathcal{V}$ and a $2$-group morphism $\tG \rightarrow 1 \Aut_{\mathcal{V}}(V)$. For detailed studies of $2$-representation theory, the reader is referred to \cite{barrett2006, elgueta2007, bartlett2011}.

Now let $\tG$ be a $\ZTwo$-graded $2$-group. Let $\mathcal{V}$ be a bicategory with weak duality involution. The bicategory of Real $2$-representations of $\tG$ on $\mathcal{V}$ is
\[
\mbox{\gls{glRRep}}
= 1\Hom_{\mathsf{Bicat}_{\mathsf{con}}}(\tG, \mathcal{V}).
\]
Here $\mathsf{Bicat}_{\mathsf{con}}$ is the tricategory of bicategories with contravariance. We regard $\tG$ and $\mathcal{V}$ as bicategories with contravariance \cite[\S \S 1.2, 3.3]{mbyoung2018c}.
The ingredients of $\RRep_{\mathcal{V}}(\tG)$ are as introduced in \cite[\S 4]{shulman2018}:
\begin{itemize}
\item objects -- contravariance preserving pseudofunctors,
\item $1$-morphisms -- pseudonatural transformations respecting contravariance, 
\item  $2$-morphisms -- modifications respecting contravariance.
\end{itemize}
In particular, all $2$-morphisms are isomorphisms. In concrete terms, a Real $2$-representation of $\tG$ on $V \in \mathcal{V}$ is a morphism $\tG \rightarrow 1\Aut^{\textnormal{gen}}_{\mathcal{V}}(V)$ of $\ZTwo$-graded $2$-groups.

A symmetric monoidal structure on $\mathcal{V}$ (which commutes with the weak duality involution) induces symmetric monoidal structures on $\Rep_{\mathcal{V}}(\tG)$ and $\RRep_{\mathcal{V}}(\tG)$. These monoidal structures are compatible in the sense that the restriction pseudofunctor
\[
\Res_{\tG_0}^{\tG} : \RRep_{\mathcal{V}}(\tG) \rightarrow \Rep_{\mathcal{V}}(\tG_0) 
\]
is monoidal.

\begin{Ex}
Let $\mathsf{Cat}$ be the $2$-category of small categories. The assignment of a category to its opposite extends to a strict duality involution $(-)^{\opp} : \mathsf{Cat}^{\co} \rightarrow \mathsf{Cat}$. A Real $2$-representation of a $\ZTwo$-graded group $G$ is the data of a category $\mathcal{C}$, equivalences
\[
\rho(g): {^{\pi(g)}}\mathcal{C} \rightarrow \mathcal{C}, \qquad g \in G
\]
and composition natural isomorphisms
\[
\rho_{g_2, g_1} : \rho(g_2) \circ {^{\pi(g_2)}}\rho(g_1) \Longrightarrow \rho(g_2 g_1), \qquad g_i \in G.
\]
This data is required to satisfy the associativity constraints
\[
\rho_{g_3 g_2, g_1} \diamond  \left( \rho_{g_3, g_2} \circ \prescript{\pi(g_3 g_2)}{}{\rho(g_1)} \right) = \rho_{g_3, g_2 g_1} \diamond \left( \rho(g_3) \circ \prescript{\pi(g_3)}{}{\rho}_{g_2, g_1}^{\pi(g_3)} \right),
\qquad g_i \in G.
\]
\end{Ex}

\begin{Ex}
Let $2 \Vect_{\Fld}$ be the bicategory of finite dimensional Kapranov--Voevodsky $2$-vector spaces over $\Fld$. We use the semi-skeletal model \cite[\S 2.2]{ganter2008}, \cite[\S 1]{rumynin2018}. Objects are $[n]$, $n \in \mathbb{Z}_{\geq 0}$. A $1$-morphism $[n] \rightarrow [m]$ is an $m \times n$ matrix $A=(A_{ij})$ of finite dimensional vector spaces over ${\Fld}$. A $2$-morphism $u: A \Rightarrow B$ is a collection of ${\Fld}$-linear maps $(u_{ij}: A_{ij} \rightarrow B_{ij})$. We omit the definition of the various compositions. The bicategory $2 \Vect_{\Fld}$ has a natural symmetric monoidal structure. A weak duality involution on $2 \Vect_{\Fld}$ is defined as follows. The pseudofunctor $(-)^{\circ} : 2 \Vect_{\Fld}^{\co} \rightarrow 2 \Vect_{\Fld}$ is given on objects, $1$-morphisms and $2$-morphisms by
\[
[n]^{\circ} = [n], \qquad (A_{ij})^{\circ} = (A_{ij}^{\vee}), \qquad
(u_{ij})^{\circ} = (u_{ij}^{\vee}),
\]
respectively, where $(-)^{\vee}$ is the ${\Fld}$-linear duality functor on the category $\Vect_{\Fld}$ of finite dimensional vector spaces. The remaining data for the duality involution is induced from the evaluation isomorphism $1_{\Vect_{\Fld}} \simeq (-)^{\vee} \circ ((-)^{\vee})^{\opp}$. Equivalence classes of Real $2$-representations of finite groups on $2 \Vect_{\Fld}$ are classified in \cite[\S 5.3]{mbyoung2018c}.
\end{Ex}

\subsection{Real \texorpdfstring{$2$}{}-modules}
\label{sec:real2Mod}

Let $2 \MMod_{\Fld}$ be the $2$-category of $2$-modules over ${\Fld}$, as defined in \cite[\S 1]{rumynin2018} (where it is denoted $2 \mhyphen \mathcal{M}od^{\Fld}$). Objects are modules categories over the monoidal category $\Vect_{\Fld}$ (see \cite[\S 7.1]{etingof2015}) that are $\Vect_{\Fld}$-module equivalent to $A \Mod$ for some ${\Fld}$-algebra $A$. We do not recall the definitions of $1$- and $2$-morphisms. A $2$-representation of a $2$-group $\tG$ on $2 \MMod_{\Fld}$ is called a $2$-module over $\tG$.

In this paper we use a variation of this set-up. Let
\gls{glAALG}
be the Morita bicategory:
\begin{itemize}
\item objects -- ${\Fld}$-algebras,
\item $1$-morphisms $A \rightarrow B$ -- $B \mhyphen A$-bimodules,
\item $2$-morphisms -- bimodule intertwiners.
\end{itemize}
The composition of $1$-morphisms is the tensor product of bimodules. Let
\gls{glAAlg}
be the subbicategory of finite dimensional algebras, finite dimensional bimodules and intertwiners.
Tensor product of algebras over ${\Fld}$ induces symmetric monoidal structures on $\AALG_{\Fld}$ and $\AAlg_{\Fld}$. 
Let
\gls{glAAlgfd}
be the fully dualizable subbicategory of $\AALG_{\Fld}$. Every algebra $A$ admits a dual: it is the opposite algebra $A^{\opp}$; the evaluation 1-morphism and the coevaluation 1-morphism are both given by the natural bimodule ${}_AA_A$. For $A$ to be fully dualizable the bimodule ${}_AA_A$ needs to be projective. This splits the multiplication map $A\otimes_{\Fld} A \rightarrow A$ as a map of bimodules, proving that $A$ is separable. The opposite statement is clear:
if $A$ is separable, then the bimodule ${}_AA_A$ is finitely generated projective
and $A$ is fully dualizable. Thus, $\AAlg_{\Fld}^{\fd}$ is the full subbicategory of $\AAlg_{\Fld}$ spanned by separable algebras, cf. \cite[Proposition 7.10]{lorand2019}. Tensor product of algebras over ${\Fld}$ induces a symmetric monoidal structure on $\AAlg^{\fd}_{\Fld}$ as well.

The $2$-representation theories of $\tG$ on $2\MMod_{\Fld}$ and $\AALG_{\Fld}$ are equivalent. This follows from the fact that, by Morita theory, equivalences in $2\MMod_{\Fld}$ can be represented by bimodules. The $2$-representation theories of $\tG$ on $\AAlg_{\Fld}$ and $\AAlg^{\fd}_{\Fld}$ can therefore be thought of as the finite dimensional and separable $2$-module theories of $\tG$, respectively.

Recall that each ${\Fld}$-algebra morphism $\phi:A \rightarrow B$ defines restriction functors from $B$-modules to $A$-modules. If $M$ is a right $B$-module, then $M_{\phi}$ is a right $A$-module, equal to $M$ as an abelian group and with right $A$-module structure
\[
m \cdot a = m \phi(a), \qquad m \in M_{\phi}, \, a \in A.
\]
Similarly, a left $B$-module $N$ determines a left $A$-module ${}_{\phi}N$. We refer to $M_{\phi}$ and ${}_{\phi}N$ as the right and left twists of the bimodules $M$ and $N$ by $\phi$, respectively. Starting with the identity bimodule ${}_{B}B_{B}$, we get a {\em representable} $B \mhyphen A$-bimodule $B_{\phi}$.
We use the representable bimodules to embed $\mathsf{ALG}_{\Fld}$ as a locally discrete subbicategory of $\AALG_{\Fld}$. Also relevant is the locally full subbicategory $\AALG_{\Fld}^{\rep}$ of $\AALG_{\Fld}$ on representable $1$-morphisms. We record the following result for later use.

\begin{Lem}
\label{lem:natTransBij}
Let $\phi$ and $\psi$ be ${\Fld}$-algebra isomorphisms $A \rightarrow B$ and $\partial: B^{\times} \rightarrow \Aut(B)$ the structure map of $\AUT(B)$. Then the map
\[
\Upsilon_{\phi}^{\psi} : \{b \in B^{\times} \mid \phi =  \partial(b) \psi \} \rightarrow 2 \Hom_{\AALG_{\Fld}}(B_{\phi}, B_{\psi})
\]
which sends an element $b \in B^{\times}$ to the map given by right multiplication by $b$ is a bijection onto the subset of $2$-isomorphisms.
\end{Lem}

\begin{Prop}[{\cite[Theorem 5.1]{lorand2019}}]
\label{prop:AlgWDI}
The bicategory $\AAlg_{\Fld}^{\fd}$ admits a weak duality involution in which the pseudofunctor $(-)^{\circ} : (\AAlg_{\Fld}^{\fd})^{\co}  \rightarrow \AAlg_{\Fld}^{\fd}$ is defined on objects by $A^{\circ} = A^{\opp}$, on a $1$-morphism $M : A \rightarrow B$ by $M^{\circ} = \Hom_{\textnormal{Mod}\mhyphen A}(M,A)$ (viewing the $A \mhyphen B$-bimodule on the right hand side as a $B^{\opp} \mhyphen A^{\opp}$-bimodule), and on a $2$-morphism $\phi: M \Rightarrow N$ by $\phi^{\circ} = (-) \circ \phi$.
\end{Prop}

\begin{proof}
The additional coherence data for the pseudofunctor $(-)^{\circ}$ and its lift to a weak duality involution are constructed using the separability of algebras and the finite dimensionality of bimodules involved. We refer to reader to \cite{lorand2019} for details.
\end{proof}

\begin{Def}
A Real $2$-module over a $\ZTwo$-graded $2$-group $\tG$ is a Real $2$-representation of $\tG$ on $\AAlg_{\Fld}^{\fd}$.
\end{Def}

There is a canonical pseudofunctor $2 \Vect_{\Fld} \rightarrow \AAlg^{\fd}_{\Fld}$ which sends the object $[n]$ to the algebra ${\Fld}^n$. If ${\Fld}$ is separably closed, then this is a monoidal biequivalence which, with $2 \Vect_{\Fld}$ equipped with the weak duality involution of Section \ref{sec:Real2Rep}, lifts to a duality biequivalence. An explicit construction of this lifting is given in \cite[Theorem 6.3]{lorand2019}. In particular, this allows us to conclude that the bicategories of Real $2$-representations of $\tG$ on $2\Vect_{\Fld}$ and $\AAlg_{\Fld}^{\fd}$ are monoidally biequivalent.

\section{Actions of crossed modules on algebras}
\label{sec:weakcGAlg}

\subsection{Weak \texorpdfstring{$\cG$}{}-algebras}
\label{sec:weakAlg}

Let $\cH$ be a crossed module. A weak $\cH$-algebra is a ${\Fld}$-algebra $A$ together with a weak morphism of crossed modules $\omega_A: \cH \rightarrow \AUT(A)$. There are (at least) two different notions of a weak morphism of crossed modules in the literature, namely those introduced by Noohi\footnote{An unrelated definition appears in \cite[Definition 8.4]{noohi2007}, which was later revised in \cite[Definition 8.4]{noohi_arxiv}. The revised definition also seems to contain a typo; compare with equation \eqref{eq:omega1Proj}.} \cite[Definition 8.4]{noohi_arxiv} (N-weak for short) and Rumynin--Wendland \cite[\S 2]{rumynin2018} (RW-weak). The latter is a particular case of the former.

Let now $\cG$ be a $\ZTwo$-graded crossed module. Influenced by the ordinary case, we introduce the following definition.

\begin{Def}
An  N-weak $\cG$-algebra is a ${\Fld}$-algebra $A$ together with a weak morphism of $\ZTwo$-graded crossed modules $\omega_A: \cG \rightarrow \AUT^{\gen}(A)$. Explicitly, $\omega_A$ is the data of
\begin{enumerate}[label=(\roman*)]
\item a function $\omega_3: G_1 \times G_1 \rightarrow A^{\times}$ that restricts to the identity on $G_1 \times \{e_{G_1}\}$ and $\{e_{G_1}\} \times G_1$,

\item a unital $\partial_*\omega_3$-projective $\ZTwo$-graded group homomorphism $\omega_1: G_1 \rightarrow \Aut^{\gen}(A)$, that is, a pointed map over $\ZTwo$ which satisfies
\begin{equation}
\label{eq:omega1Proj}
\omega_1(g_2 g_1) = \partial(\omega_3(g_2,g_1)) \omega_1(g_2) \omega_1(g_1), \qquad g_i \in G_1
\end{equation}
and

\item a unital $\partial^* \omega_3$-projective group homomorphism $\omega_2: G_2 \rightarrow A^{\times}$, that is, a pointed map which satisfies
\begin{equation}
\label{eq:omega2Proj}
\omega_2(x_2 x_1) = \omega_3(\partial x_2, \partial x_1) \omega_2(x_2) \omega_2(x_1), \qquad x_i \in G_2.
\end{equation}
\end{enumerate}
This data is required to satisfy
\begin{equation}
\label{eq:NAlgAx1}
\omega_1 \circ \partial = \partial \circ \omega_2,
\end{equation}
the non-abelian $2$-cocycle condition
\begin{equation}
\label{eq:omega3Cocycle}
\omega_3(g_3g_2,g_1) \omega_3(g_3,g_2) = \omega_3(g_3,g_2 g_1) \cdot {^{\omega_1(g_3)}} \omega_3(g_2,g_1), \qquad g_i \in G_1
\end{equation}
and the equivariance condition
\begin{equation}
\label{eq:NAlgAx2}
\omega_2({^{g}}x)
=
\omega_3(g \partial x, g^{-1}) \omega_3(g,\partial x) \cdot {^{\omega_1(g)}}\omega_2(x) \omega_3(g,g^{-1})^{-1} , \qquad g \in G_1, \, x \in G_2.
\end{equation}
\end{Def}


We write $Z^2(G_1, A_{\pi}^{\times})$ for the set of normalized functions $\omega_3$ which satisfy equation~\eqref{eq:omega3Cocycle}. The subscript $\pi$ indicates that the action of $G_1 \setminus G_0$ on $A^{\times}$ involves inversion. 
%
%

We utilize the following (partial) strictifications of the previous definition.

\begin{Def}
\begin{enumerate}[label=(\roman*)]
\item An RW-weak $\cG$-algebra is an N-weak $\cG$-algebra $A$ in which $\omega_3$ factors through $Z(A)^{\times} \leq A^{\times}$.

\item A (strict) $\cG$-algebra is an N-weak $\cG$-algebra with trivial $\omega_3$.
\end{enumerate}
\end{Def}

Said differently, an RW-weak $\cG$-algebra is an N-weak $\cG$-algebra in which $\omega_1$ is a $\ZTwo$-graded group homomorphism.

The ungraded and graded notions of weak algebras are compatible in the sense that the restriction of an N-weak (resp. RW-weak, strict) $\cG$-algebra $\omega_A$ to $\cG_0$ is an N-weak (resp. RW-weak, strict) $\cG_0$-algebra $\omega_A: \cG_0 \rightarrow \AUT(A)$.

\begin{Def}
A strict morphism of N-weak $\cG$-algebras $\phi: A \rightarrow B$ is a unital algebra morphism which makes the diagrams
\[
\begin{tikzpicture}[baseline= (a).base]
\node[scale=1] (a) at (0,0){
\begin{tikzcd}[column sep=5.0em, row sep=2.0em]
{^{\pi(g)}}A \arrow{d}[left]{\omega_{A,1}(g)} \arrow{r}[above]{{^{\pi(g)}}\phi} & {^{\pi(g)}}B \arrow{d}[right]{\omega_{B,1}(g)}\\
A \arrow{r}[below]{\phi}& B
\end{tikzcd}
};
\end{tikzpicture}
\qquad
,
\qquad
\begin{tikzpicture}[baseline= (a).base]
\node[scale=1] (a) at (0,0){
\begin{tikzcd}[column sep=2.0em, row sep=2.0em]
& G_2 \arrow{dl}[above left]{\omega_{A,2}} \arrow{dr}[above right]{\omega_{B,2}} & \\
A^{\times} \arrow{rr}[below]{\phi_{\vert A^{\times}}} && B^{\times}
\end{tikzcd}
};
\end{tikzpicture}
\]
commute for all $g \in G_1$ and which satisfies $\phi_* \omega_{A,3} = \omega_{B,3}$.
\end{Def}

Let $\cH$ and $\cK$ be crossed modules. 
Noohi's weak crossed module morphisms are defined so that there is a biequivalence
\[
\Hom_{\mathsf{CM}}(\cH, \cK) \simeq \Hom_{\mathsf{Bicat}}(\tcH, \tcK),
\]
where the left-hand side is the bicategory of weak crossed module morphisms, transformations and modifications.
Strictly speaking, the above biequivalence is not proved in \cite{noohi2007} and so we do not use it in the remainder of the paper; see, however \cite[Proposition 8.1]{noohi2007} and Proposition \ref{prop:weakAlgTo2Mod} below. Under the above biequivalence, strict crossed module morphisms correspond to strict $2$-functors. This, together with the following lemma, explains the categorical meaning of weak and strict $\cH$-algebras.

\begin{Lem}
\label{lem:eq:2AutCategorical}
For any ${\Fld}$-algebra $A$, there is a biequivalence $\widetilde{\AUT(A)} \simeq 1\Aut_{\AALG^{\rep}_{\Fld}}(A)$.
\end{Lem}

\begin{proof}
This can be proved in the same way as \cite[Proposition 2.2]{rumynin2018}.
\end{proof}

Similarly, if $A$ is separable, then one can show that $\AUT^{\gen}(A)$ models the $\ZTwo$-graded $2$-group $1 \Aut^{\gen}_{\AAlg_{\Fld}^{\fd,\rep}}(A)$. We therefore obtain an analogous categorical interpretation of (weak) $\cG$-algebras.

%

\subsection{Equivariant objects}
\label{sec:homoFixed}

We introduce the notion of an equivariant object of a Real $2$-representation on $\mathsf{Cat}$. This clarifies some of the constructions which follow.

Let $\rho$ be a Real $2$-representation of a $\ZTwo$-graded group $G$ on a category $\mathcal{C}$. An equivariant object of $\rho$ is a pair $(t, \alpha)$ consisting of an object $t \in \mathcal{C}$ and isomorphisms $\alpha_g : \rho(g)(t) \rightarrow t$, $g \in G$, which make the diagrams\footnote{For notational simplicity, we omit the symbol $(-)^{\opp}$ in this diagram.}
\[
\begin{tikzpicture}[baseline= (a).base]
\node[scale=1] (a) at (0,0){
\begin{tikzcd}[column sep=5.0em, row sep=2.0em]
\rho(g_2 g_1)(t) \arrow{d}[left]{\alpha_{g_2 g_1}} & \rho(g_2)(\rho(g_1)(t)) \arrow{l}[above]{\rho_{g_2,g_1,t}} \arrow{d}[right]{\rho(g_2)(\alpha^{\pi(g_2)}_{g_1})}\\
t & \arrow{l}[below]{\alpha_{g_2}} \rho(g_2)(t)
\end{tikzcd}
};
\end{tikzpicture}
\;
,
\qquad g_i \in G
\]
commute. A morphism of equivariant objects $\phi: (t,\alpha) \rightarrow (s, \beta)$ is an isomorphism $\phi: t \rightarrow s$ which makes the diagrams
\[
\begin{tikzpicture}[baseline= (a).base]
\node[scale=1] (a) at (0,0){
\begin{tikzcd}[column sep=5.0em, row sep=2.0em]
\rho(g)(t) \arrow{d}[left]{\alpha_g} \arrow{r}[above]{\rho(g)(\phi^{\pi(g)})} & \rho(g)(s) \arrow{d}[right]{\beta_g}\\
t \arrow{r}[below]{\phi}& s
\end{tikzcd}
};
\end{tikzpicture}
\;
,
\qquad g \in G
\]
commute. This defines a groupoid of equivariant objects of $\rho$.

When $G$ is trivially graded, there are two possible definitions. We need not require the map $\phi$ to be an isomorphism. In this case the definition reduces to the standard category of equivariant objects of $\rho$.

\subsection{Morita bicategories of weak \texorpdfstring{$\cG$}{}-algebras}
\label{sec:moritaBicat}

In this section we introduce a more flexible notion of a morphism of N-weak $\cG$-algebras. We begin with some preliminary material. We use left and right twists of bimodules by algebra homomorphisms; see Section \ref{sec:real2Mod}. Note that the pseudofunctor $(-)^{\circ}$ is defined on the bicategory $\AALG_{\Fld}$, although it only lifts to a weak duality involution on $\AAlg_{\Fld}^{\fd}$.

\begin{Lem}
\label{lem:dualBimoduleTwists}
Let $A, B \in \AALG_{\Fld}$. For each $B \mhyphen A$-bimodule $M$ and $\psi \in \Aut(A)$, $\phi \in \Aut(B)$, there is a $B^{\opp} \mhyphen A^{\opp}$-bimodule isomorphism
\[
({_{\phi}}M_{\psi})^{\circ}
\xrightarrow[]{\sim}
{_{\phi^{\opp}}}(M^{\circ})_{\psi^{\opp}},
\qquad
{\mathfrak{m}} \mapsto \big( m \mapsto \psi({\mathfrak{m}}(m)) \big) .
\]
\end{Lem}

Let $\cG$ be a $\ZTwo$-graded crossed module. Let $A$ and $B$ be separable N-weak $\cG$-algebras. The category $1\Hom_{\AAlg_{\Fld}^{\fd}}(A,B)$ of $1$-morphisms of the underlying ${\Fld}$-algebras inherits the structure of a Real $2$-representation $\lambda$ of $G_1$. Explicitly, an element $g \in G_1$ acts by the functor $\lambda(g)$ which sends a $B \mhyphen A$-bimodule $M$ to\footnote{For readability (and unlike Lemma \ref{lem:dualBimoduleTwists}), we henceforth omit the notation $(-)^{\opp}$ on morphisms.}
\[
\lambda(g)(M) = {_{\omega_{B,1}(g)^{-1}}}\left({^{\pi(g)}}M \right)_{\omega_{A,1}(g)^{-1}}.
\]
Recall that ${^{\pi(g)}}(-)$ denotes the identify pseudofunctor if $\pi(g)=+1$ and the pseudofunctor $(-)^{\circ}$ of Proposition \ref{prop:AlgWDI} if $\pi(g)=-1$. On morphisms $\lambda(g)$ acts as ${^{\pi(g)}}(-)$; the (anti-)automorphism twists $\omega_{?,1}(g)^{-1}$ act trivially. In particular, $\lambda(g)$ is contravariant precisely when $\pi(g)=-1$. The component at $M$ of the coherence natural transformation
\[
\lambda_{g_2,g_1}: \lambda(g_2) \circ {^{\pi(g_2)}}\lambda(g_1) \Longrightarrow \lambda(g_2 g_1), \qquad g_i \in G, 
\]
when viewed as a ${\Fld}$-linear map $\lambda_{g_2,g_1,M} : {^{\pi(g_2 g_1)}} M \rightarrow M$, is given by left multiplication by ${^{\omega_{B,1}(g_2 g_1)^{-1}}}\omega_{B,3}(g_2, g_1)^{-1}$ and right multiplication by ${^{\omega_{A,1}(g_2 g_1)^{-1}}}\omega_{A,3}(g_2, g_1)$, which we write as $\partial ({^{\omega_{?,1}(g_2 g_1)^{-1}}}\omega_{?,3}(g_2, g_1)^{-1})$. Implicit in this description of $\lambda_{g_2,g_1,M}$ is, when $\pi(g_2) =-1$, the use of Lemma \ref{lem:dualBimoduleTwists} and, when $\pi(g_1)=\pi(g_2) =-1$, the use of the evaluation isomorphism $\ev_M: M \rightarrow M^{\circ \circ}$.

Let us unpack the datum of a $G_1$-equivariant object $(M, \omega_M)$ of  $1\Hom_{\AAlg_{\Fld}^{\fd}}(A,B)$. First, we have a $B \mhyphen A$-bimodule $M$. Second, for each $g \in G_1$, we have a ${\Fld}$-linear isomorphism $\omega_M(g) : {^{\pi(g)}}M \rightarrow M$ which satisfies
\begin{equation}
\label{eq:Real1MorAx1}
\omega_M(g)(bma) = \omega_{B,1}(g)(b) \omega_M(g)(m) \omega_{A,1}(g)(a), \qquad
a \in {^{\pi(g)}}A, \, m \in {^{\pi(g)}}M, \, b \in {^{\pi(g)}}B.
\end{equation}
Moreover, these isomorphisms are required to satisfy
\begin{equation}
\label{eq:Real1MorAx2}
\omega_M(g_2 g_1) = \partial(\omega_{?,3}(g_2,g_1)) \circ \omega_{M}(g_2) \circ {^{\pi(g_2)}}\omega_M(g_1)^{\pi(g_2)},
\qquad
g_i \in G_1.
\end{equation}
A $1$-morphism $\phi: (M, \omega_M) \rightarrow (N, \omega_N)$ is a $B \mhyphen A$-bimodule isomorphism which is $G_1$-equivariant, in the sense that the diagrams
\[
\begin{tikzpicture}[baseline= (a).base]
\node[scale=1] (a) at (0,0){
\begin{tikzcd}[column sep=5.0em, row sep=2.0em]
\lambda(g)(M) \arrow{d}[left]{\omega_M(g)} \arrow{r}[above]{{^{\pi(g)}}\phi^{\pi(g)}} & \lambda(g)(N) \arrow{d}[right]{\omega_N(g)}\\
M \arrow{r}[below]{\phi}& N
\end{tikzcd}
};
\end{tikzpicture}
\;
,
\qquad g \in G_1
\]
commute.

Using the above notation, we define a bicategory
\gls{glGAAlgfd}:
\begin{itemize}
\item objects -- separable N-weak $\cG$-algebras,
  \item $1$-morphism category $1\Hom_{\cG \mhyphen N\AAlg_{\Fld}^{\fd}}(A,B)$ -- the full subcategory of the equivariant objects of $1\Hom_{\AAlg_{\Fld}^{\fd}}(A,B)$ spanned by pairs $(M, \omega_M)$ which, in addition, satisfy $\omega_M(e_{G_1}) = \id_M$ and
\begin{equation}
\label{eq:Real1MorAx3}
\omega_M(\partial x) = \partial(\omega_{?,2}(x)), \qquad x \in G_2.
\end{equation}
\end{itemize}
The horizontal composition of $1$-morphisms and the associativity data of $\AAlg_{\Fld}^{\fd}$ extend naturally to $\cG \mhyphen N\AAlg_{\Fld}^{\fd}$.

Denote by $\cG \mhyphen RW\AAlg_{\Fld}^{\fd}$ and $\cG \mhyphen \AAlg_{\Fld}^{\fd}$ the full subbicategories of $\cG \mhyphen N\AAlg_{\Fld}^{\fd}$ spanned by RW-weak and strict $\cG$-algebras, respectively.

\begin{Ex}
A strict morphism $\phi: A \rightarrow B$ of separable N-weak $\cG$-algebras defines in a canonical way a $1$-morphism $B_{\phi}:A \rightarrow B$ in $\cG \mhyphen N\AAlg_{\Fld}^{\fd}$.
\end{Ex}

\begin{Rem}
\begin{enumerate}[label=(\roman*)]
\item A more conceptual definition of $1\Hom_{\cG \mhyphen N\AAlg_{\Fld}^{\fd}}(A,B)$ is as the equivariant groupoid of $1\Hom_{\AAlg_{\Fld}^{\fd}}(A,B)$, viewed as a Real $2$-representation of $\tcG$. We opt to avoid defining equivariant objects for $2$-groups.

\item When $\cG$ is trivially graded, we do not require the existence of evaluation isomorphisms. We can therefore define a larger bicategory $\cG \mhyphen N\AALG_{\Fld}$. In this way we connect with the $\cG$-equivariant Morita contexts of \cite[\S 2]{rumynin2018}.
\end{enumerate}
\end{Rem}

\begin{Def}
Equivalence in the bicategory $\cG \mhyphen N\AAlg^{\fd}_{\Fld}$ is called $\cG$-Morita equivalence.
\end{Def}

We describe some additional structures on $\cG \mhyphen N\AAlg^{\fd}_{\Fld}$. Let $A$ and $B$ be N-weak $\cG$-algebras. Then the direct sum $A \oplus B$ has an obvious N-weak $\cG$-algebra structure $A \boxplus B$. Similarly, the tensor product $A \otimes_{\Fld} B$ is an N-weak $\cG$-algebra $A \boxtimes B$ with structure maps $\omega_{A \boxtimes B,i} = \omega_{A,i} \otimes \omega_{B,i}$, $i=1, 2,3$. This extends to a symmetric monoidal structure $\boxtimes$ on $\cG \mhyphen N\AAlg_{\Fld}^{\fd}$. Both $\boxplus$ and $\boxtimes$ restrict to $\cG \mhyphen RW\AAlg_{\Fld}^{\fd}$ and $\cG \mhyphen \AAlg_{\Fld}^{\fd}$.

Finally, given a strict $\cG$-algebra $A$, its dual $\cG$-algebra $A^{\vee}$ (see \cite[\S 3]{rumynin2018}) is defined so that its underlying ${\Fld}$-algebra is $A^{\opp}$ and its structure maps are
\[
\omega_{A^{\vee},1} (g) = \omega_{A,1}(g)^{\opp},
\qquad
\omega_{A^{\vee},2}(x) = \omega_{A,2}(x)^{-1}.
\]
An interested reader can generalize the construction of $A^{\vee}$ to the case of an N-weak $\cG$-algebra $A$.

\subsection{Realizability of twisted \texorpdfstring{$2$}{}-cocycles}
\label{sec:realizability}

Let $G$ be a $\ZTwo$-graded group. Fix an integer $t \geq 1$ and let $A = {\Fld}^t$ with RW-weak $G$-algebra structure $\omega_A$. Then $\omega_{A,1}: G \rightarrow \Aut^{\gen} (A) \cong \Sym\times \ZTwo$ is a group homomorphism. The datum $\omega_A$ determines a Real $2$-representation $\lambda$ of $G$ on $A\Mfd$, the category of finite dimensional left $A$-modules, considered as an object of the $2$-category $\mathsf{Cat}$ with duality involution $(-)^{\opp}$. The action functors are defined on objects by
\[
\lambda(g)(M) = {_{\omega_{A,1}(g)^{-1}}} \left({^{\pi(g)}}M\right), \qquad g \in G
\]
and the coherence natural isomorphisms
\[
\lambda_{g_2,g_1}: \lambda(g_2) \circ {^{\pi(g_2)}}\lambda(g_1) \Longrightarrow \lambda(g_2 g_1), \qquad g_i \in G
\]
are left multiplication by ${^{\omega_{A,1}(g_2 g_1)^{-1}}}\omega_{A,3}(g_2, g_1)^{-1}$ (cf. Section~\ref{sec:moritaBicat}).

Recall that $A_{\pi}^{\times}$ denotes the group of units $A^{\times}$ considered as a $G$-module via the homomorphism $\pi: G \rightarrow \ZTwo$.

\begin{Def}
A realization of $\omega_{A,3} \in Z^2(G, A_{\pi}^{\times})$ is an equivariant object  $(M, \alpha)$ of $\lambda$ whose underlying $A$-module $M$ is faithful.
\end{Def}

It is convenient to interpret $A$-linear equivariant structure maps
\[
\alpha_g: {_{\omega_{A,1}(g)^{-1}}} \left({^{\pi(g)}}M\right) \rightarrow M,
\qquad
g \in G
\]
as $\Fld$-linear maps $\mathfrak{p}_M(g): {^{\pi(g)}}M \rightarrow M$ (equal to $\alpha_g$ as maps of sets) which satisfy $\mathfrak{p}_M(g)(am) = \omega_{A,1}(g)(a) \mathfrak{p}_M(g)(m)$. The coherence constraints on $\mathfrak{p}_M$ then read
\[
\mathfrak{p}_M(g_2 g_1) = \omega_{A,3}(g_2 , g_1) \mathfrak{p}_M(g_2) \circ {^{\pi(g_2)}}\mathfrak{p}_M(g_1)^{\pi(g_2)} \circ \ev_M^{\delta_{\pi(g_2),\pi(g_1),-1}}, \qquad g_i \in G.
\]
In particular, with this interpretation we recover, for trivially graded $G$, the notion of realizability introduced in \cite[\S 3]{rumynin2018}.

\begin{Lem}
\label{lem:finiteRealizable}
If $G$ is finite, then any cocycle $\omega_{A,3} \in Z^2(G, A_{\pi}^{\times})$ is realizable.
\end{Lem}

\begin{proof}
A realization of the underlying untwisted $2$-cocycle $\omega_{A,3} \in Z^2(G_0, A^{\times})$ is given by the left regular representation of the skew group algebra $A \sharp_{\omega_A} G_0$. The hyperbolic representation on the left regular representation of $A \sharp_{\omega_A} G_0$ is then a realization of $\omega_{A,3} \in Z^2(G, A_{\pi}^{\times})$.
\end{proof}


The following result is used repeatedly in the remainder of the paper.

\begin{Prop}
\label{prop:2RepFromRep}
A realization $(M,\mathfrak{p}_M)$ of $\omega_{A,3} \in Z^2(G, A_{\pi}^{\times})$ determines a $\ZTwo$-graded group homomorphism $\mathfrak{a}:G \rightarrow \Aut^{\gen}(\End_A(M))$, $g\mapsto\mathfrak{a}_g$, by the formula
\[
\mathfrak{a}_{g}(\phi) = \mathfrak{p}_M(g) \circ {^{\pi(g)}} \phi \circ \mathfrak{p}_M(g)^{-1}, \qquad g \in G,\, \phi \in \End_A(M)
\]
where ${^{\pi(g)}}(-)$ determines the application of $A$-linear duality $(-)^{\circ}$.
\end{Prop}

\begin{proof}
Since $\mathfrak{p}_M(g)^{-1} (a m) = \omega_{A,1}(g)^{-1}(a) \mathfrak{p}_M(g)^{-1} (m)$, the map $\mathfrak{a}_{g}(\phi): M \rightarrow M$ is again $A$-linear. It is clear that $\mathfrak{a}_{g} \in \Aut(\End_A(M))$ when $\pi(g)=1$. If $\pi(g)=-1$, then
\begin{eqnarray*}
\mathfrak{a}_{g}(\phi_2 \circ \phi_1)
&=&
\mathfrak{p}_M(g) \circ (\phi_2 \circ \phi_1)^{\circ} \circ \mathfrak{p}_M(g)^{-1}
=
\mathfrak{a}_{g}(\phi_1) \circ \mathfrak{a}_{g}(\phi_2).
\end{eqnarray*}
Hence, we indeed have a map $\mathfrak{a}:
G \rightarrow \Aut^{\gen}(\End_A(M))$. To verify that $\mathfrak{a}$ is a group homomorphism, suppose, for instance, that $\pi(g_1)=\pi(g_2)=-1$. We compute
\begin{eqnarray*}
\mathfrak{a}_{g_2} (\mathfrak{a}_{g_1}(\phi))
&=&
\mathfrak{p}_M(g_2) \circ \mathfrak{p}_M(g_1)^{- \circ} \circ \phi^{\circ \circ} \circ \mathfrak{p}_M (g_1)^{\circ} \circ \mathfrak{p}_M(g_2)^{-1}\\
&=&
\omega_{A,3}(g_2 , g_1)^{-1}\mathfrak{p}_M(g_2 g_1) \circ \ev^{-1}_M \circ \phi^{\circ \circ} \circ \ev_M \circ \mathfrak{p}_M(g_2 g_1)^{-1} \omega_{A,3}(g_2 , g_1) \\
&=&
\mathfrak{p}_M(g_2 g_1) \circ \phi \circ \mathfrak{p}_M(g_2 g_1)^{-1}
=
\mathfrak{a}_{g_2 g_1} (\phi).
\end{eqnarray*}
The other cases are similar.
\end{proof}

\begin{Rem}
This section can be retold for N-weak $G$-algebras, but we do not require that level of generality.
\end{Rem}

\subsection{Strictification of weak \texorpdfstring{$\cG$}{}-algebras}
\label{sec:strictWeakAlg}

In this section we prove that, subject to a realizability condition, every split semisimple N-weak $\cG$-algebra is $\cG$-Morita equivalent to a strict $\cG$-algebra. When the $\ZTwo$-grading of $\cG$ is trivial this result is known \cite[Corollary 3.3]{rumynin2018}, although the proof there contains a gap. We provide a complete proof in this section, which covers also the $\ZTwo$-graded generalization. Our proof is more conceptual and is different from the proof in \cite{rumynin2018}, even in the ordinary case.

Let $A$ be a split semisimple N-weak $\cG$-algebra. The Artin--Wedderburn theorem asserts that there is a ${\Fld}$-algebra decomposition
\[
A \simeq \bigoplus_{n =1}^m M_n({\Fld})^{\oplus t_n}.
\]
Each summand $A_n\coloneqq M_n({\Fld})^{\oplus t_n}$ is an N-weak $\cG$-subalgebra and $A \simeq \boxplus_{n =1}^m A_n$. Let us consider the case $A = M_n({\Fld})^{\oplus t}$. There are isomorphisms
\[
\Aut(A) \simeq PGL_n({\Fld})^{\times t} \rtimes \Sym,
\qquad
\Aut^{\gen}(A) \simeq  \Aut(A) \rtimes \ZTwo.
\]
There is a canonical choice of a generator $s\in \ZTwo \leq \Aut^{\gen} (A)$: the transposition $s(a_1,\ldots ,a_t)=(a_1^T,\ldots, a_t^T)$. Its action on $\Aut (A)$ is the inverse transpose:
$$
{}^{s}\big( (a_1,\ldots ,a_t), \tau \big) = \big( (a_1^{-1})^T,\ldots, (a_t^{-1})^T), \tau \big).
$$
The composition
$$
\sigma: G_1 \xrightarrow{\omega_{A,1}} \Aut^{\gen}(A) \twoheadrightarrow \Sym
$$
is a group homomorphism. Indeed, $\omega_{A,1}$ fails to be a homomorphism by conjugation by elements of $GL_n({\Fld})^{\times t}$, which is not seen at the level of permutations. Choose a lift ${\Lambda}$ of $\omega_{A,1}$ along the quotient $GL_n({\Fld})^{\times t}\rtimes \Sym \rightarrow PGL_n({\Fld})^{\times t}\rtimes \Sym$. This determines a function $\mu_{\Lambda}: G_1 \times G_1 \rightarrow ({\Fld}^{\times})^t$ via the equation
\[
{\Lambda}(g_2 g_1) = \mu_{\Lambda}(g_2,g_1) \omega_{A,3}(g_2,g_1) {\Lambda}(g_2) ( {^{\pi(g_2)}}{\Lambda}(g_1)^{\pi(g_2)}).
\]
Using the $2$-cocycle condition on $\omega_{A,3}$, we find that $\mu_{\Lambda} \in Z^2(G_1, ({\Fld}^{\times})_{\pi}^t)$, where ${\Fld}^t$ is viewed as a $G_1$-algebra via $\sigma$. It is straightforward to verify that different choices of ${\Lambda}$ lead to cohomologous $2$-cocycles. In this way, we attach a cohomology class $[\mu_n] \in H^2(G_1, ({\Fld}^{\times})_{\pi}^{t_n})$ to $M_n({\Fld})^{\oplus t_n} \unlhd A$.

\begin{Thm}
\label{thm:strictcGAlg}
Let $\cG$ be a $\ZTwo$-graded crossed module, $A$ a split semisimple N-weak $\cG$-algebra. If each cohomology class $[\mu_n] \in H^2(G_1, ({\Fld}^{\times})_{\pi}^{t_n})$, $n \geq 1$, is realizable, then there exists a $\cG$-algebra $B$ which is $\cG$-Morita equivalent to $A$.
\end{Thm}

\begin{proof}
By the discussion preceding the theorem, it suffices to consider the case $A = M_n({\Fld})^{\oplus t}$. Fix a lift ${\Lambda}$ of $\omega_{A,1}$ with associated cocycle $\mu =\mu_{\Lambda}\in Z^2(G_1, ({\Fld}^{\times})_{\pi}^t)$. The definition of ${\Lambda}$ implies the equality
\[
\omega_{A,1}(g) (a) = {\Lambda}(g) ( {^{\pi(g)}}a) {\Lambda}(g)^{-1}, \qquad g \in G_1, \, a \in A.
\]
Equation \eqref{eq:NAlgAx1} implies that $\omega_{A,2}$ satisfies
\[
{\Lambda}(\partial x ) \circ a \circ {\Lambda}(\partial x)^{-1} = \omega_{A,2}(x) a \omega_{A,2}(x)^{-1}, \qquad x \in G_2, \, a \in A.
\]
Hence, there exists a function $\gamma : G_2 \rightarrow ({\Fld}^{\times})^t$ such that
\[
\omega_{A,2}(x) = \gamma (x) {\Lambda}(\partial x), \qquad x \in G_2
\]
which, by equation \eqref{eq:omega2Proj}, satisfies
\begin{equation}
\label{eq:cocycRatTriv}
\gamma(x_2 x_1) = \mu(\partial x_2, \partial x_1)^{-1} \gamma(x_2) \gamma(x_1).
\end{equation}

By the realizability assumption, there exists a $\mu^{-1}$-projective Real representation $(U,\eta_U)$ of $G_1$. Put $B = \End_{{\Fld}^t}(U)$ with the $G_1$-action of Proposition \ref{prop:2RepFromRep}. Define $\omega_{B,2}: G_2 \rightarrow \Aut_{{\Fld}^t}(U)$ by $\omega_{B,2}(x) = \gamma(x)^{-1} \eta_U(\partial x)$. Equation \eqref{eq:cocycRatTriv} implies that this makes $B$ into a $\cG$-algebra.

Let $V = ({\Fld}^n)^{\oplus t}$, viewed as a Real representation of $G_1$ via ${\Lambda}$. Let $M = V \otimes_{\Fld} U$ with $A \mhyphen B^{\vee}$-bimodule structure $a \cdot v \otimes u \cdot b = a v \otimes b u$. Here $B^{\vee}$ denotes the $\cG$-algebra dual to $B$; see Section \ref{sec:moritaBicat}. For each $g \in G_1$, define a ${\Fld}$-linear map $\omega_M : {^{\pi(g)}}M \rightarrow M$ by
\[
\omega_M(g)(v \otimes u) = {\Lambda}(g)(v) \otimes \eta_U(g)(u).
\]
Then we have
\begin{eqnarray*}
\omega_M(g)(a \cdot v \otimes u \cdot b) 
&=&
{\Lambda}(g)({^{\pi(g)}}av) \otimes \eta_U(g)({^{\pi(g)}}bu)\\
&=&
{\Lambda}(g) ( {^{\pi(g)}}a) {\Lambda}(g)^{-1} {\Lambda}(g)(v) \otimes \eta_U(g) ({^{\pi(g)}}a ) \eta_U(g)^{-1} \eta_U(g)(u) \\
&=&
\omega_{A,1}(g)(a) \cdot \omega_M(g)(v \otimes u) \cdot \omega_{B,1}(g)(b),
\end{eqnarray*}
so that equation \eqref{eq:Real1MorAx1} is satisfied. Moreover, since $\omega_{B,3}$ is trivial,
\[
\partial(\omega_{?,3}(g_2,g_1)) \omega_M(g_2)({^{\pi(g_2)}}\omega_M(g_1)^{\pi(g_2)}(v \otimes u))
\]
is equal to
\begin{eqnarray*}
&&
\omega_{A,3}(g_2,g_1) {\Lambda}(g_2) ({^{\pi(g_2)}}{\Lambda}(g_1)^{\pi(g_2)} (v)) \otimes \eta_U(g_2)({^{\pi(g_2)}}\eta_U(g_1)^{\pi(g_2)} (u))\\
&=&
\mu(g_2,g_1)^{-1} {\Lambda}(g_2 g_1)(v) \otimes \mu(g_2,g_1) \eta_U(g_2 g_1)(u) \\
&=&
{\Lambda}(g_2 g_1)(v) \otimes \eta_U(g_2 g_1)(u)
=
\omega_M(g_2 g_1)(v \otimes u),
\end{eqnarray*}
so that equation \eqref{eq:Real1MorAx2} is satisfied. Finally, we have
\[
\omega_M(\partial x)(v \otimes u)
=
{\Lambda}(\partial x) (v) \otimes \eta_U(\partial x) (u)
=
\gamma(x)^{-1} \omega_{A,2}(x) (v) \otimes \eta_U(\partial x) (u),
\]
so that equation \eqref{eq:Real1MorAx3} is satisfied. Hence $(M,\omega_M): B^{\vee} \rightarrow A$ is a $1$-morphism in $\cG \mhyphen N \AAlg_{\Fld}^{\fd}$. The bimodule $M$ is clearly an equivalence in $\AAlg_{\Fld}^{\fd}$. This implies that $(M, \omega_M)$ is also an equivalence. 
\end{proof}

\subsection{Strictification of $2$-groups}
The following proposition is an alternative version of Theorem~\ref{thm:strictcGAlg}, where the strictification alters the crossed module representing the $2$-group, instead of altering the algebra. Its part~(i) is quite formal and probably known in the ungraded but we are still compelled to give a complete proof.

\begin{Prop}
  \label{prop:strictification}
  Let $\cG$ be a $\ZTwo$-graded crossed module and $(A,\omega)$ an N-weak $\cG$-algebra.
There exists a $\ZTwo$-graded crossed module $\cH$, a strict homomorphism of $\ZTwo$-graded crossed modules $\varphi:\cH \rightarrow \cG$ and a strict $\cH$-algebra structure $\psi: \cH \rightarrow \AUT^{\gen}(A)$ such that
  \begin{enumerate}[label=(\roman*)]
  \item the induced homomorphism of $2$-groups $\widetilde{\varphi}:\tcH \rightarrow\tcG$ is an equivalence, and
    \item the N-weak $\cH$-algebra $(A,\omega\varphi)$ and the strict $\cH$-algebra $(A,\psi)$ are $\cH$-Morita equivalent.
      \end{enumerate}
\end{Prop}

%
%
\begin{proof}
Define constituents $H_i$, $i=1,2$ of the crossed module $\cH$ as extensions of $G_i$ by the multiplicative group $A^\times$, using the systems of factors that arise from the N-weak structure:
  $$
  1 \rightarrow A^\times \rightarrow H_i \xrightarrow{\varphi_i} G_i \rightarrow 1.
  $$
 The homomorphism $\varphi$ is $(\varphi_1,\varphi_2)$, i.e., $\varphi|_{H_i}=\varphi_i$. Writing $H_i = A^{\times} \times G_i$ as sets, we can express the homomorphism $\psi=(\psi_1,\psi_2)$ as
 \begin{equation} \label{eq:psi}
 \psi_1 (a,g) = \Ad_A (a)  \omega_1(g),
 \qquad
 \psi_2 (a,x) = a \omega_2 (x),
 \end{equation}
where, as usual, $a,b\in A^{\times}$, $g,h\in G_1$ and $x,y\in G_2$. The rest of the proof is devoted to verifying the technical details.
 
{\bf 1) Groups $H_i$:} the multiplications for $H_1$ and $H_2$ are slight variations of the standard product defined by the system of factors:  
\begin{equation}\label{ext_product}
  (a,g)\cdot (b,h) \coloneqq
  (a \omega_1 (g)(b) \omega_3 (g,h)^{-1}, gh)
  , \ \ 
(a,x)\cdot (b,y) \coloneqq 
(a ( {}^{\omega_2 (x)}b ) \omega_3 (\partial x, \partial y)^{-1}, xy).
\end{equation}
Associativity follows from the cocycle condition. The inverses are 
\[
(a,g)^{-1} =
(\omega_3 (g^{-1},g)\omega_1(g^{-1})(a^{-1}), g^{-1})
  , \ \ 
(a,x)^{-1} =
(\omega_3 (\partial x^{-1}, \partial x) ({}^{\omega_2(x^{-1})}(a^{-1})), x^{-1}).
\]

{\bf 2) Action of $H_1$ on $H_2$:}  define the action by
\begin{equation}\label{ext_action}
{}^{(a,g)}(b,x) \coloneqq  (
a
\omega_1(g)(b) 
\omega_3 (g,\partial x)^{-1}
\omega_1(g \partial x) (\omega_3(g^{-1},g))
\omega_3 (g\partial x, g^{-1})^{-1}
({}^{\omega_2({}^gx)}(a^{-1}))
,{}^{g}x).
\end{equation}
We have derived this formula by assuming that $G_2$ is a normal subgroup of $G_1$ and computing $((a,g)\cdot(b,x))\cdot (a,g)^{-1}$. Observe that the evaluation of $(a,g)\cdot((b,x)\cdot (a,g)^{-1})$ produces the same formula (after a longer calculation, utilizing the cocycle condition).
For clarity we simplify the notation by $\omega=\omega_3$ and dropping $\omega_1$ and $\omega_2$:
\[
{}^{(a,g)}(b,x) = (
a \, 
({}^{g}b) \,  
\omega (g,\partial x)^{-1} \, 
({}^{g \partial x} \omega(g^{-1},g)) \, 
\omega (g\partial x, g^{-1})^{-1} \, 
({}^{{}^gx}(a^{-1}))
,{}^{g}x).
\]

We need to verify that each element of $H_1$ acts by a group endomorphism. It suffices to verify this for elements of $G_i$ and $A^\times$ separately. Eight separate verifications are required but it is easier due to the complexity of equation ~\eqref{ext_action}. Here we show one of them, leaving the remainder to an interested reader:
\begin{align*}  
  {}^{(a,e)}(1,x) \cdot {}^{(a,e)} (b,e) =
  ( a ({}^{x}(a^{-1})),x) & 
  ( a b a^{-1},e) =  ( a ({}^{x}(a^{-1})) {}^{x}( a b a^{-1}),x) =
  \\
  ( a ({}^{x}(b)) {}^{x}(a^{-1}),x) =
  &
    {}^{(a,e)}({}^{x}b,x) =
      {}^{(a,e)} ((1,x) \cdot(b,e)).
  \end{align*}
Further we need to verify that it is actually an action. Again, verifying this for the elements of $G_i$ and $A^\times$ separately requires eight separate verifications. We perform just one of them as an illustration.
%
%
%
%
We need to show that
\[
{}^{(1,g)} ( {}^{(a,e)} (1,x))
=
({}^{g}a  \, {}^g ({}^{x}(a^{-1}))  \omega (g,\partial x)^{-1}\, 
    {}^{g\partial x}\omega (g^{-1},g) \omega(g\partial x, g^{-1})^{-1},{}^{g}x)
\]
is equal to
\[
{}^{(1,g)(a,e)} (1,x) = {}^{({}^ga,g)} (1,x)
=
( {}^{g}a   \omega(g,\partial x)^{-1} \, {}^{g\partial x}\omega(g^{-1},g) \omega (g\partial x,g^{-1})^{-1}\, {}^{{}^gx}({}^{g}(a^{-1})),{}^{g}x).
\]
Using ${}^{g}({}^{h} a) = \omega(g,h)^{-1} \, {}^{gh} a \, \omega(g,h)$
(equation~\eqref{eq:omega2Proj}), we are left to verify that
\[
\underline{\omega(g,\partial x)^{-1}} {^{g \partial x}}(a^{-1}) \underline{\omega(g,\partial x) \omega(g,\partial x)^{-1}} {^{g \partial x}} \omega(g^{-1},g) \omega(g \partial x, g^{-1})^{-1}
\]
is equal to
\[
\underline{\omega(g, \partial x)^{-1}} \;
\underline{{^{g \partial x}} \omega(g^{-1},g) \omega(g \partial x, g^{-1})^{-1} \omega(g\partial(x) g^{-1}, g)^{-1}} {^{g \partial x}} (a^{-1}) \omega(g\partial(x) g^{-1}, g).
\]
We can cancel all underlined parts, for instance, the three terms in the second expression is equation~\eqref{eq:omega3Cocycle} with $g_3=g\partial (x)$, $g_2=g^{-1}$ and $g_1=g$. Hence, we need to see that the equality
\[
{^{g \partial x}}(a^{-1}) {^{g \partial x}} \omega(g^{-1},g) \omega(g \partial x, g^{-1})^{-1}
\overset{?}{=}
{^{g \partial x}} (a^{-1}) \omega(g \partial(x) g^{-1}, g)
\]
holds. Again this follows from equation~\eqref{eq:omega3Cocycle} with $g_3=g\partial (x)$, $g_2=g^{-1}$ and $g_1=g$.


{\bf 3) Differential for $\cH$ and the Peiffer identity:}  
define the differential by
\[
  \partial (a,x) \coloneqq (a, \partial x). 
\]
  The similarities in the definitions of the products in $H_1$ and $H_2$ (formula~(\ref{ext_product})) ensure that $\partial$ is a group homomorphism. The Peiffer identity holds automatically, due to the way we have derived formula~(\ref{ext_action}) in {\bf 2)}.
    

{\bf 4) Homomorphism $\varphi$:} define $\varphi_i : H_i \rightarrow G_i$ by
$$
\varphi_1  (a,g) =g,
\qquad
\varphi_2  (a,x) =x.
$$
It is easy to see that $\varphi = (\varphi_1 , \varphi_2)$ is a strict homomorphism of crossed modules. Notice that $\pi_i(\cH) = \{e\}\times \pi_i(\cG)$, $i=1,2$, and $\pi_i (\varphi):\pi_i(\cH)\rightarrow\pi_i(\cG)$ are identities.

{\bf 5) Grading:} if $\pi: G_1 \rightarrow {\mathbb Z}_2$ is the grading on $\cG$, a grading on $\cH$ is given by $\pi \circ \varphi_1 : (a,g)\mapsto \pi (g)$. Clearly, $\varphi = (\varphi_1 , \varphi_2)$ is a homomorphism of graded crossed modules.

{\bf 6) Homomorphism $\psi$:}
it is defined above (formula~\eqref{eq:psi}).
Let us verify that $\psi_2$ is a homomorphism:
\begin{align*}  
  \psi_2 ((a,x)(b,y))  = \psi_2 ((a\, {}^{x}b \omega_1(\partial x,\partial y)^{-1},xy))
    &
    = a\, {}^{x}b \omega_1(\partial x, \partial y)^{-1} \omega_2(xy)   =
      \\
      a \omega_2 (x) b \omega_2(x)^{-1} \omega_2(x) \omega_2(y) =
      &
      a \omega_2 (x) b \omega_2(y) =
           \psi_2 (a,x) \psi_2(b,y).
  \end{align*}
The verification for $\psi_1$ is similar. If $\pi(g)=1$, it is identical. If $\pi(g)=-1$, we denote $\circ{}^{-1}$ by $\bullet$. The key observation is
$\omega_1(g)\bullet \omega_1(g) = \id_A$ so that
\begin{align*}
  \psi_1 ((a,g)(b,h)) =
  \Ad(a)\circ \Ad({}^{g}b) \circ
  \Ad(\omega_1(g, h)^{-1}) \circ  \omega_1(gh) &   =
  \Ad(a)\circ 
  \\
  \omega_1 (g) \bullet \Ad(b) \circ \omega_1(g^{-1}) \bullet \omega_1(g) \bullet \omega_1(h)   =  \Ad(a)\circ \omega_1 (g)
  &
  \bullet \Ad(b) \circ \omega_1(h),
  \end{align*}
which is equal to $\psi_1 (a,g) \psi_1(b,h)$.

{\bf 7) Morita equivalence:} It suffices to observe that the identity bimodule $(M={}_AA_A,\theta_M)$ with
$$
\theta_M (a,g) (m) = a \omega_1(g)(m)
$$
yields the Morita equivalence of $\cH$-algebras $(A,\omega\varphi)\rightarrow (A,\psi)$. Let us carry out the necessary verifications: 
\begin{align*}
  \theta_M (a,g) (bmc) = a\omega_1(g)(bmc) =
  &
\, a  \omega_1(g)(b) \omega_1(g)(m) \omega_1(g)(c) =
  \\
   \Ad_A(a)( \omega_1(g)(b)) a \omega_1(g)(m) \omega_1(g)(c) =
  &
 \,   \psi_1(a,g)(b) \theta_M(a,g)(m) \omega_1(\varphi_1(a,g))(c)
\end{align*}
and
\begin{align*}
  \theta_M \big( (a,g)(b,h) \big) (m) = &
  \, a \, {}^{g}b \omega_3(g,h)^{-1} \omega_1(gh)(m) =
  \\
  a \, {}^{g}b \omega_3(g,h)^{-1} \Ad \big( \omega_3(g,h)\big) \big(  \omega_1(g) (\omega_1(h)(m))\big)\,
  &
  =
  a \, {}^{g}b \omega_1(g) \big( \omega_1(h)(m)\big) \omega_3(g,h)^{-1}
  \\
  =
  \partial\big(\omega_{?,3}((a,g),(b,h)) \big) & \big(\theta_M(a,g) \big( \theta_M (b,h)(m)\big)\big). 
   \end{align*}
\end{proof}

Proposition~\ref{prop:strictification} reduces all computations with Real $2$-modules of a $2$-group $\tG$ to manipulations with strict $\cG$-algebras, albeit for different crossed modules. For instance, take two $2$-modules $V$ and $W$. A straightforward variation of Proposition~\ref{prop:strictification} yields a crossed module $\cG$ and strict $\cG$-algebras $A$ and $B$ that realize $V$ and $W$. This allows us to define 
$V^\vee$, $V\boxplus W$ and $V\boxtimes W$ as in the end of Section~\ref{sec:moritaBicat}.


\begin{Prob}
  Characterize those crossed modules $\cG$ such that any N-weak $\cG$-algebra is $\cG$-Morita equivalent to a strict $\cG$-algebra. Characterize those $2$-groups $\tG$ which admit a crossed module realization $\tG\cong\tcG$ such that $\cG$ satisfies the property in the previous sentence. 
\end{Prob}

\section{Induction of \texorpdfstring{$\cG$}{}-algebras}
\label{sec:induction}

In this section we define and study induction for $\cG$-algebras, in both the ordinary and $\ZTwo$-graded settings. In view of Theorem \ref{thm:strictcGAlg}, Proposition~\ref{prop:strictification} and our later applications to Real $2$-representation theory, we restrict our attention to strict algebras.

\subsection{The ordinary case}
\label{sec:indKAlg}

Let $\cG$ be a crossed module with a crossed submodule $\cH$. There is an associated restriction pseudofunctor $\Res_{\cH}^{\cG}: \cG \mhyphen \AALG_{\Fld} \rightarrow \cH \mhyphen \AALG_{\Fld}$. The goal of this section is to define an induction pseudofunctor $\Ind_{\cH}^{\cG} : \cH \mhyphen \AALG_{\Fld} \rightarrow \cG \mhyphen \AALG_{\Fld}$. The $\ZTwo$-graded case is treated in Section \ref{sec:indRealKAlg}.
Our construction generalizes the known definition (at the level of objects) in the case $H_2 =G_2$ and $\vert G_1:H_1 \vert <\infty$ \cite[\S 3]{rumynin2018}.

Let $A$ be an $\cH$-algebra. We define a $\cG$-algebra $\widetilde{A} = \Ind_{\cH}^{\cG} A$ as follows. Fix a left transversal $\sT$ to $H_1$ in $G_1$. For each $t \in \sT$, denote by ${\Fld}G_{2,t}$ the group algebra of $G_2$ with right ${\Fld}H_2$-module structure
\[
x \cdot z = x ({^{t}}z), \qquad x \in G_2, \, z \in H_2.
\]
As a vector space, set
\[
\widetilde{A} = \prod_{t \in \mathcal{T}} {\Fld}G_{2,t} \otimes_{{\Fld}H_2} A.
\]
Explicitly, the tensor relations in $\widetilde{A}$ read
\[
[x ({^{t}}z) \otimes a]_t = [x \otimes \omega_{A,2}(z) a]_t, \qquad x \in G_2, \, z \in H_2, \, a \in A
\]
where $[-]_t$ denotes an element of the $t$\textsuperscript{th} factor of $\widetilde{A}$.
Let $\delta_{g_2,g_1}$ be the $H_1$-coset delta function: given $g_1,g_2 \in G_1$, 
$$
\delta_{g_2,g_1} =
\begin{cases}
  1 & \mbox{ if } g_1H_1=g_2H_1 , \\
  0 & \mbox{ otherwise.}
\end{cases}
$$


\begin{Lem}
The formula
\begin{equation}
\label{eq:indMult}
[x_2 \otimes a_2]_{t_2} \cdot [x_1 \otimes a_1]_{t_1}
=
\delta_{\partial (x_1^{-1}) t_2, t_1}
[x_2 x_1 \otimes \omega_{A,1}(  t_1^{-1} \partial ( x_1^{-1}) t_2 ) (a_2) a_1]_{t_1}
\end{equation}
defines an associative algebra structure on $\widetilde{A}$ with identity $1_{\tilde{A}} =([e_{G_2} \otimes 1_A]_t)_{t \in \sT}$.
\end{Lem}

\begin{proof}
To begin, observe that if $\delta_{\partial (x_1^{-1}) t_2, t_1}$ is non-zero, then the argument $t_1^{-1} \partial(x_1^{-1}) t_2$ of $\omega_{A,1}$ lies in $H_1$. The right hand side of equation \eqref{eq:indMult} is therefore well-defined. By construction only finite sums appear in the calculation of the product of two arbitrary elements of $\widetilde{A}$, and so the product in $\widetilde{A}$ is well-defined.

To be well-defined, equation \eqref{eq:indMult} must respect the tensor relations in $\widetilde{A}$. For $z \in H_2$, we have
\begin{equation}
\label{eq:firstProd}
[x_2 \cdot z \otimes a_2]_{t_2} \cdot [x_1 \otimes a_1]_{t_1}
=
\delta_{\partial (x_1^{-1}) t_2, t_1}
[x_2 ({^{t_2}}z) x_1 \otimes \omega_{A,1}(t_1^{-1} \partial ( x_1^{-1}) t_2)(a_2) a_1]_{t_1}.
\end{equation}
The crossed module axioms for $\cG$ give $x_2 ({^{t_2}}z) x_1 = x_2 x_1 ({^{\partial (x_1^{-1}) t_2}}z)$. Write $\partial (x_1^{-1}) t_2 = t^{\prime} h$ for $t^{\prime} \in \mathcal{T}$ and $h \in H_1$. If the product \eqref{eq:firstProd} is non-zero, then $t^{\prime}=t_1$. In this case, ${^{\partial (x_1^{-1}) t_2}}z = {^{t_1}}({^{h}}z)$ and, since $\cH$ is a crossed submodule of $\cG$, we have ${^{h}}z \in H_2$. Noting that $t_1^{-1} \partial ( x_1^{-1}) t_2 = h$, the product \eqref{eq:firstProd} becomes
\[
\delta_{\partial (x_1^{-1}) t_2, t_1}
[x_2 x_1 \otimes \omega_{A,2}({^{h}}z)\omega_{A,1}( h)(a_2) a_1]_{t_1}.
\]
The axiom  \eqref{eq:NAlgAx2} gives $\omega_{A,2}({^{h}}z) = \omega_{A,1}(h)(\omega_{A,2}(z))$, so that \eqref{eq:firstProd} can be written as
\[
\delta_{\partial (x_1^{-1}) t_2, t_1}
[x_2 x_1 \otimes \omega_{A,1}(h)(\omega_{A,2}(z)a_2) a_1]_{t_1}.
\]
This is plainly equal to $[x_2 \otimes \omega_{A,2}(z)a_2]_{t_2} \cdot [x_1 \otimes a_1]_{t_1}$, as required.

Similarly, for each $z \in H_2$, we have
\[
[x_2 \otimes a_2]_{t_2} \cdot [x_1 ({^{t_1}}z)  \otimes a_1]_{t_1}
=
\delta_{\partial (({^{t_1}}z)^{-1} x_1^{-1}) t_2, t_1}
[x_2 x_1 ({^{t_1}}z) \otimes \omega_{A,1}(  t_1^{-1} \partial ( x_1^{-1}) t_2 ) (a_2) a_1]_{t_1}
\]
and
\[
[x_2 \otimes a_2]_{t_2} \cdot [x_1 \otimes \omega_{A,2}(z) a_1]_{t_1}
=
\delta_{\partial(x_1^{-1}) t_2,t_1} [x_2 x_1 \otimes \omega_{A,1}(t_1^{-1} \partial(x_1^{-1}) t_2)(a_2) \omega_{A,2}(z)a_1]_{t_1}.
\]
Using that $\partial (({^{t_1}}z)^{-1} x_1^{-1}) = t_1 \partial(z^{-1}) t_1^{-1} \partial(x_1^{-1})$, a short calculation shows that the $\delta$-functions appearing in the two products are equal. Since
\[
[x_2 x_1 ({^{t_1}}z) \otimes \omega_{A,1}(  t_1^{-1} \partial ( x_1^{-1}) t_2 ) (a_2) a_1]_{t_1}
= 
[x_2 x_1 \otimes \omega_{A,2}(z) \omega_{A,1} (  t_1^{-1} \partial ( x_1^{-1}) t_2 ) (a_2) a_1]_{t_1},
\]
we see that the products are indeed equal.

We omit the verification of associativity and the identity property.
\end{proof}

Next, we define a $\cG$-algebra structure on $\widetilde{A}$.

\begin{Prop}
\label{thm:2IndCplx}
The maps
\[
\omega_{\tilde{A},1}(g)( [x \otimes a]_t ) = [{^{g}}x \otimes \omega_{A,1}(h)(a)]_{t^{\prime}}, \qquad g \in G_1
\]
where $g t =t^{\prime} h$ for $t^{\prime} \in \mathcal{T}$ and $h \in H_2$, and
\[
\omega_{\tilde{A},2}(x) = 
( \, [x \otimes 1_A]_t\, )_{t \in \mathcal{T}}, \qquad x \in G_2
\]
supply $\widetilde{A}$ with the structure of a $\cG$-algebra.
\end{Prop}

\begin{proof}
To begin, we verify that $\omega_{\tilde{A},1}(g)$ is an algebra homomorphism:
\begin{eqnarray*}
\omega_{\tilde{A},1}(g) \left( [x_2 \otimes a_2]_{t_2} \cdot [x_1 \otimes a_1]_{t_1} \right)
&=&
\delta_{\partial (x_1^{-1}) t_2, t_1} \omega_{\tilde{A},1}(g) [x_2 x_1 \otimes \omega_{A,1}(h) (a_2) a_1]_{t_1} \\
&=&
\delta_{\partial (x_1^{-1}) t_2, t_1} [{^{g}}(x_2 x_1) \otimes 
\omega_{A,1}(h_1  h) (a_2) \omega_{A,1}(h_1)(a_1)]_{t_1^{\prime}}.
\end{eqnarray*}
Here $g t_i = t_i^{\prime}h_i$ for $t_i^{\prime} \in \mathcal{T}$ and $h_i \in H_1$ and we have written $h$ for $t_1^{-1} \partial ( x_1^{-1}) t_2$. On the other hand, $\omega_{\tilde{A},1}(g)  ([x_2 \otimes a_2]_{t_2} )\cdot (\omega_{\tilde{A},1}(g)[x_1 \otimes a_1]_{t_1})$ is equal to
\begin{eqnarray*}
&&\delta_{\partial({^{g}}x_1^{-1}) g t_2, g t_1}
[{^{g}}x_2 {^{g}}x_1 \otimes
\omega_{A,1}(t_1^{\prime -1} \partial({^{g}}x_1^{-1}) t_2^{\prime} h_2)(a_2) \omega_{A,1}(h_1)(a_1)]_{t^{\prime}_1} \\
&=&
\delta_{\partial(x_1^{-1}) t_2, t_1} [{^{g}}(x_2 x_1) \otimes
\omega_{A,1}(h_1 h h_2^{-1})(\omega_{A,1}(h_2)(a_2) \omega_{A,1}(h_1)(a_1)]_{t^{\prime}_1} \\
&=&
\delta_{\partial(x_1^{-1}) t_2, t_1}
[{^{g}}(x_2 x_1) \otimes
\omega_{A,1}(h_1 h)(a_2) \omega_{A,1}(h_1)(a_1)]_{t^{\prime}_1},
\end{eqnarray*}
as required. To verify that $\omega_{\tilde{A},1}$ is a group homomorphism, we compute
\[
\omega_{\tilde{A},1}(g_2) \left(\omega_{\tilde{A},1}(g_1)(\, [x \otimes a]_t \, )\right)
=
[{^{g_2}}({^{g_1}}x) \otimes \omega_{A,1}(h_2)( \omega_{A,1}(h_1)(a))]_{t_2^{\prime}}
=
\omega_{\tilde{A},1}(g_2 g_1)[x \otimes a]_t.
\]
To verify that $\omega_{\tilde{A},2}$ is a group homomorphism, we compute
\[
\omega_{\tilde{A},2} (x_2) \cdot \omega_{\tilde{A},2} (x_1)
=
( \, \delta_{\partial(x_1^{-1}) t_2, t_1} [ x_2 x_1 \otimes 1_A]_{t_1} \, )_{t_1 \in \mathcal{T}}
=
( \, [x_2 x_1 \otimes 1_A]_{t_1}\, )_{t_1 \in \mathcal{T}}
=
\omega_{\tilde{A},2} (x_2 x_1).
\]
To verify equation \eqref{eq:NAlgAx1}, we compute
\[
\omega_{\tilde{A},1}(g)(\omega_{\tilde{A},2}(x))
=
\omega_{\tilde{A},1}(g)(\,  [x \otimes 1_A]_t \, )_{t \in \mathcal{T}}
=
(\, [{^{g}}x \otimes 1_A]_t \, )_{t \in \mathcal{T}}
=
\omega_{\tilde{A},2}({^{g}}x).
\]
Finally, to verify equation \eqref{eq:NAlgAx2}, we compute
\[
\omega_{\tilde{A},1}(\partial x_2)( [x_1 \otimes a_1]_{t_1} )
=
[{^{\partial x_2}}x_1 \otimes \omega_{A,1}(h)(a_1)]_{t^{\prime}_1},
\]
where $\partial(x_2) t_1 = t_1^{\prime} h$, while
\begin{eqnarray*}
&&\omega_{\tilde{A},2}(x_2) \cdot [x_1 \otimes a_1]_{t_1} \cdot \omega_{\tilde{A},2}(x_2)^{-1}
=
( \, [x_2 \otimes 1_A]_t \, )_{t \in \mathcal{T}} \cdot (x_1 \otimes a_1)_{t_1} \cdot (\, [x_2^{-1} \otimes 1_A]_t \, )_{t \in \mathcal{T}} \\
&=&
(\, 
\delta_{\partial(x_2) t_1, t}
[x_2 x_1 \otimes \omega_{A,1}(t^{-1}  \partial(x_2) t)(a)]_{t} \, )_{t \in \mathcal{T}}
=
[x_2 x_1 x_2^{-1} \otimes  \omega_{A,1}(t^{-1} \partial (x_2) t_1)(a_1)]_{t_1^{\prime}}.
\end{eqnarray*}
It follows that $t=t_1^{\prime}$ and hence $t^{-1} \partial (x_2) t_1 = h$. This completes the proof.
\end{proof}

We complete the construction of $\Ind_{\cH}^{\cG}$ as a pseudofunctor in the following section. We end this section by describing two particularly simple cases of the construction $\Ind_{\cH}^{\cG}$.

\begin{Ex}
\begin{enumerate}[label=(\roman*)]
\item Suppose that $G_1$ is trivial. The groups $H_2$ and $G_2$ are then abelian and the group homomorphism $\omega_{A,2} : H_2 \rightarrow A^{\times}$ factors through the centre $Z(A)$. In this case, there is an algebra isomorphism $\widetilde{A} = k[G_2] \otimes_{k [H_2]} A$, the tensor relations reading
\[
l k \otimes a =  l \otimes \omega_{A,2}(k) a, \qquad l \in G_2, \, k \in H_2, \, a \in A,
\]
and the $\cG$-algebra structure of $\widetilde{A}$ is $\omega_{\tilde{A},2}(l) = l \otimes 1_A$. Note that $\omega_{\tilde{A},2}(l) \in \widetilde{A}$ is central because $G_2$ is abelian and $1_A \in A$ is central. In particular, $\widetilde{A}$ can be seen as the result of inducing the $H_2$-representation $A$ to a $G_2$-representation.

\item Suppose instead that $G_2$ is trivial. Then the subgroup $H_1 \leq G_1$ is arbitrary and we are in the setting of \cite[\S 5]{rumynin2018} (without the finiteness assumption). There is an algebra isomorphism
\[
\widetilde{A} = \Gamma(G_1 \slash H_1, G_1 \times_{H_1} A) \simeq \prod_{t \in \mathcal{T}} A_t
\]
and the $\cG$-algebra structure reads
\[
\omega_{\tilde{A},1}(g)( [a]_t ) = [\omega_{A,1}(h)(a)]_{t^{\prime}}, \qquad g \in G_1
\]
where $g t =t^{\prime} h$ for $t^{\prime} \in \mathcal{T}$. In particular, $\widetilde{A}$ can be seen as the result of coinducing the $H_1$-representation $A$ to a $G_1$-representation.
\end{enumerate}
\end{Ex}

\subsection{A Maschke-type theorem for induced \texorpdfstring{$\cG$}{}-algebras}
\label{sec:maschke}

In order to ensure that induction transforms $2$-representations into $2$-representations, we need the following version of Maschke's Theorem for the ${\Fld}$-algebra $\widetilde{A}$.

\begin{Prop}
\label{thm:maschke}
Let $A$ be an $\cH$-algebra. Assume that both indices $\vert G_1:H_1 \vert$ and $\vert G_2:H_2 \vert$ are finite.
\begin{enumerate}[label=(\roman*)]
\item If $A$ is finite dimensional over ${\Fld}$, then so too is $\widetilde{A}$.

\item Assume that $\vert G_2:H_2 \vert$ is not divisible by the characteristic of ${\Fld}$. If $A$ is separable, then so too is $\widetilde{A}$.
\end{enumerate}
\end{Prop}

\begin{proof}
The first statement follows immediately from the definition of $\widetilde{A}$.

Turning to the second statement, let $w=\sum_j a_j \otimes b_j\in A\otimes_{\Fld} A^{\opp}$ be a separability idempotent for $A$. Pick left transversals $\sT$ to $H_1$ in $G_1$ and $\sX$ to $H_2$ in $G_2$. For each $x\in \sX$ and $t\in\sT$, there exists a unique $t_x \in \sT$ such that
  $\delta_{t\partial x,t_x}=1$. Consider the element
\[
\mathfrak{w} = \sum_{x\in\sX, t\in\sT, j}  [{}^tx \otimes a_j]_{t} \otimes [e \otimes b_j]_{t} [{}^tx^{-1} \otimes 1]_{t_x}
\in \widetilde{A} \otimes_{\Fld} \widetilde{A}^{\opp},
\]
where the multiplication is in $\widetilde{A}$. We claim that $\tilde{\mathfrak{w}} = \vert G_2:H_2 \vert^{-1} \mathfrak{w}$ is a separability idempotent for $\widetilde{A}$. First, we show that $\tilde{\mathfrak{w}}$ is sent to $1_{\tilde{A}}$ under the multiplication map $\widetilde{A} \otimes_{\Fld} \widetilde{A}^{\opp} \rightarrow \widetilde{A}$:
\begin{align*}  
\sum_{x,t,j}  [{}^tx \otimes a_j]_{t} [e \otimes b_j]_{t} [{}^tx^{-1} \otimes 1]_{t_x}=&
\sum_{x,t,j}  [{}^tx \otimes a_jb_j]_{t} [{}^tx^{-1} \otimes 1]_{t_x}=
\\
\sum_{x,t}  [{}^tx \otimes 1]_{t} [{}^tx^{-1} \otimes 1]_{t_x}
= &
\vert G_2:H_2 \vert \sum_{t} \delta_{\partial({}^tx)t,t_x} [e \otimes 1]_{t_x} = 
\vert G_2:H_2 \vert 1_{\tilde{A}}.
\end{align*}
In the final equality we used that $\partial({}^tx)t=t\partial(x)$. To verify that $\tilde{\mathfrak{w}}$ is $\widetilde{A}$-central, it is useful to write
  $$
  \mathfrak{w} = 
  \sum_{x, t, j}  [{}^tx \otimes 1]_{t} [e_{G_2} \otimes a_j]_{t} \otimes [e \otimes b_j]_{t} [{}^{t}x^{-1} \otimes 1]_{t_x}
  =
  \sum_{x, t}  [{}^tx \otimes 1]_{t} w_{t} [{}^tx^{-1} \otimes 1]_{t_x},
  $$
where $w_{t}$ denotes $w$ considered as an element of $A \otimes_{\Fld} A^{\opp}$ in degree $t \in \sT$. It suffices to check that $\mathfrak{a}\tilde{\mathfrak{w}}=\tilde{\mathfrak{w}}\mathfrak{a}$ when $\mathfrak{a} \in \widetilde{A}$ is of one of the following two forms:
\[
\mathfrak{a}= [e \otimes \gamma]_s
\qquad
\mbox{or}
\qquad
\mathfrak{a}= [y \otimes 1]_s.
\]
In the first case, pick $q = q(x,s) \in \sT$ such that $s=q_x$.
Observe that for each $t\in\sT$, the equalities $\delta_{t\partial x,s}=\delta_{\partial({}^tx^{-1})s,t}=1$ hold if and only if $t=q$. We compute
\begin{eqnarray*}
&&[e \otimes\gamma]_{s} \mathfrak{w}
=
\sum_{x, t}  [e \otimes \gamma]_{s} [{}^tx \otimes 1]_{t} w_{t} [{}^tx^{-1} \otimes 1]_{t_x} \\
&=&
\sum_{x, t}  \delta_{\partial ({}^tx^{-1}) s, t} [{}^tx \otimes \omega_{A,1}(  t^{-1} \partial ( {}^tx^{-1}) s) (\gamma) ]_{t} w_{t} [{}^tx^{-1} \otimes 1]_{t_x} \\
&=&
\sum_x  [{}^qx \otimes \omega_{A,1}( \partial ( x^{-1}) q^{-1}s) (\gamma) ]_{q} w_{q} [{}^qx^{-1} \otimes 1]_{q_x} \\
&=&
\sum_{x}  [{}^qx \otimes 1]_q [e \otimes \omega_{A,1}(\partial ( x^{-1}) q^{-1}s) (\gamma) ]_{q} w_{q} [{}^qx^{-1} \otimes 1]_{q_x} \\
&=&
\sum_{x}  [{}^qx \otimes 1]_q w_{q}
[e \otimes \omega_{A,1}( \partial ( x^{-1}) q^{-1}s) (\gamma) ]_{q}  [{}^qx^{-1} \otimes 1]_{q_x} \\
&=&
\sum_{x}  [{}^qx \otimes 1]_q w_{q} [{}^qx^{-1}\otimes \gamma ]_{q_x}
=
\sum_{x} [{}^qx \otimes 1]_q w_{q} [{}^q x^{-1}\otimes 1]_{q_x} [e \otimes \gamma ]_{q_x} \\
&=&
\sum_{x,t} \delta_{\partial ({}^tx^{-1}) s, t} [{}^tx \otimes 1]_t w_{t} [{}^tx^{-1}\otimes 1]_{t_x} [e \otimes \gamma ]_{s}
= \mathfrak{w} [e \otimes \gamma ]_{s}.
\end{eqnarray*}
For the second case, we rewrite $y\,{}^tx=\,{}^tx_y \,{}^th_y$ with $x_y = x_y (t) \in \sX$ and $h_y = h_y (t) \in H_2$ and compute
\begin{eqnarray*}
&&
[y\otimes 1]_{s}\mathfrak{w}
=
\sum_{x, t} [y\otimes 1]_{s}[{}^tx \otimes 1]_{t} w_{t} [{}^tx^{-1} \otimes 1]_{t_x} \\
&=&
\sum_{x,t}  \delta_{\partial({}^tx^{-1})s,t} [y\,{}^tx \otimes 1]_{t} w_{t} [{}^tx^{-1} \otimes 1]_{t_x}
=
\sum_{x}  [\,{}^qx_y \otimes \omega_{A,1} (h_y ) ]_{q} w_{q} [{}^qx^{-1} \otimes 1]_{s} \\
&=&
\sum_{x}  [\,{}^qx_y \otimes 1]_{q} [e \otimes \omega_{A,1} (h_y)]_{q} w_{q} [{}^qx^{-1} \otimes 1]_{s}
=
\sum_{x}  [\,{}^qx_y \otimes 1]_{q} w_{q} [{}^{q}h_y \otimes 1]_{q} [{}^qx^{-1} \otimes 1]_{s} \\
&=&
\sum_{x}  \delta_{\partial({}^qx)q,s} [\,{}^qx_y \otimes 1]_{q} w_{q} [{}^{q}h_y {}^qx^{-1} \otimes 1]_{s}
=
\sum_{x} [\,{}^qx_y \otimes 1]_{q} w_{q} [\,{}^qx_y^{-1} \, y \otimes 1]_{s}.
\end{eqnarray*}
%
Notice that we have used the equality $\delta_{\partial({}^qx)q,s} = 1$. Similarly, we have
$$
\mathfrak{w}  [y\otimes 1]_{s}= \sum_{z, t}  [{}^tz \otimes 1]_{t} w_{t} [{}^tz^{-1} \otimes 1]_{t_z} [y\otimes 1]_{s}
= \sum_{z, t}  \delta_{\partial(y^{-1})t_z , s} [{}^tz \otimes 1]_{t} w_{t} [{}^tz^{-1} y\otimes 1]_{s} \, .
$$
Since both $z$ and $x_y$ run over the set $\sX$, it remains to put $z=x_y$ and verify that the $\delta$-function in the final sum is non-zero if (and hence only if) $q=t$:
$$
\delta_{\partial(y^{-1})q_z , s}
=
\delta_{\partial(y^{-1})q \partial z,  q \partial x}
=
\delta_{q \partial(z) q^{-1} q h_y,  \partial (y) q \partial(x) q^{-1} q}
=
\delta_{\partial({}^qz\, {}^{q}h_y) q ,  \partial (y\,{}^qx) q}
= 1.
$$
This completes the proof.
\end{proof}

We can now complete the construction of $\Ind_{\cH}^{\cG}$.

\begin{Thm}
\phantomsection
\label{thm:IndFunctor}
\begin{enumerate}[label=(\roman*)]
\item The assignment $A \mapsto \Ind_{\cH}^{\cG} A$ extends to a pseudofunctor $\Ind_{\cH}^{\cG}:
\cH \mhyphen \AALG_{\Fld} \rightarrow \cG \mhyphen \AALG_{\Fld}$. 

\item If both indices $\vert G_1 : H_1\vert$ and $\vert G_2 : H_2\vert$ are finite, then $\Ind_{\cH}^{\cG}$ restricts to a pseudofunctor $\cH \mhyphen \AAlg_{\Fld} \rightarrow \cG \mhyphen \AAlg_{\Fld}$. If, moreover, $\vert G_2:H_2 \vert$ is not divisible by the characteristic of ${\Fld}$, then $\Ind_{\cH}^{\cG}$ restricts to a pseudofunctor $\cH \mhyphen \AAlg^{\fd}_{\Fld} \rightarrow \cG \mhyphen \AAlg^{\fd}_{\Fld}$.
\end{enumerate}
\end{Thm}

\begin{proof}
Let $M: A \rightarrow B$ be a $1$-morphism in $\cH \mhyphen \AALG_{\Fld}$. Define $\widetilde{M} =\Ind_{\cH}^{\cG} M$ to be
\[
\widetilde{M} = \prod_{t \in \sT} {\Fld}G_{2,t} \otimes_{{\Fld}H_2} M,
\]
where the tensor relations read
\[
[x({^{t}}z) \otimes m]_t = [x \otimes \omega_M(\partial z)(m)]_t, \qquad x \in G_2,\, z \in H_2, \, m \in M.
\]
The left $\widetilde{B} \mhyphen \widetilde{A}$-bimodule structure of $\widetilde{M}$ is defined by
\[
[x_2 \otimes b]_{t_2} \cdot [x_1 \otimes m]_{t_1}
=
\delta_{\partial (x_1^{-1}) t_2, t_1}
[x_2 x_1 \otimes \omega_{B,1}(  t_1^{-1} \partial ( x_1^{-1}) t_2 ) (b) m]_{t_1}
\]
and
\[
[x_2 \otimes m]_{t_2} \cdot [x_1 \otimes a]_{t_1}
=
\delta_{\partial (x_1^{-1}) t_2, t_1}
[x_2 x_1 \otimes \omega_{M}(  t_1^{-1} \partial ( x_1^{-1}) t_2 ) (m) a]_{t_1}.
\]
The structure maps for $\widetilde{M}$ are defined by
\[
\omega_{\tilde{M}}(g) ([x \otimes m]_t)_{t \in \sT} = ([{^{g}}x \otimes \omega_{M}(h)(m)]_{t^{\prime}})_{t^{\prime} \in \sT},
\]
where $g t = t^{\prime} h$ for $h \in H_1$. To verify equation \eqref{eq:Real1MorAx3}, we note that 
\[
\omega_{\tilde{M}}(x_2) ([x_1 \otimes m]_t) = [{^{\partial(x_2)}}x_1 \otimes \omega_{M}(h)(m)]_{t^{\prime}},
\]
where $\partial(x_2) t = t^{\prime} h$, while
\[
\omega_{\tilde{B},2}(x_2) \cdot [x_1 \otimes m]_t \cdot \omega_{\tilde{A},2}(x_2)^{-1} 
=
[x_2 x_1 \otimes m]_t \cdot \omega_{\tilde{A},2}(x_2)^{-1}
=
[x_2 x_1 x_2^{-1} \otimes \omega_{M}(t^{\prime ^{-1}} \partial(x_2) t)(m)]_t.
\]

Given a $2$-morphism $\phi: M \Rightarrow N$ in $\cH \mhyphen \AALG_{\Fld}$, define $\tilde{\phi} = \Ind_{\cH}^{\cG}\phi$ by
\[
\tilde{\phi}([x \otimes m]_t) = [x \otimes \phi(m)]_t. 
\]
This is left $\widetilde{B}$-linear because
\begin{eqnarray*}
\tilde{\phi}([x_2 \otimes b]_{t_2} \cdot [x_1 \otimes m]_{t_1})
&=&
\delta_{\partial (x_1^{-1}) t_2, t_1}
[x_2 x_1 \otimes \phi(\omega_{B,1}(  t_1^{-1} \partial ( x_1^{-1}) t_2 ) (b) m)]_{t_1} \\
&=&
\delta_{\partial (x_1^{-1}) t_2, t_1}
[x_2 x_1 \otimes \omega_{B,1}(  t_1^{-1} \partial ( x_1^{-1}) t_2 ) (b) \phi(m)]_{t_1},
\end{eqnarray*}
which is clearly equal to $[x_2 \otimes b]_{t_2} \cdot [x_1 \otimes \phi(m)]_{t_1}$.
Similarly, it is right $\widetilde{A}$-linear. 
The $H_1$-equivariance of $\phi$ implies the $G_1$-equivariance of $\tilde{\phi}$.

The $2$-isomorphisms relating compositions of $1$-morphisms are induced by those of $\AALG_{\Fld}$. It follows directly from the definitions that $\Ind_{\cH}^{\cG}$ strictly preserves identity $1$-morphisms. This completes the construction of $\Ind_{\cH}^{\cG}$.

The second statement now follows from Proposition~\ref{thm:maschke} and the observation that $\widetilde{M}$ is finite dimensional if $M$ is so.
\end{proof}

\subsection{Induction as a biadjunction}
\label{sec:indBiadjunct}

In this section, we prove the biadjointness of $\Res_{\cH}^{\cG}$ and $\Ind_{\cH}^{\cG}$ in two, in a sense, opposite situations. For the notion of a (left or right) biadjunction between pseudofunctors, see \cite[Definition 2.1]{gurski2012}.

\begin{Prop}
\label{thm:indBiadjunct}
If either $H_2=G_2$ or $H_1=G_1$, then $\Ind_{\cH}^{\cG}$ is right biadjoint to $\Res_{\cH}^{\cG}$. This statement holds for pseudofunctors between $? \mhyphen \AALG_{\Fld}$ or, with the assumptions of Theorem \ref{thm:IndFunctor}(ii), between $? \mhyphen \AAlg_{\Fld}$ or $? \mhyphen \AAlg^{\fd}_{\Fld}$.
\end{Prop}

\begin{proof}
Suppose that $H_2 = G_2$. To begin, we need to define pseudonatural transformations
\[
\epsilon: \Res_{\cH}^{\cG} \circ \Ind_{\cH}^{\cG} \Longrightarrow \id_{\cH \mhyphen \AALG_{\Fld}},
\qquad
\eta: \id_{\cG \mhyphen \AALG_{\Fld}} \Longrightarrow \Ind_{\cH}^{\cG} \circ \Res_{\cH}^{\cG}.
\]
Write $t_e \in \sT$ for the representative of the identity coset. Given an $\cH$-algebra $A$, we claim that the ${\Fld}$-linear map
\[
\epsilon_A: \Res_{\cH}^{\cG} \Ind_{\cH}^{\cG} A \rightarrow A , \qquad [a]_t \mapsto \delta_{t,t_e} \omega_{A,1}(t)(a)
\]
is a strict $\cH$-algebra morphism. Firstly, $\epsilon_A$ is a ${\Fld}$-algebra morphism:
\[
\epsilon_A([a_2]_{t_2} \cdot [a_1]_{t_1})
=
\delta_{t_2,t_1} \epsilon_A([a_2 a_1]_{t_1}) =
\delta_{t_2,t_1} \delta_{t_1,t_e} \omega_{A,1}(t_1)(a_2 a_1)
\]
is equal to
\[
\epsilon_A([a_2]_{t_2}) \cdot \epsilon_A([a_1]_{t_1})
=
\delta_{t_2,t_e} \delta_{t_1,t_e} \omega_{A,1}(t_2)(a_2) \omega_{A,1}(t_1)(a_1).
\]
Clearly $\epsilon_A$ preserves multiplicative identities. Moreover, $\epsilon_A$ is $H_1$-equivariant:
\[
\epsilon_A \left( \omega_{\tilde{A},1}(h)([a]_t) \right)
=
\epsilon_A \left( [\omega_{A,1}(h^{\prime})(a)]_{t^{\prime}} \right)
=
\delta_{t^{\prime},t_e} \omega_{A,1}(t^{\prime} h^{\prime})(a),
\]
where $h t = t^{\prime}h^{\prime}$, is equal to
\begin{eqnarray*}
\omega_{A,1}(h)(\epsilon_A([a]_t))
&=&
\delta_{t,t_e}
\omega_{A,1}(h t_e)(a).
\end{eqnarray*}
Finally, $\epsilon_A$ is $\omega_{?,2}$-compatible: for each $z \in H_2$, we have
\[
\epsilon_A (\omega_{\tilde{A},2}(z))
=
\epsilon_A (( [\omega_{A,2}({^{t^{-1}}}z)]_t)_{t \in \mathcal{T}})
=
\omega_{A,1}(t_e)(\omega_{A,2}({^{t_e^{-1}}}z))
=
\omega_{A,2}(z).
\]
We henceforth interpret $\epsilon_A$ as a representable $1$-morphism $\Res_{\cH}^{\cG} \Ind_{\cH}^{\cG} A \rightarrow A$ in $\cH \mhyphen \AALG_{\Fld}$. Given a $1$-morphism $M: A \rightarrow B$ in $\cH \mhyphen \AALG_{\Fld}$, define
\[
\epsilon_M : B \otimes_{\tilde{B}} \widetilde{M} \Longrightarrow M \otimes_A A_{\epsilon_A} \simeq M_{\epsilon_A}, \qquad b \otimes [m]_t \mapsto \delta_{t,t_e} b \omega_M(t)(m).
\]
Calculations similar to those above show that $\epsilon_M$ is indeed a $2$-morphism in $\cH \mhyphen \AALG_{\Fld}$ and that $\{\epsilon_A\}_A$ and $\{\epsilon_M\}_M$ satisfy the coherence conditions required to define the pseudonatural transformation $\epsilon$.

Similarly, given a $\cG$-algebra $B$, we claim that the ${\Fld}$-linear map
\[
\eta_B: B \rightarrow \Ind_{\cH}^{\cG} \Res_{\cH}^{\cG} B, \qquad b \mapsto ( [\omega_{B,1}(t^{-1})(b)]_t )_{t \in \sT}
\]
is a strict $\cG$-algebra morphism. First, $\eta_B$ is multiplicative:
\begin{eqnarray*}
\eta_B( b_2 ) \cdot \eta_B(b_1)
&=&
( [\omega_{B,1}(t_2^{-1})(b_2)]_{t_2} )_{t_2 \in \sT} \cdot ( [\omega_{B,1}(t_1^{-1})(b_1)]_{t_1} )_{t_1 \in \sT} \\
&=&
(\delta_{t_2,t_1} [\omega_{B,1}(t_1^{-1} t_2)(\omega_{B,1}(t_2^{-1})(b_2))\omega_{B,1}(t_1^{-1})(b_1)]_{t_1} )_{t_1 \in \sT} \\
&=&
([\omega_{B,1}(t^{-1})(b_2 b_1)]_t )_{t \in \sT}
=
\eta_B( b_2 b_1).
\end{eqnarray*}
Clearly, $\eta_B$ is unital. Moreover, $\eta_B$ is $G_1$-equivariant:
\begin{eqnarray*}
\eta_B(\omega_{B,1}(g)(b))
&=&
( [\omega_{B,1}(t^{-1} g)(b)]_t )_{t \in \sT}
\end{eqnarray*}
while
\[
\omega_{\tilde{B},1}(g) (\eta_B(b))
=
\omega_{\tilde{B},1}(g) ( [\omega_{B,1}(t^{-1})(b)]_t )_{t \in \sT}
=
( [\omega_{B,1}(h t^{-1})(b)]_{t^{\prime}} )_{t^{\prime} \in \sT}
\]
where $g t = t^{\prime} h$. Finally, $\eta_B$ is $\omega_{?,2}$-compatible: for $x \in G_2$, we have
\[
\eta_B (\omega_{B,2}(x))
=
( [\omega_{B,1}(t^{-1})(\omega_{B,2}(x))]_t )_{t \in \sT}
=
( [\omega_{B,2}({^{t^{-1}}}x)]_t )_{t \in \sT} = \omega_{\tilde{B},2}(x).
\]
We interpret $\eta_B$ as a $1$-morphism $B \rightarrow \Ind_{\cH}^{\cG} \Res_{\cH}^{\cG} B$ in $\cG \mhyphen \AALG_{\Fld}$. Given a $1$-morphism $N: A \rightarrow B$ in $\cG \mhyphen \AALG_{\Fld}$, define the required $2$-morphism by
\[
\eta_B: \widetilde{B}_{\eta_B} \otimes_B N \Longrightarrow \widetilde{N} \otimes_{\tilde{A}} \widetilde{A}_{\eta_A} \simeq \widetilde{N}_{\eta_A},
\qquad
[b]_t \otimes n \mapsto [b \omega_{N}(t^{-1})(n)]_t .
\]
This data defines the pseudonatural transformation $\eta$.

It remains to define invertible zig-zag modifications
\[
\Gamma: (\Ind_{\cH}^{\cG} \diamond \epsilon) \circ (\eta \diamond \Ind_{\cH}^{\cG}) \xRrightarrow{\;\;\;\;\;\;} 1_{\Ind},
\qquad
\Lambda : 1_{\Res} \xRrightarrow{\;\;\;\;\;\;} (\epsilon \diamond \Res_{\cH}^{\cG}) \circ (\Res_{\cH}^{\cG} \diamond \eta).
\]
Hence, for each $A \in \cH \mhyphen \AALG_{\Fld}$ we need to define a $2$-isomorphism
\[
\Gamma_A: \Ind_{\cH}^{\cG} \epsilon_A \otimes_{\Ind \Res \Ind A} \eta_{\Ind A} \Longrightarrow \Ind_{\cH}^{\cG} A.
\]
Since $\iota\coloneqq {\Ind \epsilon_A \circ \eta_{\Ind A}}$ is the identity map, the domain bimodule $(\Ind_{\cH}^{\cG} A)_{\iota}$ of $\Gamma_A$ is the identity bimodule. After making this identification, we take $\Gamma_A$ to be the identity map. For $B \in \cG \mhyphen \AALG_{\Fld}$, we define a $2$-isomorphism
\[
\Lambda_B: \Res_{\cH}^{\cG} B \Longrightarrow \epsilon_{\Res B} \otimes_{\Res \Ind \Res B} \Res_{\cH}^{\cG} \eta_B.
\]
The codomain bimodule is isomorphic to $(\Res_{\cH}^{\cG} B)_{\epsilon_{\Res B} \circ \Res \eta_B} \simeq \Res_{\cH}^{\cG} B$, and we use this isomorphism to define $\Lambda_B$. It is straightforward to verify that the above data satisfies the required coherence conditions, proving the proposition in the case $H_2=G_2$.

The case in which $H_1 = G_1$ is similar, so we are brief. Define an algebra homomorphism
\[
\epsilon^{\prime}_A: A \rightarrow \Res_{\cH}^{\cG} \Ind_{\cH}^{\cG}  A, \qquad a \mapsto e_{G_2} \otimes a.
\]
The $G_1$-equivariance of $\epsilon^{\prime}_A$ is clear and it is $\omega_{?,2}$-compatible because
\[
\epsilon^{\prime}_A(\omega_{A,2}(z))
=
e_{G_2} \otimes \epsilon^{\prime}_A(\omega_{A,2}(z))
=z \otimes 1_A.
\]
We henceforth interpret $\epsilon_A^{\prime}$ as the $1$-morphism
\[
\epsilon_A = {_{\epsilon^{\prime}_A}}\widetilde{A}: \Res_{\cH}^{\cG} \Ind_{\cH}^{\cG}  A \rightarrow A
\]
in $\cH \mhyphen \AALG_{\Fld}$. Given a $1$-morphism $M: A \rightarrow B$ in $\cH \mhyphen \AALG_{\Fld}$, define a $2$-morphism
\[
\epsilon_M : {_{\epsilon^{\prime}_B}}\widetilde{B} \otimes_{\tilde{B}} \widetilde{M} \simeq {_{\epsilon^{\prime}_B}}\widetilde{M} \Longrightarrow M \otimes_A ({_{\epsilon^{\prime}_A}} \widetilde{A}),
 \qquad x \otimes m \mapsto m \otimes (x \otimes 1_A).
\]
This defines the pseudonatural transformation $\epsilon$.

Let 
\[
\eta^{\prime}_B: \Ind_{\cH}^{\cG} \Res_{\cH}^{\cG} B \rightarrow B , \qquad x \otimes b \mapsto \omega_{B,2}(x) b.
\]
This map is well-defined because
\[
x \cdot z \otimes b
\mapsto
\omega_{B,2}(xz)b
=
\omega_{B,2}(x) \omega_{B,2}(z)b
\]
and $x  \otimes \omega_{B,2}(z) b \mapsto \omega_{B,2}(x) \omega_{B,2}(z)b$. Moreover, $\eta^{\prime}_B$ is clearly a $G_1$-equivariant unital ${\Fld}$-algebra homomorphism and is $\omega_{?,2}$-compatible:
\[
\eta^{\prime}_B(\omega_{\tilde{B},2}(x))
=
\eta^{\prime}_B(x \otimes 1_B)
=
\eta^{\prime}_{B,2}(x).
\]
We interpret $\eta_B^{\prime}$ as the $1$-morphism
\[
\eta_B={_{\eta^{\prime}_B}}B: B \rightarrow \Ind_{\cH}^{\cG} \Res_{\cH}^{\cG} B
\]
in $\cG \mhyphen \AALG_{\Fld}$. Given a $1$-morphism $N: A \rightarrow B$ in $\cG \mhyphen \AALG_{\Fld}$, define a $2$-morphism
\[
\eta_N : {_{\eta^{\prime}_B}} B \otimes_B N \simeq {_{\eta^{\prime}_B}} N  \Longrightarrow \widetilde{N} \otimes_{\tilde{A}} ({_{\eta^{\prime}_A}} A),
\qquad
n \mapsto (e_{G_2} \otimes n) \otimes 1_A.
\]

Similar to the case $H_2=G_2$, after suitable identifications, we can take the $2$-isomorphisms $\Gamma_A$ and $\Lambda_B$ to be the respective identities.
\end{proof}

\begin{Rem}
\begin{enumerate}[label=(\roman*)]
\item In the setting of Proposition~\ref{thm:indBiadjunct}, one can also prove that $\Ind_{\cH}^{\cG}$ is left biadjoint to $\Res_{\cH}^{\cG}$. When $H_2=G_2$, for example, this is done by interpreting $\epsilon_A$ as a $1$-morphism $A \rightarrow \Res_{\cH}^{\cG} \Ind_{\cH}^{\cG} A$ in $\cH \mhyphen \AALG_{\Fld}$, and similarly for $\eta_B$. This is analogous to what was done in the proof of Proposition~\ref{thm:indBiadjunct} in the case $H_1=G_1$.

\item We expect that $\Ind_{\cH}^{\cG}$ is in fact left and right biadjoint to $\Res_{\cH}^{\cG}$ without the assumption $H_1=G_1$ or $H_2 = G_2$. This would categorify the left and right adjunctions between induction and restriction in the representation theory of finite groups.
\end{enumerate}
\end{Rem}

\subsection{The Real case}
\label{sec:indRealKAlg}

We extend the constructions of Sections \ref{sec:indKAlg} and \ref{sec:maschke} to the $\ZTwo$-graded setting. Since the calculations are similar, we are occasionally brief.

Crossed submodules $\cH$ of a $\ZTwo$-graded crossed module $\cG$ come in two flavours:
\begin{enumerate}[label=(\roman*)]
\item non-trivially graded: $H_1$ is a non-trivially $\ZTwo$-graded subgroup of $G_1$,

\item trivially graded: $H_1$ is a trivially $\ZTwo$-graded subgroup of $G_1$.
\end{enumerate}
There is a restriction pseudofunctor $\Res_{\cH}^{\cG} : \cG \mhyphen \AAlg^{\fd}_{\Fld} \rightarrow \cH \mhyphen \AAlg^{\fd}_{\Fld}$. We define a pseudofunctor $\Ind_{\cH}^{\cG} : \cH \mhyphen \AAlg^{\fd}_{\Fld} \rightarrow \cG \mhyphen \AAlg^{\fd}_{\Fld}$. To avoid confusion, we sometimes denote $\Ind_{\cH}^{\cG}$ by $\RInd_{\cH}^{\cG} A$ and $\HInd_{\cH}^{\cG} A$ in the case of non-trivially and trivially graded $\cH$, respectively, matching the notation for the Real and hyperbolic induction of \cite[\S 7]{mbyoung2018c}. We continue to use the  notation of Section \ref{sec:indKAlg}.

Let $A$ be an $\cH$-algebra. Define
\[
\widetilde{A} = \prod_{t \in \sT} {\Fld}G_{2,t} \otimes_{{\Fld}H_2} {^{\pi(t)}}A,
\]
where ${^{\pi(t)}}A$ is $A$ or $A^{\vee}$, depending on $\pi(t) \in \ZTwo$. The ${\Fld}$-algebra structure of $\widetilde{A}$ is again defined by equation \eqref{eq:indMult}, keeping in mind that we use the multiplication of ${^{\pi(t)}}A$ in the $t$\textsuperscript{th} factor. Define $\omega_{\tilde{A},1}$ and $\omega_{\tilde{A},2}$ by
\[
\omega_{\tilde{A},1}(g)( [x \otimes a]_t )
=
[{^{g}}x \otimes {^{\pi(t^{\prime})}}\omega_{A,1}(h)(a)]_{t^{\prime}},
\qquad
\omega_{\tilde{A},2}(x) = ( \, [x \otimes 1_A]_t \, )_{t \in \sT},
\]
where $g t = t^{\prime} h$ for $t^{\prime} \in \sT$ and $h \in H_1$.

Let us make explicit the trivially graded case. Since the morphism $\cH \hookrightarrow \cG$ factors through $\cG_0 \hookrightarrow \cG$, it suffices to consider the case $\cH = \cG_0$. The general case can then obtained as the composition
\[
\HInd_{\cH}^{\cG} = \HInd_{\cG_0}^{\cG} \circ \Ind_{\cH}^{\cG_0},
\]
where $\Ind_{\cH}^{\cG_0}$ is as in Theorem \ref{thm:IndFunctor}. Fix an element $\spcl \in G_1 \setminus G_0$ and take $\sT = \{e, \spcl\}$. Then $\widetilde{A} = [A]_e \oplus [A^{\opp}]_{\spcl}$ with
\[
\omega_{\tilde{A},1}(g) ([a]_t)
=[\omega_{A,1}(g^{\prime})(a)]_{t^{\prime}},
\qquad
\omega_{\tilde{A},2}(x) = [\omega_{A,2}(x)]_e + [\omega_{A,2}({^{\spcl^{-1}}}x)^{-1}]_{\spcl}, 
\]
where $g t = t^{\prime} g^{\prime}$ with $g^{\prime} \in G_0$. For example, $g \in G_0$ acts on $[A^{\opp}]_{\spcl}$ by $\spcl^{-1} g \spcl$.

\begin{Thm}
\label{thm:2IndReal}
The above constructions define a $\cG$-algebra $\widetilde{A}$.
\end{Thm}

\begin{proof}
  The proof in the non-trivially graded case is similar to that of
  Proposition~\ref{thm:2IndCplx}, 
    so we focus on the trivially graded case. Let us check that $\omega_{\tilde{A},1}$ is a generalized algebra automorphism. When $\pi(g)=-1$, for example, we have
\begin{eqnarray*}
\omega_{\tilde{A},1}(g) \left( [a_2]_e \cdot [a_1]_e \right)
&=&
\omega_{\tilde{A},1}(g) ([a_2 a_1]_e) \\
&=&
[\omega_{A,1}(\spcl^{-1} g)(a_1)]_{\spcl} \bullet [\omega_{A,1}(\spcl^{-1} g)(a_2)]_{\spcl} \\
&=&
\omega_{A,1}(g) ([a_1]_e) \bullet \omega_{A,1}(g) ([a_2]_e),
\end{eqnarray*}
where $\bullet$ indicates multiplication in $A^{\opp}$, and
\[
\omega_{\tilde{A},1}(g) \left( [a_2]_{\spcl} \cdot [a_1]_{\spcl} \right)
=
[\omega_{A,1}(g \spcl)(a_1 a_2)]_e
=
\omega_{\tilde{A},1}(g)([a_1]_{\spcl}) \bullet \omega_{\tilde{A},1}(g)([a_2]_{\spcl}).
\]
We omit the proof that $\omega_{\tilde{A},1}$ is a group homomorphism. To see that $\omega_{\tilde{A},2}$ is a group homomorphism, we compute
\begin{eqnarray*}
\omega_{\tilde{A},2}(x_2 x_1)
&=&
[\omega_{A,2}(x_2 x_1)]_e + [\omega_{A,2}({^{\spcl^{-1}}}(x_2 x_1))^{-1}]_{\spcl} \\
&=&
[\omega_{A,2}(x_2) \omega_{A,2}(x_1)]_e + [\omega_{A,2}({^{\spcl^{-1}}}x_1)^{-1} \omega_{A,2}( {^{\spcl^{-1}}}x_2)^{-1}]_{\spcl}
\end{eqnarray*}
while
\begin{eqnarray*}
\omega_{\tilde{A},2}(x_2) \omega_{\tilde{A},2}(x_1)
&=&
[\omega_{A,2}(x_2)]_e \cdot [\omega_{A,2}(x_1)]_e + [\omega_{A,2}({^{\spcl^{-1}}}x_2)^{-1}]_{\spcl} \cdot [\omega_{A,2}({^{\spcl^{-1}}}x_1)^{-1}]_{\spcl} \\
&=& 
[\omega_{A,2}(x_2) \omega_{A,2}(x_1)]_e + [\omega_{A,2}({^{\spcl^{-1}}}x_1)^{-1} \omega_{A,2}({^{\spcl^{-1}}}x_2)^{-1}]_{\spcl},
\end{eqnarray*}
as required. To verify equation \eqref{eq:NAlgAx1}, fix $x \in G_2$ and compute
\[
\omega_{\tilde{A},1}(\partial x) ([a]_{\spcl})
=
[a]_{\partial(x) \spcl}
=
[a]_{\spcl \partial({^{\spcl^{-1}}}x)}
=
[\omega_{A,2}({^{\spcl^{-1}}}x)^{-1} \bullet a \bullet \omega_{A,2}({^{\spcl^{-1}}}x)]_{\spcl},
\]
which is equal to $\partial(\omega_{\tilde{A},2}(x))([a]_{\spcl})$. The computation with $e$ in place of $\spcl$ is similar. Turning to equation \eqref{eq:NAlgAx2}, consider, for example, the case $\pi(g)=-1$. We have
\begin{eqnarray*}
\omega_{\tilde{A},1}(g)(\omega_{\tilde{A},2}(x)^{-1})
&=&
\omega_{\tilde{A},1}(g)([\omega_{A,2}(x)^{-1}]_e) + \omega_{\tilde{A},1}(g)([\omega_{A,2}({^{\spcl^{-1}}}x)]_{\spcl}) \\
&=&
[\omega_{A,1}(\spcl^{-1} g)(\omega_{A,2}(x)^{-1})]_{\spcl} + [\omega_{A,1}(g \spcl)(\omega_{A,2}({^{\spcl^{-1}}}x))]_e \\
&=&
[\omega_{A,2}({^{\spcl^{-1} g}}x)^{-1}]_{\spcl} + [\omega_{A,2}({^{g}}x)]_e,
\end{eqnarray*}
which is plainly equal to $\omega_{\tilde{A},2}({^{g}}x)$.
\end{proof}

There is an extension of Proposition~\ref{thm:indBiadjunct} to the $\ZTwo$-graded setting. We briefly indicate the construction in the hyperbolic setting; the construction for $\RInd_{\cH}^{\cG}$ is similar. More precisely, the assignment $A \mapsto \HInd_{\cG_0}^{\cG}$ extends to a pseudofunctor
\[
\HInd_{\cG_0}^{\cG}: \cG_0 \mhyphen \AAlg_{\Fld}^{\fd, 2 \mhyphen \simeq} \rightarrow \cG \mhyphen \AAlg^{\fd}_{\Fld}
\]
with domain the maximal locally groupoidal subbicategory of $\cG_0 \mhyphen \AAlg^{\fd}_{\Fld}$. Given a $1$-morphism $M : A \rightarrow B$ in $\cG_0 \mhyphen \AAlg^{\fd}_{\Fld}$, define $\widetilde{M}=\HInd_{\cG_0}^{\cG} M$ as follows. As a $\widetilde{B} \mhyphen \widetilde{A}$-bimodule, $\widetilde{M}$ is simply $M \oplus M^{\circ}$. The structure maps of $\widetilde{M}$ are defined by
\[
\omega_{\tilde{M}}(g)
=
\left( \begin{matrix}
\omega_M(g) & 0 \\
0 & \omega_{M^{\circ}}(\spcl^{-1} g \spcl)
\end{matrix} \right),
\qquad
\omega_{\tilde{M}}(f)
=
\left( \begin{matrix}
0 & \omega_{M^{\circ}}(f \spcl) \\
\ev_M \circ \omega_M(\spcl^{-1} f) & 0
\end{matrix} \right)
\]
where $g \in G_0$ and $f \in G_1 \setminus G_0$ and
\[
\omega_{M^{\circ}}(\spcl^{-1} g \spcl)({\mathfrak{m}})(m) = \omega_{A,1}(\spcl^{-1} g \spcl)\left[{\mathfrak{m}}(\omega_M(\spcl^{-1} g \spcl)^{-1}(m))
\right],
\qquad {\mathfrak{m}} \in M^{\circ}, \, m \in M.
\]
Given a $2$-isomorphism $\phi: M \Rightarrow N$ in $\cG_0 \mhyphen \AAlg^{\fd}_{\Fld}$, the $G_1$-equivariant intertwiner $\HInd_{\cG_0}^{\cG} \phi$ is defined to be $\phi \oplus \phi^{-\circ}$, where
\[
\phi^{-\circ}({\mathfrak{m}})(n) = {\mathfrak{m}}(\phi^{-1}(n)), \qquad {\mathfrak{m}} \in M^{\circ}, \, n \in N.
\]

The $\ZTwo$-graded analogue of Proposition \ref{thm:indBiadjunct} is more subtle.

\begin{Prob}
  \label{prob:adjunction}
  Investigate the adjunction properties of  $\Ind_{\cH}^{\cG}$ and $\Res_{\cH}^{\cG}$ when $\cG$ is non-trivially $\ZTwo$-graded.
\end{Prob}

The difficulty stems from the fact that $\RInd_{\cH}^{\cG}$, for example, is defined only on $2$-isomorphisms of $\cH \mhyphen \AAlg^{\fd}_{\Fld}$, while the $2$-morphism components of $\epsilon$ (as in the proof of Proposition~\ref{thm:indBiadjunct}) are not $2$-isomorphisms. For this reason, we expect the biadjointness properties to be more naturally formulated in the anti-linear approach to Real $2$-representations.


It is useful to decompose the assignment $A \mapsto \HInd_{\cG_0}^{\cG} A$ into two steps. If $\cG$ is a $\ZTwo$-graded crossed module, $\spcl \in G_1 \setminus G_0$ and $A$ is a $\cG_0$-algebra, denote by $\spcl \cdot A$ the $\cG_0$-algebra with underlying ${\Fld}$-algebra $A$ and structure maps
\[
\omega_{\spcl \cdot A,1} (g) = \omega_{A,1}(\spcl^{-1} g \spcl),
\qquad
\omega_{\spcl \cdot A,2}(x) = \omega_{A,2}({^{\spcl^{-1}}}x).
\]
It then follows immediately from the definitions that there is a $\cG$-algebra isomorphism
\[
\HInd_{\cG_0}^{\cG} A \simeq A \boxplus \spcl \cdot A^{\vee},
\]
where an element $f \in G_1 \setminus G_0$ acts on $A \boxplus \spcl \cdot A^{\vee}$ by the matrix 
\[
\left( \begin{matrix}
0 & \omega_{A,1}(f \spcl) \\
\omega_{A,1}(\spcl^{-1} f)^{\opp} & 0
\end{matrix} \right).
\]
This isomorphism generalizes as follows.

\begin{Lem}
\label{lem:HIndSymmetry}
Let $\cG$ be a $\ZTwo$-graded crossed module with trivially graded crossed submodule $\cH$. For each $\cH$-algebra $A$ and $\spcl \in H_1 \setminus H_0$, there is a $\cG$-algebra isomorphism
\[
\HInd_{\cH}^{\cG} A \simeq \HInd_{\spcl \cdot \cH \cdot \spcl^{-1}}^{\cG} \spcl \cdot A^{\vee}.
\]
\end{Lem}

While the pseudofunctor $\HInd_{\cG_0}^{\cG}$ is not monoidal, it does admit a natural enhancement to a $\cG \mhyphen \AAlg_{\Fld}^{\fd}$-module pseudofunctor. We do not use the full strength of this statement, only that for each $\cG_0$-algebra $A$ and $\cG$-algebra $B$, there is a $\cG$-algebra isomorphism
\begin{equation}
\label{eq:hypModFunctor}
B \boxtimes \HInd_{\cG_0}^{\cG} A
\simeq
\HInd_{\cG_0}^{\cG} 
\left( \Res_{\cG_0}^{\cG} B \boxtimes A \right).
\end{equation}
Indeed, from the perspective described above Lemma \ref{lem:HIndSymmetry}, the left and right hand sides of the desired isomorphism \eqref{eq:hypModFunctor} are represented by
\[
\left(
(B \boxtimes A) \boxplus (B \boxtimes \spcl \cdot A^{\vee}),
\left( \begin{matrix}
0 & \omega_{B,1}(\spcl) \otimes \omega_{A,1}(\spcl^2) \\
\omega_{B,1}(\spcl) \otimes \id_A & 0
\end{matrix} \right)
\right)
\]
and
\[
\left(
(B \boxtimes A) \boxplus (\spcl \cdot B^{\vee} \boxtimes \spcl \cdot A^{\vee}),
\left( \begin{matrix}
0 & \omega_{B,1}(\spcl^2) \otimes \omega_{A,1}(\spcl^2) \\
\id_{B \otimes A} & 0
\end{matrix} \right)
\right),
\]
respectively, where we have displayed the matrices giving the action of $\spcl$ in each case (which, together with the underlying $\cG_0$-algebra structure, determines the $\cG$-algebra structure). These pairs are equivalent via the map $\left( \begin{smallmatrix}
\id_{B \otimes A} & 0 \\
0 & \omega_{B,1}(\spcl) \otimes 1_A
\end{smallmatrix} \right)$.


\section{The classification of Real \texorpdfstring{$2$}{}-representations on \texorpdfstring{$2\Vect_{\Fld}$}{}}
\label{sec:classReal2Reps}

We apply the results of the previous chapters to Real $2$-representation theory.

\subsection{From N-weak \texorpdfstring{$\cG$}{}-algebras to Real \texorpdfstring{$2$}{}-modules}
\label{subsec:FromTo}
We begin by connecting N-weak algebras to $2$-representation theory. The proof of the following result can be seen as justifying (or even deriving) the definition of an N-weak algebra.

\begin{Prop}
\label{prop:weakAlgTo2Mod}
Let $\cG$ be a $\ZTwo$-graded crossed module. Every separable N-weak $\cG$-algebra $\omega_A: \cG \rightarrow \AUT^{\gen}(A)$ induces a Real $2$-module
$\mbox{\gls{glTh}}: \tcG \rightarrow \AAlg_{\Fld}^{\fd}$.
\end{Prop}

\begin{proof}
The proof is similar to that of \cite[Proposition 2.3]{rumynin2018}, which treats the ordinary RW-weak case. The present setting is complicated by the fact that we work in the non-abelian group $A^{\times}$, instead of $Z(A)^{\times}$.
 
Set $\Theta_A(\star) = A$. For a $1$-morphism $g: \star \rightarrow \star$, let $\Theta_A(g)$ be the $A \mhyphen {^{\pi(g)}}A$-bimodule $A_{\omega_1(g)}$. For a $2$-morphism $x: g \Rightarrow \partial(x) g$, set (in the notation of Lemma \ref{lem:natTransBij})
\[
\Theta_A (g,x) = \Upsilon_{\omega_1(g)}^{\omega_1(\partial(x) g)}
\left(
(\omega_3(\partial x, g) \omega_2(x))^{-1}
\right).
\]
Equations \eqref{eq:omega1Proj} and \eqref{eq:NAlgAx1} imply that $\Theta_A (g,x)$ is well-defined. For $g_i \in G_1$, define
\[
\Theta_A(g_2,g_1) : \Theta_A (g_2) \diamond \Theta_A (g_1) \Longrightarrow \Theta_A(g_2 g_1)
\]
to be $\Upsilon_{\omega_1(g_2) \circ \omega_1(g_1)}^{\omega_1(g_2g_1)}(\omega_3(g_2,g_1)^{-1})$. Equation \eqref{eq:omega1Proj} implies that $\Theta_A(g_2,g_1) $ is well-defined.

To see that $\Theta_A$ is compatible with the vertical composition of $2$-morphisms, note that $\Theta_A(\partial(x_1) g_1,x_2) \circ \Theta_A(g_1,x_1)$ is equal to $\Upsilon_{\omega_1(g_1)}^{\omega_1(\partial(x_2 x_1) g_1)}(\clubsuit^{-1})$, where
\begin{eqnarray*}
\clubsuit
&=&
\omega_3(\partial x_2,\partial(x_1) g_1) \underline{\omega_2(x_2)\omega_3(\partial x_1, g_1)} \omega_2(x_1) \\
&\overset{\eqref{eq:NAlgAx1}}{=}&
\underline{\omega_3(\partial x_2 ,\partial(x_1) g_1) \cdot {^{\partial x_2}}\omega_3(\partial x_1, g_1)} \omega_2(x_2) \omega_2(x_1) \\
&\overset{\eqref{eq:omega3Cocycle}}{=}&
\omega_3(\partial(x_2 x_1), g_1) \underline{\omega_3(\partial x_2,\partial x_1) \omega_2(x_2) \omega_2(x_1)}\\
&\overset{\eqref{eq:omega2Proj}}{=}&
\omega_3(\partial(x_2 x_1), g_1) \omega_2(x_2 x_1).
\end{eqnarray*}
By abuse of notation, we write ${^{\partial x_2 }}\omega_3(\partial x_1, g_1)$ in place of ${^{\omega_1(\partial x_2)}}\omega_3(\partial x_1, g_1)$. At each stage of the calculation where we specify $\stackrel{(N)}{=}$, this indicates that equation~($N$) is applied to the underlined part of the previous expression. It follows that $\Theta_A(\partial(x_1) g_1,x_2) \circ \Theta_A(g_1,x_1) = \Theta_A(g_1,x_2 x_1)$.

Consider next the horizontal multiplicativity of $\Theta_A$. Let $g_i \in G_1$ and $x_i \in G_2$. Then\footnote{For notational simplicity, we omit the sub/superscripts of $\Upsilon$ for the remainder of the proof.} $\Theta_A(g_2 g_1, x_2 \cdot {^{g_2}}x_1) \circ \Theta_A(g_2,g_1) = 
\Upsilon (\spadesuit^{-1})$, where
\begin{eqnarray*}
\spadesuit
&=&
\omega_3(\partial(x_2 \cdot {^{g_2}}x_1), g_2 g_1) \underline{\omega_2(x_2 \cdot {^{g_2}}x_1)} \omega_3(g_2,g_1) \\
&\overset{\eqref{eq:omega2Proj}}{=}&
\omega_3(\partial(x_2 \cdot {^{g_2}}x_1), g_2 g_1) \omega_3(\partial x_2, \partial({^{g_2}}x_1)) \omega_2(x_2) \omega_2({^{g_2}}x_1) \omega_3(g_2,g_1).
\end{eqnarray*}
Two applications of the cocycle condition \eqref{eq:omega3Cocycle} give
\begin{multline*}
\omega_3(\partial(x_2) g_2 \partial(x_1) g_2^{-1}, g_2 g_1) \omega_3(\partial x_2, g_2 \partial(x_1) g_2^{-1})
= \\
\omega_3(\partial x_2, g_2 \partial(x_1) g_1) \cdot {^{\partial x_2}}  \left[ \omega_3(g_2 \partial(x_1),g_1)\omega_3(g_2 \partial(x_1) g_2^{-1},g_2) \cdot {^{\partial({^{g_2}} x_1)}}\omega_3(g_2,g_1)^{-1} \right].
\end{multline*}
Equation \eqref{eq:NAlgAx1} implies that
\[
{^{\partial x_2}}\left[{^{\partial({^{g_2}} x_1)}}\omega_3(g_2,g_1)^{-1} \right]
\omega_2(x_2) \omega_2({^{g_2}} x_1) \omega_3(g_2,g_1)
=
\omega_2(x_2) \omega_2({^{g_2}} x_1)
\]
so that
\begin{eqnarray*}
\spadesuit
&=&
\omega_3(\partial x_2, g_2 \partial(x_1) g_1) \cdot \underline{{^{\partial x_2 }} \left[ \omega_3(g_2 \partial x_1, g_1) \omega_3(g_2 \partial(x_1) g_2^{-1},g_2) \right] \omega_2(x_2)} \omega_2({^{g_2}} x_1) \\
&\overset{\eqref{eq:NAlgAx1}}{=}&
\omega_3(\partial x_2, g_2 \partial(x_1) g_1) \omega_2(x_2) \omega_3(g_2 \partial x_1, g_1) \underline{\omega_3(g_2 \partial(x_1) g_2^{-1},g_2)}  \omega_2({^{g_2}} x_1) \\
&\overset{\eqref{eq:omega3Cocycle}}{=}&
\omega_3(\partial x_2, g_2 \partial(x_1) g_1) \omega_2(x_2) \omega_3(g_2 \partial x_1,g_1) \cdot {^{g_2 \partial x_1}} \omega_3(g_2^{-1},g_2) \omega_3(g_2 \partial x_1, g_2^{-1})^{-1} \omega_2({^{g_2}} x_1).
\end{eqnarray*}
Using equation \eqref{eq:NAlgAx2} to rewrite $\omega_2({^{g_2}}x_1)$ in terms of ${^{g_2}}\omega_2(x_1)$, and then using equation \eqref{eq:omega3Cocycle}, we find that
\[
\omega_3(g_2 \partial x_1, g_2^{-1})^{-1} \omega_2({^{g_2}} x_1)
=
\omega_3(g_2, \partial x_1) \cdot {^{g_2}}\omega_2(x_1) \omega_3(g_2,g_2^{-1})^{-1}.
\]
Continuing, we have
\begin{eqnarray*}
\spadesuit
&{=}&
\omega_3(\partial(x_2) g_2, \partial(x_1) g_1) \omega_3(\partial x_2,g_2) \cdot \underline{{^{\partial x_2}}\omega_3(g_2, \partial(x_1) g_1)^{-1} \cdot} \\ &&
\underline{{^{\partial x_2}}\omega_3(g_2 \partial x_1,g_1)} \omega_2(x_2) \cdot {^{g_2 \partial x_1}} \omega_3(g_2^{-1},g_2)
\omega_3(g_2, \partial x_1) \cdot {^{g_2}}\omega_2(x_1) \omega_3(g_2,g_2^{-1})^{-1} \\
&\overset{\eqref{eq:omega3Cocycle}}{=}&
\omega_3(\partial(x_2) g_2, \partial(x_1) g_1)
\underline{\omega_3(\partial x_2 ,g_2) \cdot {^{\partial x_2 }}\left(
{^{g_2}}\omega_3(\partial x_1,g_1)
\right)} \\ &&
\cdot {^{\partial x_2}} \omega_3(g_2, \partial x_1 )^{-1} \omega_2(x_2) \cdot \underline{{^{g_2 \partial x_1}} \omega_3(g_2^{-1},g_2)
\omega_3(g_2, \partial x_1 )} \cdot {^{g_2}}\omega_2(x_1) \omega_3(g_2,g_2^{-1})^{-1} \\
&\overset{\eqref{eq:omega1Proj},\eqref{eq:omega1Proj}}{=}&
\omega_3(\partial(x_2) g_2, \partial(x_1) g_1) \cdot
{^{\partial(x_2) g_2}}\omega_3(\partial x_1,g_1) \omega_3(\partial x_2,g_2) \cdot
\\ && 
\underline{ {^{\partial x_2}} \omega_3(g_2, \partial x_1 )^{-1} \omega_2(x_2)}
\omega_3(g_2, \partial x_1 )
\cdot
    \underline{{}^{g_2}({}^{\partial x_1} \omega_3(g_2^{-1},g_2))
      \cdot {^{g_2}}\omega_2(x_1)} \omega_3(g_2,g_2^{-1})^{-1}
    \\
&\overset{\eqref{eq:NAlgAx1},\eqref{eq:NAlgAx1}}{=}&
\omega_3(\partial(x_2) g_2, \partial(x_1) g_1)
 \cdot {^{\partial(x_2) g_2}}\omega_3(\partial x_1,g_1) \omega_3(\partial x_2,g_2)
      \underline{\omega_2(x_2)}
\\ && 
\omega_3(g_2, \partial x_1 )^{-1}
\omega_3(g_2, \partial x_1 )
\cdot
    \underline{{}^{g_2} (\omega_2(x_1)}\omega_3(g_2^{-1},g_2)) 
\omega_3(g_2,g_2^{-1})^{-1}
\\
&\overset{\eqref{eq:NAlgAx1}}{=}&
\omega_3(\partial(x_2) g_2, \partial(x_1) g_1)
\cdot {^{\partial(x_2) g_2}}\omega_3(\partial x_1,g_1)
\underline{\omega_3(\partial x_2,g_2){}^{\partial x_2} ( {}^{g_2} (\omega_2(x_1)))}
      \\ &&
\omega_2(x_2) \cdot \underline{{}^{g_2} \omega_3(g_2^{-1},g_2) \omega_3(g_2,g_2^{-1})^{-1}}
\\
&\overset{\eqref{eq:NAlgAx1},\eqref{eq:omega3Cocycle}}{=}&
\omega_3(\partial(x_2) g_2, \partial(x_1) g_1) \cdot {^{\partial(x_2) g_2}} \left[
\omega_3(\partial(x_1),g_1) \omega_2(x_1)
\right]
\omega_3(\partial(x_2),g_2) \omega_2(x_2).
\end{eqnarray*}
We need to show that $\Upsilon(\spadesuit^{-1})$ is equal to the composition
\[
\Theta_A(\partial (x_2) g_2, \partial (x_1)  g_1) \circ (\Theta_A(g_2,x_2) \diamond \Theta_A(g_1, x_1)).
\]
Since $\Theta_A(g_2,x_2) \diamond \Theta_A(g_1, x_1)$ is the result of applying $\Upsilon$ to the inverse of
\[
{^{\partial(x_2) g_2}} \left[
\omega_3(\partial x_1,g_1) \omega_2(x_1)
\right]
\omega_3(\partial x_2,g_2) \omega_2(x_2),
\]
the desired equality follows.

The pentagon identity holds by equation \eqref{eq:omega3Cocycle}; details are left to the reader.
\end{proof}

\begin{Prop}
\label{thm:weakAlgTo2ModBicat}
The assignment $A \mapsto \Theta_A$ of Proposition \ref{prop:weakAlgTo2Mod} extends to a locally fully faithful pseudofunctor $\Theta: \cG \mhyphen N\AAlg_{\Fld}^{\fd} \rightarrow \Hom_{\mathsf{Bicat}_{\mathsf{con}}}(\tcG, \AAlg_{\Fld}^{\fd})$.
\end{Prop}

\begin{proof}
After writing out the explicit description of $1$- and $2$-morphisms in the bicategory $1\Hom_{\mathsf{Bicat}_{\mathsf{con}}}(\tcG, \AAlg_{\Fld}^{\fd})$ (see Section~\ref{sec:Real2Rep}) and comparing them to the definition of those of $\cG \mhyphen N\AAlg_{\Fld}^{\fd}$, the statement is clear.
\end{proof}

\begin{Rem}
\begin{enumerate}[label=(\roman*)]
\item More conceptually, $\Theta_A$ can be described as the composition of the structure map $\widetilde{\omega}_A: \tcG \rightarrow \widetilde{\AUT^{\gen}(A)}$ (constructed as in the proof of Proposition \ref{prop:weakAlgTo2Mod}) with
\[
\widetilde{\AUT^{\gen}(A)} \simeq 1 \Aut^{\gen}_{\AAlg_{\Fld}^{\fd,\rep}}(A) \hookrightarrow 1 \Aut^{\gen}_{\AAlg_{\Fld}^{\fd}}(A).
\]
\item When $\cG$ is trivially graded, the proof of Proposition \ref{prop:weakAlgTo2Mod} yields a $2$-module $\Theta_A: \tcG \rightarrow \AALG_{\Fld}$ without any separability assumption on $A$. Compare \cite[Proposition 2.3]{rumynin2018}.
\end{enumerate}
\end{Rem}

\subsection{Morita bicategories and \texorpdfstring{$2$}{}-representations on \texorpdfstring{$2\Vect_{\Fld}$}{}}
\label{sec:morita2Reps}

Permutation actions are at the heart of Real $2$-representations on $2\Vect_{\Fld}$, as we now review.

\begin{Prop}
\label{prop:RealizeCocycle_1}
A Real $2$-representation of $\tcG$ on $2 \Vect_{\Fld}$ of dimension $t$ determines a homomorphism $\sigma: G_1 \rightarrow \Sym$ which factors through the quotient $G_1 \rightarrow \pi_1(\cG)$ and a cocycle $\mu \in Z^2(G_1, ({\Fld}^{\times})^t_{\pi})$, where ${\Fld}^t$ is considered as a $G_1$-algebra via $\sigma$.
\end{Prop}

\begin{proof}
Since $G_1$ is a $\ZTwo$-graded crossed submodule of $\cG$, each Real $2$-representation of $\tcG$ restricts to a Real $2$-representation of $G_1$. The construction of $\sigma$ and $\mu$ are then given in the proof of \cite[Theorem 5.7]{mbyoung2018c}. There it is assumed that $G_1$ is finite, but this is not used in this part of the proof. It is straightforward to show that $\sigma$ factors through $G_1 \rightarrow \pi_1(\cG)$.
\end{proof}

\begin{Lem}[{cf. \cite[Lemma 2.4]{rumynin2018}}]
\label{lem:redToGroups}
Let $\cG$ be a $\ZTwo$-graded crossed module and let $\rho$ be a Real $2$-module over $\tcG$ which restricts to a $G_1$-algebra $A$. Then there exists an RW-weak $\cG$-algebra structure extending the $G_1$-structure of $A$ such that $\rho$ and $\Theta_A$ are equivalent.
\end{Lem}

\begin{proof}
  In the ordinary case this is \cite[Lemma 2.4]{rumynin2018}. The definition of the RW-weak $\cG$-algebra and the verification of the required axioms extend verbatim to the $\ZTwo$-graded case:
  $$
  \omega_2 (x) \coloneqq \Upsilon^{-1} (\rho ( e \xRightarrow[]{x} \partial x)), \qquad
  \omega_3 (f,g) \coloneqq \Upsilon^{-1} (\rho (f,g)).
  $$
It follows from the definition of $\Theta_A$, given in the proof of Proposition~\ref{prop:weakAlgTo2Mod}, that $\Theta_A$ is equivalent to $\rho$. 
\end{proof}

\begin{Prop}
\label{thm:RealizeCocycle}
Let $\rho$ be a Real $2$-representation of $\tcG$ on $2 \Vect_{\Fld}$ with induced morphism $\sigma: G_1 \rightarrow \Sym$ and cocycle $\mu \in Z^2(G_1, ({\Fld}^{\times})^t_{\pi})$. If $\mu$ is realizable, then there exists a split semisimple RW-weak $\cG$-algebra $A$ such that $\rho$ and $\Theta_A$ are equivalent.
\end{Prop}

\begin{proof}
Write $R$ for ${\Fld}^t$, viewed as the target object of the Real $2$-module associated to $\rho$. Let $M$ be a realization of $\mu^{-1}$ with Peirce decomposition $M \simeq \oplus_i M_i$. Then $A = \End_R(M) \simeq \oplus_i \End_{\Fld}(M_i)$ is a split semisimple ${\Fld}$-algebra which, by Proposition \ref{prop:2RepFromRep}, has a $G_1$-algebra structure $\mathfrak{a}$.

We first define an equivalence $F: \rho \Rightarrow \Theta_A$ of Real $2$-modules over $G_1$. Let $F(\star) : R \rightarrow A$ be the $A \mhyphen R$-bimodule $M$, its right $R$-module structure defined via its left $R$-module structure (recall that $R$ is commutative). For each $g \in G_1$, define an $A \mhyphen R$-bimodule intertwiner
\[
F(g) : A_{\mathfrak{a}_g} \otimes_{{^{\pi(g)}}A} ({^{\pi(g)}}M) \Longrightarrow M \otimes_R \rho(g)
\]
as follows. Identify $M \otimes_R \rho(g)$ with $M_{\rho(g)}$ and set $F(g) (a \otimes m) = a(\mathfrak{p}_M(g)(m))$. The map is well-defined because, for each $b \in A$, we have
\begin{eqnarray*}
F(a \mathfrak{a}_g(b) \otimes m)
&=&
a (\mathfrak{a}_g(b)(\mathfrak{p}_M(g)(m))) =
a (\mathfrak{p}_M(g)(b(\mathfrak{p}_M(g)^{-1}(\mathfrak{p}_M(g)(m)))) \\
&=&
a (\mathfrak{p}_M(g)(b(m)))
=
F(a \otimes b (m)).
\end{eqnarray*}
Left $A$-linearity of $F(g)$ is clear. Right $R$-linearity follows from the computation
\[
F(g)((a \otimes m)r)
=
a(\mathfrak{p}_M(g)(r m))
=
(g \cdot r) a(\mathfrak{p}_M(g)(m))
=
F(g)(a \otimes m) \cdot_{\rho(g)} r,
\]
where $-\cdot_{\rho(g)}-$ indicates the $\rho(g)$-twisted right $R$-module structure of $M$.

The coherence of the assignments $g \mapsto F(g)$ amounts to the commutativity of the diagram
\[
\begin{tikzpicture}[baseline= (a).base]
\node[scale=1] (a) at (0,0){
\begin{tikzcd}[column sep=8.0em, row sep=2.0em]
M \otimes (\rho(g_2) \circ {^{\pi(g_2)}}\rho(g_1)) \arrow{r}[above]{\rho_{g_2,g_1}} & M \otimes \rho(g_2 g_1) \\
(M \otimes \rho(g_2)) \circ {^{\pi(g_2)}}\rho(g_1) \arrow{u} & \Theta_A(g_2 g_1) \otimes {^{\pi(g_2 g_1)}}M \arrow{u}[right]{F(g_2g_1)} \\
(\Theta_A(g_2) \otimes {^{\pi(g_2)}}M) \circ {^{\pi(g_2)}}\rho(g_1) \arrow{u}[left]{F(g_2)} & (\Theta_A(g_2) \circ {^{\pi(g_2)}}\Theta_A(g_1)) \otimes {^{\pi(g_2 g_1)}}M \arrow{u}[right]{\Theta_A(g_2,g_1)} \\
\Theta_A(g_2) \circ {^{\pi(g_2)}}(M \otimes \rho(g_1)) \arrow{u} & \Theta_A(g_2) \circ {^{\pi(g_2)}}(\Theta_A(g_1) \otimes {^{\pi(g_1)}}M), \arrow{u} \arrow{l}[below]{{^{\pi(g_2)}}F(g_1)} \arrow[dashed,bend left=5]{luuu}
\end{tikzcd}
};
\end{tikzpicture}
\]
where all unlabelled arrows are associativity isomorphisms. By construction, the cocycle $\mu\in Z^2(G_1,R^{\times}_{\pi})$ is determined by the $2$-isomorphisms $\rho_{g_2,g_1}$. Since $A$ is a strict $G_1$-algebra, the map $\Theta_A(g_2,g_1)$ is induced by the corresponding map in $\AAlg_{\Fld}^{\fd}$ and, in particular, does not involve a cocycle. On the other hand, ``the composition'' of $F(g_2)$ with ${^{\pi(g_2)}}F(g_1)$ (indicated by the dashed arrow in the diagram above) is equal to $\mu(g_2,g_1)^{-1} F(g_2g_1)$.
This gives the required commutativity.

We can now apply Lemma \ref{lem:redToGroups} to conclude that we can extend the $G_1$-algebra structure of $A$ to an RW-weak $\cG$-algebra structure such that $\rho \simeq \Theta_A$ as Real $2$-modules.
\end{proof}

Denote by $\RRep^r_{2 \Vect_{\Fld}}(\tG)$ the full subbicategory of $\RRep_{2 \Vect_{\Fld}}(\tG)$ on realizable Real $2$-representations of $\tG$. It is a monoidal subbicategory.

\begin{Thm}
\label{thm:moritaModel}
Assume that ${\Fld}$ is separably closed. Let $\cG$ be a $\ZTwo$-graded crossed module. Under the identification $\RRep_{2 \Vect_{\Fld}}(\tcG) \simeq \RRep_{\AAlg^{\fd}_{\Fld}}(\tcG)$, the subbicategory $\RRep^r_{2 \Vect_{\Fld}}(\tcG)$ is biequivalent to $\cG \mhyphen \AAlg^{\fd}_{\Fld}$.
\end{Thm}

\begin{proof}
By Theorem ~\ref{thm:strictcGAlg} and Proposition~\ref{thm:RealizeCocycle}, the biessential image of $\cG \mhyphen \AAlg_{\Fld}^{\fd}$ under the embedding of Proposition~\ref{thm:weakAlgTo2ModBicat} can be identified with $\RRep^r_{2 \Vect_{\Fld}}(\tcG)$. 
\end{proof}

We can now describe equivalence classes of Real $2$-representations of $\tcG$ on $2 \Vect_{\Fld}$ in terms of $\cG$-algebras.

\begin{Cor}
\label{thm:Real2RepIsoClass}
Assume that ${\Fld}$ is separably closed and that $G_1$ is finite. Then there is a  bijection
\[
\pi_0(\RRep_{2 \Vect_{\Fld}}(\tcG)) \simeq \{ \textnormal{separable strict } \cG \mhyphen \textnormal{algebras} \} \slash \cG \mhyphen \textnormal{Morita equivalence}.
\]
\end{Cor}

\begin{proof}
Since $G_1$ is finite, each element of  $Z^2(G_1, ({\Fld}^{\times})^t_{\pi})$ is realizable (see Lemma \ref{lem:finiteRealizable}) and we have $\RRep^r_{2 \Vect_{\Fld}}(\tcG) = \RRep_{2 \Vect_{\Fld}}(\tcG)$. Theorem \ref{thm:moritaModel} then gives the desired bijection. 
\end{proof}

\begin{Rem}
\begin{enumerate}[label=(\roman*)]
\item 
  Without the finiteness assumption on $G_1$, the following version of Corollary~\ref{thm:Real2RepIsoClass} still holds:
  $$
  \pi_0(\RRep^r_{2 \Vect_{\Fld}}(\tcG)) \simeq \{ \textnormal{separable strict } \cG \mhyphen \textnormal{algebras} \} \slash \cG \mhyphen \textnormal{Morita equivalence}.
$$
\item When $\cG$ is trivially graded, we recover the finite dimensional results of \cite[\S 3]{rumynin2018}. It is not clear, however, if the results involving semi-matrix algebras (used to treat non-realizable cocycles) admit a $\ZTwo$-graded generalization.
\end{enumerate}
\end{Rem}

\begin{Ex}
Suppose that $\cG$ is a $\ZTwo$-graded crossed module in which $\partial$ is trivial; this corresponds to so-called split $\ZTwo$-graded $2$-groups, that is, those with trivial Sinh $3$-cocycle. The case of non-trivial $\partial$ can be treated by modifying the construction below to include a non-trivial $3$-cocycle. The set $\pi_0(\cG \mhyphen \AAlg_{\Fld}^{\fd})$ appearing in Corollary~\ref{thm:Real2RepIsoClass} can be described as follows.

First, we note that any split semisimple $\cG$-algebra $A$ with $t$ simple summands is $\cG$-Morita equivalent to an RW-weak $\cG$-algebra whose underlying ${\Fld}$-algebra is $\Fld^t$. This can be seen by using the proof of Theorem \ref{thm:strictcGAlg}: the permutation $\pi$, the cocycle $\mu$ and the function $\gamma$ are used to define $\omega_{\Fld^t,1}$, $\omega_{\Fld^t,2}$ and $\omega_{\Fld^t,3}$, respectively, while the projective Real representation $(U,\eta_U)$ is used to define the between $A \simeq \Fld^t$.

It therefore suffices to restrict attention to RW-weak $\cG$-algebra structures on $A= \Fld^t$. In this case, we have a homomorphism $\omega_{A,1}: G_1 \rightarrow \Sym$, a $G_1$-invariant homomorphism $\omega_{A,2}: G_2 \rightarrow ({\Fld}^{\times})_{\pi}^t$ and a cocycle $\omega_{A,3} \in Z^2(G_1, ({\Fld}^{\times})^t_{\pi})$. However, different triples can $\{\omega_{A,i}\}_i$ define $\cG$-Morita equivalent RW-weak $\cG$-algebras. There are two sources of ambiguity in the data $\{\omega_{A,i}\}_i$, which we now describe. Recall (see Section \ref{sec:strictWeakAlg}) that the generalized automorphism $2$-group of ${\Fld}^t \in \AAlg_{\Fld}^{\fd}$ is modelled by $\AUT^{\gen}({\Fld}^t) = (({\Fld}^{\times})^t \xrightarrow[]{\partial} \Sym \times \ZTwo)$. Conjugating by an arbitrary element $\tau \in \Sym \times \ZTwo$ sends $\omega_A$ to the $\cG$-Morita equivalent RW-weak $\cG$-algebra determined by the triple
\[
(
\tau \omega_{A,1}(-) \tau^{-1},
{^{\tau}}\omega_{A,2}(-),
\tau_* \omega_{A,3}(-)
).
\]
The second source of ambiguity is due to non-trivial $G_1$-module structures on the identity ${\Fld}^t$-bimodule. Given a such a $G_1$-module structure, determined by a map $\mu: G_1 \rightarrow ({\Fld}^{\times})^t$, we obtain an RW-weak $\cG$-algebra $(\omega_{A,1}, \omega_{A,2}, \omega_{A,3} d \mu)$ which is $\cG$-Morita equivalent to $A$. In this way, we obtain a $\ZTwo$-graded generalization of the split case of \cite[Theorem 5.5]{elgueta2007}. See also \cite[\S 5.3]{mbyoung2018c}.
\end{Ex}

We use Theorem \ref{thm:moritaModel} to model direct sums, tensor products and duals of realizable (Real) $2$-modules in terms of the corresponding operations for $\cG$-algebras. We also define various inductions of (Real) $2$-modules in the same way. For example, if $\Theta_A$ is a realizable $2$-module over an non-trivially graded crossed submodule $\cH$ of $\cG$, then the Real $2$-module $\RInd_{\cH}^{\cG} \Theta_A$ is defined to be $\Theta_{\RInd_{\cH}^{\cG} A}$.

\subsection{A structure theorem for Real \texorpdfstring{$2$}{}-modules}

The next result plays an important role in Section \ref{sec:K0RRep}.

\begin{Thm}
\label{thm:ksDecomp}
Let $\cG$ be a $\ZTwo$-graded crossed module, $\rho$ a realizable Real $2$-module over $\tcG$. The following statements hold.
\begin{enumerate}[label=(\roman*)]
\item There exist indecomposable Real $2$-modules $\rho_1, \dots, \rho_n$ over $\tcG$, unique up to equivalence, such that $\rho \simeq \boxplus_{i=1}^n \rho_i$.

\item If $\rho$ is indecomposable, then there exists a subgroup $H$ of $G_1$ such that $\cG_H\coloneqq (G_2 \xrightarrow[]{\partial} H)$ is a crossed submodule of $\cG$ and a one dimensional realizable $2$-module $\nu$ over $\widetilde{\cG}_H$, Real if $H$ is non-trivially graded, such that $\rho \simeq \Ind_{\tcG_H}^{\tcG} \nu$. Moreover, the pair $(H, \nu)$ is unique up to $G_1$-conjugation, $(H, \nu)\mapsto(gHg^{-1}, g\cdot {}^{\pi(g)}\nu)$, $g\in G_1$, and equivalence in $\nu$.
\end{enumerate}
\end{Thm}

\begin{proof}
By Theorem \ref{thm:moritaModel}, it suffices to prove the statement at the level of $\cG$-algebras.

Let $A$ be a split semisimple $\cG$-algebra. Decompose the underlying ${\Fld}$-algebra into simple factors, $A = \oplus_{i=1}^t A_i$. Then $G_1$ acts on the set $\{1, \dots, t\}$, giving an orbit decomposition $\mathcal{O}_1 \sqcup \cdots \sqcup \mathcal{O}_n$, so that each $A_{\mathcal{O}_i}\coloneqq\oplus_{j \in \mathcal{O}_i} A_j$, $i=1, \dots, n$, is an indecomposable $\cG$-algebra and $A \simeq \boxplus_{i=1}^n A_{\mathcal{O}_i}$. The uniqueness statement is clear.

Turning to the second statement, if $A$ is indecomposable, then it is necessarily of the form $M_n({\Fld})^{\oplus t}$ and the $G_1$-action on $\{1, \dots, t\}$ is transitive. Let $H \leq G_1$ be the stabilizer of $1 \in \{1, \dots, t\}$; it is trivially graded if no element of $G_1 \setminus G_0$ fixes $1$, and is non-trivially graded otherwise. Set $B=M_n({\Fld})$, regarded as the first summand of $A$. Then $B$ inherits from $A$ a $\cG_H$-algebra structure. The resulting (Real) $\tcG_H$-module has dimension one, that is, its underlying $\Fld$-algebra $B$ is simple. It follows immediately from the definitions that $\Ind_{\cG_H}^{\cG} B \simeq A$ as $\cG$-algebras, where $\Ind$ denotes $\HInd$ or $\RInd$ as appropriate. It is clear that we can replace $B$ with any $\cG_H$-Morita equivalent $\cG_H$-algebra. If we consider instead the stabilizer of a point other than $1$, then $H$ is replaced with a $G_1$-conjugate, say $H^{\prime} = g H g^{-1}$, and $B$ is replaced $g \cdot {^{\pi(g)}}B$.
\end{proof}

\begin{Rem}
When $\cG$ is trivially graded, variants of Theorem \ref{thm:ksDecomp}(ii) are well-known; see, for example,
\cite[Theorem 3.2]{ostrik2003}, \cite[Proposition 7.3]{ganter2008} and \cite[Theorem 3.7]{rumynin2018}.
\end{Rem}


\section{The Grothendieck ring of Real \texorpdfstring{$2$}{}-representations}
\label{sec:GrothGrpRRep}

In this section we describe the Grothendieck ring of $\RRep_{2\Vect_{\Fld}}(\tcG)$ in terms of Real generalized Burnside rings.

\subsection{Generalized Burnside rings}

We recall some basic material about generalized Burnside rings. The reader is referred to \cite[\S 1]{gunnells2012} and \cite[\S 4]{rumynin2018} for details.

Given a finite group $G$, let $\mathcal{S}(G)$ be the category whose objects are subgroups of $G$ and whose morphisms are conjugations,
\[
\Hom_{\mathcal{S}(G)}(H_1, H_2) = \{ \gamma_g \mid g \in G, \; g H_1 g^{-1} \leq H_2\}.
\]
Morphisms are composed using the multiplication in $G$. We emphasize that, even if $e \neq g_1g_2^{-1}$ is in the centralizer of $H_1$, the morphisms $\gamma_{g_1}, \gamma_{g_2}: H_1 \rightarrow H_2$ are not identified. This is as in \cite{rumynin2018} and is in contrast to \cite{gunnells2012}.

Suppose now that $G$ is $\ZTwo$-graded. The degree $\pi(H)$ of $H \in \mathcal{S}(G)$ is defined to be $+1$ if $H$ is trivially graded and $-1$ otherwise. There are no morphisms in $\mathcal{S}(G)$ from an object of degree $-1$ to an object of degree $+1$.

Fix a functor $\Phi: \mathcal{S}(G)^{\opp} \rightarrow \mathsf{SGrp}$ to the category of semigroups. A $\Phi$-decorated $G$-set is a $G$-set $X$ together with frills $\mathfrak{f}_{\vx} \in \Phi(\Stab_G(\vx))$, $\vx \in X$, which satisfy $\Phi(\gamma_g) (\mathfrak{f}_{\vx}) = \mathfrak{f}_{g\vx}$ for each $g \in G$. The generalized Burnside ring $\mathbb{B}^{\Phi}(G)$ is defined to be the Grothendieck group of the category of finite $\Phi$-decorated $G$-sets. An explicit description of $\mathbb{B}^{\Phi}(G)$ in our case of interest is given in Section~\ref{sec:K0RRep}. Given a commutative ring $\mathbb{A}$, we write $\mathbb{B}_{\mathbb{A}}^{\Phi}(G)$ for $\mathbb{B}^{\Phi}(G) \otimes_{\mathbb{Z}} \mathbb{A}$.

\subsection{The Grothendieck ring of Real \texorpdfstring{$2$}{}-representations}
\label{sec:K0RRep}

Let $\cK$ be a $\ZTwo$-graded crossed module. Denote by $\pi_0(\RRep_{[1]}(\tcK))$ the abelian group of equivalence classes of Real $2$-representations of $\tcK$ on $[1] \in 2 \Vect_{\Fld}$. If $\cK$ is trivially graded, then this is simply $\pi_0(\Rep_{[1]}(\tcK))$. The group operation is the tensor product $\boxtimes$ of (Real) $2$-representations.

Let $\cG$ be a $\ZTwo$-graded crossed module with $G_1$ finite. The group $\pi_1(\cG)$ inherits a $\ZTwo$-grading from that of $G_1$. Given a subgroup $P \leq \pi_1(\cG)$, denote by $\overline{P}$ its pre-image under the quotient map $G_1 \rightarrow \pi_1(\cG)$. Write $\pi(P)$ for the image of $P$ under the induced map $\pi_1(\cG) \rightarrow \ZTwo$. The crossed module $\cG_P \coloneqq (G_2 \xrightarrow[]{\partial} \overline{P})$ is a $\ZTwo$-graded crossed submodule of $\cG$; the grading trivial if and only if $P \leq \pi_1(\cG)_0$.

Define a functor $\Phi: \mathcal{S}(\pi_1(\cG))^{\opp} \rightarrow \mathsf{Ab}$ to the category of abelian groups as follows. At the level of objects, set
\[
\Phi(P) =
\pi_0(\RRep_{[1]}(\tcG_P)).
\]
Let $\gamma_g: P_1 \rightarrow P_2$ be a morphism in $\mathcal{S}(\pi_1(\cG))$. The choice of a lift $\dot{g} \in G_1$ of $g \in \pi_1(\cG)$ induces a strict $\ZTwo$-graded crossed module homomorphism $\gamma_{\dot{g}}: \cG_{P_1} \rightarrow \cG_{P_2}$ by
\[
\overline{P}_1 \owns p \mapsto \dot{g} p \dot{g}^{-1}, \qquad G_2 \owns x \mapsto {^{\dot{g}}} x.
\]
The associated $\ZTwo$-graded $2$-group homomorphism is $\gamma_{\dot{g}}: \tcG_{P_1} \rightarrow \tcG_{P_2}$. With this notation, define $\Phi(\gamma_g)$ by
\[
\Phi(\gamma_g)(\rho)=
\begin{cases}
{^{\pi(g)}}\rho \circ \gamma_{\dot{g}} & \mbox{if } \pi(P_1) = \pi(P_2), \\
\mathfrak{F}(\rho) \circ \gamma_{\dot{g}} & \mbox{otherwise}.
\end{cases}
\]
Here $\mathfrak{F}: \pi_0(\RRep_{[1]}(\tcG_P)) \rightarrow \pi_0(\Rep_{[1]}(\tcG_{P_0}))$ is the forgetful map. Well-definedness of $\Phi$, that is, independence of the lift $\dot{g}$ of $g$, can be verified as in \cite[Lemma 4.1]{rumynin2018}.

\begin{Def}
The Real Burnside ring of $\cG$ is $\mathbb{B}^{\Phi}_{\mathbb{A}}(\cG) : = \mathbb{B}^{\Phi}_{\mathbb{A}}(\pi_1(\cG))$.
\end{Def}

It follows from the general theory of generalized Burnside rings that $\mathbb{B}^{\Phi}_{\mathbb{A}}(\cG)$ is generated as an $\mathbb{A}$-module by pairs $\langle \rho, P\rangle$, where $P \leq \pi_1(\cG)$ and $\rho \in \pi_0(\RRep_{[1]}(\tcG_P))$. The ring structure of $\mathbb{B}^{\Phi}_{\mathbb{A}}(\cG)$ is determined by the formula
\begin{multline*}
\langle \rho, P \rangle \cdot \langle \theta, Q \rangle = \\
\sum_{P g Q \in P \backslash \pi_1(\cG) \slash Q} \langle \Phi(\gamma_e : P \cap g Q g^{-1} \rightarrow P)(\rho) \boxtimes \Phi(\gamma_{g^{-1}} : P \cap g Q g^{-1} \rightarrow Q)(\theta), P \cap g Q g^{-1} \rangle.
\end{multline*}

\begin{Ex}
We consider three illustrative cases of the product.
\begin{enumerate}[label=(\roman*)]
\item Let $\spcl \in \pi_1(\cG)\setminus\pi_1(\cG)_0$ with lift $\dot{\spcl} \in G_1 \setminus G_0$ and identify $\pi_1(\cG)_0 \backslash \pi_1(\cG) \slash \pi_1(\cG)_0$ with $\{e, \spcl\}$. Then we have
\begin{eqnarray*}
\langle \rho, \pi_1(\cG)_0 \rangle \cdot \langle \theta, \pi_1(\cG)_0 \rangle
&=&
\langle \rho \boxtimes \theta, \pi_1(\cG)_0 \rangle 
+
\langle \rho \boxtimes \dot{\spcl} \cdot \theta^{\vee}, \pi_1(\cG)_0 \rangle,
\end{eqnarray*}
where the notation $\dot{\spcl} \cdot \theta^{\vee}$ is as in Section \ref{sec:indRealKAlg}.

\item Similarly, we can identify $\pi_1(\cG) \backslash \pi_1(\cG) \slash \pi_1(\cG)_0$ with $\{e\}$, so that
\[
\langle \rho, \pi_1(\cG) \rangle \cdot \langle \theta, \pi_1(\cG)_0 \rangle = \langle \Res^{\tcG}_{\tcG_0} (\rho) \boxtimes \theta, \pi_1(\cG)_0 \rangle.
\]

\item Finally, identifying $\pi_1(\cG) \backslash \pi_1(\cG) \slash \pi_1(\cG)$ with $\{e\}$, we have
\[
\langle \rho, \pi_1(\cG) \rangle \cdot \langle \theta, \pi_1(\cG) \rangle = \langle \rho \boxtimes \theta, \pi_1(\cG) \rangle.
\]
\end{enumerate}
\end{Ex}

More generally, we can define the symbol $\langle \rho, P \rangle$ without the assumption that $\rho$ is one dimensional. To do so, set $\langle \rho_1 \boxplus \rho_2, P \rangle \coloneqq \langle \rho_1, P \rangle + \langle \rho_2, P \rangle$ for any two (Real) 2-representations $\rho_1$, $\rho_2$. Set also $\langle \Ind_{\tcG_P}^{\tcG_Q} \rho,  Q \rangle \coloneqq \langle \rho,  P \rangle$ for subgroups $P \leq Q \leq \pi_1(\cG)$. Here $\Ind_{\tcG_P}^{\tcG_Q}$ has one of three meanings, depending on the $\ZTwo$-gradings of $P$ and $Q$. Theorem \ref{thm:ksDecomp} implies that these definitions are unambiguous. 

For any subgroup $P \leq \pi_1(\cG)$ and $g \in \pi_1(\cG)$, the relation
\[
\langle \rho, P \rangle = \langle \Phi(\gamma_g)(\rho), g^{-1} P g \rangle
\]
holds. In particular, when $P$ is trivially graded and $\pi(g)=-1$, this relation becomes that of Lemma \ref{lem:HIndSymmetry}. Note that there is no reason for this relation to hold in $\mathbb{B}_{\mathbb{A}}^{\Phi}(\cG_0)$.

The Grothendieck group $K_0(\mathcal{V})$ of a bicategory $\mathcal{V}$ is defined to be the free abelian group generated by equivalence classes of objects of $\mathcal{V}$. If $\mathcal{V}$ is symmetric monoidal, then $K_0(\mathcal{V})$ has the structure of a commutative ring.

The (ungraded) Burnside ring of $\cG_0$, defined to be $\mathbb{B}^{\Phi}_{\mathbb{A}}(\cG_0) : = \mathbb{B}^{\Phi}_{\mathbb{A}}(\pi_1(\cG_0))$, is shown in \cite[Proposition 4.2]{rumynin2018} to be isomorphic to $ K_0(\Rep_{2\Vect_{\Fld}}(\tcG_0)) \otimes_{\mathbb{Z}} \mathbb{A}$. The next result gives a Real generalization.

\begin{Thm}
\label{thm:grothGroupDescr}
Assume that $G_1$ is finite. Then the assignment $\langle \rho, P \rangle \mapsto \Ind_{\tcG_P}^{\tcG} \rho$ extends to an $\mathbb{A}$-algebra isomorphism $\mathcal{I}: \mathbb{B}^{\Phi}_{\mathbb{A}}(\cG) \xrightarrow[]{\sim} K_0(\RRep_{2 \Vect_{\Fld}}(\tcG)) \otimes_{\mathbb{Z}} \mathbb{A}$.
\end{Thm}

\begin{proof}
Since $G_1$ is finite, the finiteness assumptions of Section \ref{sec:indBiadjunct} are satisfied, so that we have (Real) induction $2$-functors on (Real) $2$-representations. Moreover, Lemma \ref{lem:finiteRealizable} applies, so that all Real $2$-modules over $\tcG$ are realizable.

  It follows from Theorem \ref{thm:ksDecomp} that $\mathcal{I}$ is an $\mathbb{A}$-module isomorphism. We need to show that $\mathcal{I}$ is also a map of algebras.

As a first case, suppose that $P, Q \leq \pi_1(\cG)$ are trivially graded. Let $\spcl \in \pi_1(\cG) \setminus \pi_1(\cG)_0$. Fix a complete set $\mathcal{T} \subset \pi_1(\cG)_0$ of representatives of $P \backslash \pi_1(\cG)_0 \slash Q$. Then $\mathcal{T} \sqcup \mathcal{T}\cdot \spcl$ is a complete set of representatives of $P \backslash \pi_1(\cG) \slash Q$ and we can write
\begin{multline*}
\langle \rho, P \rangle \cdot \langle \theta, Q \rangle = \\
\sum_{g \in \mathcal{T}} \langle \Phi(\gamma_e : P \cap g Q g^{-1} \rightarrow P)(\rho) \boxtimes \Phi(\gamma_{g^{-1}} : P \cap g Q g^{-1} \rightarrow Q)(\theta), P \cap g Q g^{-1} \rangle +\\
\sum_{g \in \mathcal{T}} \langle \Phi(\gamma_e : P \cap g \spcl Q (g\spcl)^{-1} \rightarrow P)(\rho) \boxtimes \Phi(\gamma_{(g\spcl)^{-1}} : P \cap g \spcl Q (g\spcl)^{-1} \rightarrow Q)(\theta), P \cap g \spcl Q (g\spcl)^{-1}  \rangle.
\end{multline*}
Let us first interpret the right hand side as an element of the (ungraded) ring $\mathbb{B}_{\mathbb{A}}^{\Phi}(\cG_0)$ with product $- \cdot_0 -$. The first and second lines are then $\langle \rho, P \rangle \cdot_0 \langle \theta, Q \rangle$ and $\langle \rho, P \rangle \cdot_0 \langle \dot{\spcl} \cdot \theta^{\vee}, \spcl Q \spcl^{-1} \rangle$, respectively. Under the isomorphism $\mathbb{B}_{\mathbb{A}}^{\Phi}(\cG_0) \simeq K_0(\Rep_{2\Vect_{\Fld}}(\tcG_0)) \otimes_{\mathbb{Z}} \mathbb{A}$, the right hand side is thus
\[
\big(
\Ind_{\tcG_P}^{\tcG_0} (\rho) \boxtimes \Ind_{\tcG_Q}^{\tcG_0} (\theta)
\big)
\boxplus
\big(
\Ind_{\tcG_P}^{\tcG_0} (\rho) \boxtimes 
\Ind_{\tcG_{\spcl Q \spcl^{-1}}}^{\tcG_0} (\dot{\spcl} \cdot \theta^{\vee})
\big)
\simeq
\Ind_{\tcG_P}^{\tcG_0} (\rho) \boxtimes \Res_{\tcG_0}^{\tcG} \big(
\HInd_{\tcG_Q}^{\tcG} (\theta)
\big).
\]
Applying the map $\mathcal{I}$ then corresponds to applying hyperbolic induction. Doing so gives $\HInd_{\tcG_P}^{\tcG} (\rho) \boxtimes \HInd_{\tcG_Q}^{\tcG} (\theta)$, as required.

If instead $P$ is non-trivially graded and $Q$ is trivially graded, then we can choose a complete set of representatives of $P \backslash \pi_1(\cG) \slash Q$ of the form $\mathcal{T} \subset \pi_1(\cG)_0$ and
\begin{multline*}
\langle \rho, P \rangle \cdot \langle \theta, Q \rangle = \\
\sum_{g \in \mathcal{T}} \langle \Phi(\gamma_e : P \cap g Q g^{-1} \rightarrow P)(\rho) \boxtimes \Phi(\gamma_{g^{-1}} : P \cap g Q g^{-1} \rightarrow Q)(\theta), P \cap g Q g^{-1} \rangle.
\end{multline*}
Interpreted as an element of $\mathbb{B}_{\mathbb{A}}^{\Phi}(\cG_0)$, the right hand side is $\langle \rho, P_0 \rangle \cdot_0 \langle \theta, Q \rangle$, which corresponds to
\[
\Ind_{\tcG_{P_0}}^{\tcG_0} \big(
\Res_{\tcG_{P_0}}^{\tcG_P}(\rho)
\big)
\boxtimes
\Ind_{\tcG_Q}^{\tcG_0} (\theta)
\simeq
\Res_{\tcG_0}^{\tcG} \big(
\RInd_{\tcG_P}^{\tcG} (\rho)
\big)
\boxtimes
\Ind_{\tcG_Q}^{\tcG_0} (\theta)
\]
in $\mathbb{B}_{\mathbb{A}}^{\Phi}(\cG_0)$. Applying $\HInd_{\tcG_0}^{\tcG}$ and using the equivalence \eqref{eq:hypModFunctor} then gives $\RInd_{\tcG_P}^{\tcG} (\rho) \boxtimes \HInd_{\tcG_Q}^{\tcG} (\theta)$, as required.

The case in which both $P$ and $Q$ are non-trivially $\ZTwo$-graded is similar.
\end{proof}


\section{Real categorical character theory}
\label{sec:RealChar}

The categorical character theory of $2$-representations was developed by Bartlett \cite[\S 3]{bartlett2011} and Ganter and Kapranov \cite[\S 4]{ganter2008} in the case of finite groups and extended to essentially finite $2$-groups by Rumynin and Wendland \cite[\S 5]{rumynin2018}. A Real generalization for finite groups was introduced in \cite[\S 5]{mbyoung2018c}. In each of these settings, there are two levels of characters, reflecting to the higher categorical nature of the representations involved. Motivated by work of Willerton \cite{willerton2008}, in this section we give a geometric formulation of this theory, simultaneously generalizing it to Real $2$-representations of essentially finite $2$-groups.

\subsection{Loop spaces of crossed modules}
\label{sec:loopSpaces}

Let $G$ be a finite group. It is well-known (see, for example, \cite{willerton2008}) that the loop groupoid of $B G$ is equivalent to action groupoid $G \git G$, with $G$ acting by conjugation, while the double loop groupoid is equivalent to the action groupoid associated to the simultaneous conjugation of commuting pairs in $G$. In this section we describe crossed module generalizations of these statements.

Let $\cG$ be a crossed module. We do not impose any finiteness conditions on $\cG$. Denote by $\vert \cG \vert$ the geometric realization of $\cG$. The homotopy $2$-type of the free loop space $\textnormal{Maps}(S^1, \vert \cG \vert)$ can be modelled by a crossed module in groupoids, which we denote by
\gls{glLG}.
A result of Brown \cite[Theorem 2.1]{brown2010} gives the following explicit description of $\mathcal{L} \cG$. The base groupoid $\mathcal{L}_{\leq 1} \cG = (\mathcal{L}_1 \cG \righttwoarrow \mathcal{L}_0 \cG)$ has objects $\mathcal{L}_0 \cG =G_1$ and morphisms $g_1 \rightarrow g_2$  given by pairs $(f,x) \in G_1 \times G_2$ which satisfy $g_2 = f \partial(x) g_1 f^{-1}$. Morphisms are composed according to the rule
\[
\big(
g_2\xrightarrow[]{(f_2,x_2)} g_3
\big)
\circ
\big(
g_1 \xrightarrow[]{(f_1,x_1)} g_2
\big)
=
g_1 \xrightarrow[]{(f_2 f_1,{^{f_1^{-1}}}x_2 x_1)} g_3.
\]
The groupoid $\mathcal{L}_2\cG$ is totally disconnected, that is, $\mathcal{L}_2\cG$ is equal to the disjoint union of its connected components. The group $\mathcal{L}_2\cG(g)$ sitting over $g \in \mathcal{L}_0\cG$ is $G_2$ and the restriction of the boundary functor $\partial: \mathcal{L}_2 \cG \rightarrow \mathcal{L}_{\leq 1} \cG$ to $\mathcal{L}_2\cG(g)$ is given on morphisms by
\[
\partial_g(z) =(\partial z, z^{-1}({^{g}}z)), \qquad z \in G_2.
\]
The action of $(g_1 \xrightarrow[]{(f,x)} g_2) \in \mathcal{L}_{\leq 1} \cG$ on $z \in \mathcal{L}_2\cG(g)$ is defined only when $g=g_1$, in which case it is equal to ${^{f}}z \in \mathcal{L}_2\cG(g_2)$.

The fundamental groupoid $\pi_{\leq 1}(\mathcal{L} \cG)$ of $\mathcal{L} \cG$ is defined to be the quotient of $\mathcal{L}_{\leq 1} \cG$ by the totally disconnected normal subgroupoid $\partial(\mathcal{L}_2 \cG)$. Similarly, the loop groupoid $\Lambda \pi_{\leq 1}(\mathcal{L} \cG)$ is defined to be the groupoid of functors $B \mathbb{Z} \rightarrow \pi_{\leq 1}(\mathcal{L} \cG)$. Concretely, objects of $\Lambda \pi_{\leq 1}(\mathcal{L} \cG)$ are pairs $(g, \gamma)$, where $g \in G_1$ and $\gamma \in \End_{\pi_{\leq 1}(\mathcal{L} \cG)}(g) \simeq \End_{\mathcal{L}_{\leq 1} \cG}(g) \slash \im(\partial_g)$. A morphism $\mu: (g_1,\gamma_1) \rightarrow (g_2,\gamma_2)$ is a morphism $\mu: g_1 \rightarrow g_2$ in $\pi_{\leq 1}(\mathcal{L} \cG)$ which satisfies $\mu \gamma_1 = \gamma_2 \mu$.

Reflection of the circle $S^1$ defines a weak involution $\mathfrak{i}: \mathcal{L} \cG \rightarrow \mathcal{L} \cG$. Following the proof of \cite[Theorem 2.1]{brown2010}, we find that $\mathfrak{i}$ is given on $\mathcal{L}_0 \cG$ by $\mathfrak{i}(g) = g^{-1}$, on $\mathcal{L}_1 \cG$ by
\[
\mathfrak{i}(g_1 \xrightarrow[]{(f,x)} g_2) = g^{-1}_1  \xrightarrow[]{(f, {^{g_1^{-1}}} x^{-1})} g^{-1}_2
\]
and on $\mathcal{L}_2\cG$ by $\mathfrak{i}(z) = z$.

Suppose now that we are given a $\ZTwo$-grading $\pi: \cG \rightarrow \ZTwo$. At the level of classifying spaces, $\pi$ determines an equivalence class of double covers of $\cG$ which, as is easy to verify, is represented by the canonical map $\cG_0 \rightarrow \cG$. By choosing an element $\spcl \in G_1 \setminus G_0$, the non-trivial deck transformation of $\cG_0$ can be realized as
\[
G_0 \owns g \mapsto {\spcl}g \spcl^{-1}, \qquad G_2 \owns x \mapsto {^{\spcl}}x.
\]
Note that this squares to the inner automorphism determined by ${\spcl}^2 \in G_0$ and, hence, is homotopic to the identity.

Similarly, $\pi$ induces a grading $\mathcal{L} \cG \rightarrow \ZTwo$. The associated double cover is realized by the canonical map $\mathcal{L} (\cG_0) \rightarrow \mathcal{L} \cG$ together with the non-trivial deck transformation $\sigma_{{\spcl}} : \mathcal{L} (\cG_0) \rightarrow \mathcal{L} (\cG_0)$ given on $\mathcal{L}_0 (\cG_0)$ by $\sigma_{{\spcl}}(g) = {\spcl} g {\spcl}^{-1}$, on $\mathcal{L}_1 (\cG_0)$ by
\[
\sigma_{{\spcl}}(g_1 \xrightarrow[]{(f,x)} g_2) = {\spcl} g_1 {\spcl}^{-1} \xrightarrow[]{({\spcl} f {\spcl}^{-1}, {^{{\spcl}}} x)} {\spcl} g_2 {\spcl}^{-1}
\]
and on $\mathcal{L}_2(\cG_0)$ by $\sigma_{{\spcl}}(z) = {^{{\spcl}}}z$.

In the setting of Real representation theory, an unoriented version
\gls{glLGr}
of the free loop space $\mathcal{L} (\cG_0)$ arises in a natural way; the superscript ``ref'' stands for reflection. This is now a $\ZTwo$-graded crossed module in groupoids. The base groupoid $\mathcal{L}^{\refl}_{\leq 1} \cG$ has objects $\mathcal{L}^{\refl}_0 \cG =G_0$ and morphisms $g_1 \rightarrow g_2$ given by pairs $(f,x) \in G_1 \times G_2$ which satisfy $g_2 = f \partial(x) g_1^{\pi(f)} f^{-1}$. The composition law is $(f_2,x_2) \circ (f_1,x_1) = (f_2 f_1,({^{f_1^{-1}}}x_2) \tilde{x}_1)$, where
\[
\tilde{x}_1
=
\begin{cases}
x_1 & \mbox{ if } \pi(f_2)=+1, \\
{^{(g_1^{-\pi(f_1)})}}x^{-1}_1 & \mbox{ if } \pi(f_2)=-1.
\end{cases}
\]
Each group $\mathcal{L}^{\refl}_2\cG(g)$ is again $G_2$ and $\partial_g(z) =(\partial z, z^{-1}({^{g}}z))$. The action of $\mathcal{L}^{\refl}_{\leq 1} \cG$ on $\mathcal{L}^{\refl}_2 \cG$ is as for $\mathcal{L} \cG$. There is an induced $\ZTwo$-grading $\pi: \mathcal{L}^{\refl} \cG \rightarrow B\ZTwo$ which records the degree of morphisms in $\mathcal{L}_{\leq 1}^{\refl} \cG$. The associated double cover is the canonical morphism $\mathcal{L} (\cG_0) \rightarrow \mathcal{L}^{\refl} \cG$ with the non-trivial deck transformation $\mathfrak{i}_{{\spcl}} \coloneqq \mathfrak{i} \circ \sigma_{{\spcl}} = \sigma_{{\spcl}} \circ \mathfrak{i}$, that is, the diagonal $\ZTwo$-action induced by loop reflection and the deck transformation of $\cG_0$.

For later convenience, if $\cG$ is trivially graded, we take $\mathcal{L}^{\refl} \cG$ to mean $\mathcal{L} \cG$.

\subsection{Real categorical characters and \texorpdfstring{$2$}{}-characters}
\label{sec:RealCatChar}

Let $\cG$ be a $\ZTwo$-graded crossed module. Let $\rho$ be a Real $2$-representation of $\tcG$ on an object $V$ of a ${\Fld}$-linear bicategory $\mathcal{V}$ with weak duality involution $(-)^\circ$. The Real categorical character
\gls{glTr}
of $\rho$ is defined as follows. For each $g \in G_0$, define a vector space 
\[
\Tr_{\rho}(g) \coloneqq 2 \Hom_{\mathcal{V}}(\id_V,\rho(g)).
\]
We assume that $\Tr_{\rho}(g)$ is finite dimensional over $\Fld$. We do not assign a value of $\Tr_{\rho}$ to elements of $G_1 \setminus G_0$. For each $(f,x) \in G_1 \times G_2$, define a linear map
\[
\Tr_{\rho}(g; f,x) : \Tr_{\rho}(g) \rightarrow \Tr_{\rho}(f \partial(x) (g^{\pi(f)})  f^{-1})
\]
so that its value on $( \id_V \xRightarrow[]{u} \rho(g) ) \in \Tr_{\rho}(g)$ is the composition
\[
\id_V \Rightarrow \rho(f) \circ \rho(f^{-1}) \xRightarrow[]{u} \rho(f) \circ \rho(g) \circ \rho(f^{-1}) \xRightarrow[]{\rho(x)} \rho(f) \circ \rho(\partial(x) g) \circ \rho(f^{-1}) \Rightarrow \rho(f \partial(x) g f^{-1})
\]
when $\pi(f)=+1$, as in \cite{ganter2008,rumynin2018}, and
\begin{multline*}
\id_V \Rightarrow \rho(f) \circ \rho(f^{-1})^{\circ} \Rightarrow \rho(f) \circ \rho(g)^{\circ} \circ \rho(g^{-1})^{\circ} \circ \rho(f^{-1})^{\circ} \xRightarrow[]{u^{\circ}} \\
\rho(f) \circ \rho(g^{-1})^{\circ} \circ \rho(f^{-1})^{\circ} \xRightarrow[]{(\rho(x)^{-1})^{\circ}} \rho(f) \circ \rho(\partial(x) g^{-1})^{\circ} \circ \rho(f^{-1})^{\circ} \Rightarrow \rho(f \partial(x) g^{-1} f^{-1})
\end{multline*}
when $\pi(f)=-1$. The unlabelled maps appearing in the previous compositions are $2$-isomorphisms constructed from the composition $2$-isomorphisms of $\rho$. In the second composition $\rho(x)$ is a $2$-morphism $\rho(g^{-1}) \Rightarrow \rho(\partial(x) g^{-1})$, so that $\rho(x)^{\circ}: \rho(\partial(x)g^{-1})^{\circ} \Rightarrow \rho(g^{-1})^{\circ}$.

\begin{Prop}
\label{prop:RealCatChar}
The Real categorical character of $\rho$ is a functor $\Tr_{\rho} : \mathcal{L}^{\refl} \cG \rightarrow \Vect_{\Fld}$.
\end{Prop}

\begin{proof}
Because $\Vect_{\Fld}$ is a 1-category, this statement is equivalent to $\Tr_{\rho}$ defining a functor $\pi_{\leq 1}(\mathcal{L}^{\refl} \cG) \rightarrow \Vect_{\Fld}$. It is straightforward to verify that, as defined above, the structure maps assemble to a functor $\Tr_{\rho}: \mathcal{L}_{\leq 1}^{\refl} \cG \rightarrow \Vect_{\Fld}$; see \cite[Proposition 4.10]{ganter2008}, \cite[\S 5]{rumynin2018}. We need to show that this functor factors through $\pi_{\leq 1}(\mathcal{L}^{\refl} \cG)$, that is, for each $z \in G_2$, the linear maps
\[
\Tr_{\rho}(g; f,x) \; , \; \Tr_{\rho}(g; (f,x) \circ \partial_g(z)) : \Tr_{\rho}(g) \rightarrow \Tr_{\rho}(f \partial(x) (g^{\pi(f)}) f^{-1})
\]
are equal. To do so, it suffices to show that $\Tr_{\rho}(g; \partial z, z^{-1} ({^{g}}z))$ is the identity endomorphism of $\Tr_{\rho}(g)$. This endomorphism is given by the diagram
\begin{equation}
\label{diag:idenComp}
\begin{gathered}
\begin{tikzcd}[column sep=3.5cm,row sep=0.80cm,scale=0.8]
V \arrow[r,bend right=20, "\rho(\partial z)"{name=A, below}]{} \arrow[r,bend right=-20, "\rho(\partial z)"{name=B, above}]{} \arrow[rrr,bend right=30, "\id_{V}"{name=X, below}]{} \arrow[rrr,bend right=-30, "\rho(g)"{name=Y, above}]{}
&
V \arrow[r, "\rho(g)"{name=D, above}]{} \arrow[r,bend right=40, "\id_{V}"{name=C, below}]{} \arrow[r,bend right=-40, "\rho(\partial(x) g)"{name=E, above}]{} &
V
\arrow[r,bend right=20, "\rho(\partial z^{-1})"{name=F, below}]{} \arrow[r,bend right=-20, "\rho(\partial z^{-1})"{name=G, above}]{} & V.
\arrow[shorten <=3pt,shorten >=3pt,Rightarrow,to path={(A) -- node[label=right:\footnotesize \scriptsize $\id$] {} (B)}]{}
\arrow[shorten <=3pt,shorten >=3pt,Rightarrow,to path={(C) -- node[label=right:\footnotesize \scriptsize $u$] {} (D)}]{}
\arrow[shorten <=3pt,shorten >=3pt,Rightarrow,to path={(D) -- node[label=right:\footnotesize \scriptsize $\rho(x)$] {} (E)}]{}
\arrow[shorten <=3pt,shorten >=3pt,Rightarrow,to path={(F) -- node[label=right:\footnotesize \scriptsize $\id$] {} (G)}]{}
\arrow[shorten <=3pt,shorten >=3pt,Rightarrow,to path={(X) -- node[label=right:\footnotesize \scriptsize $\rho^{-1}_{\partial z,e,\partial z^{-1}}$] {} (C)}]{}
\arrow[shorten <=3pt,shorten >=3pt,Rightarrow,to path={(E) -- node[label=right:\footnotesize \scriptsize $\rho_{\partial z,\partial(x) g,\partial z^{-1}}$] {} (Y)}]{}
\end{tikzcd}
\end{gathered}
\end{equation}
The composition rules in $\tcG$ give
\[
\begin{tikzpicture}[baseline= (a).base]
\node[scale=0.9] (a) at (0,0){
\begin{tikzcd}[column sep=6.5em,scale=0.8]
V
  \arrow[bend left=25]{r}[name=U,below]{}{\rho(\partial(x)g)} 
  \arrow[bend right=25]{r}[name=D]{}[swap]{\rho(g)}
& 
V
     \arrow[Rightarrow,to path={(D) -- node[label=right:\footnotesize \scriptsize $\rho(x)$] {} (U)}]{}
\end{tikzcd}
};
\end{tikzpicture}
=
\begin{tikzpicture}[baseline= (a).base]
\node[scale=0.9] (a) at (0,0){
\begin{tikzcd}[column sep=6.5em,scale=0.8]
V
  \arrow[bend left=25]{r}[name=U,below]{}{\rho(\partial z^{-1})} 
  \arrow[bend right=25]{r}[name=D]{}[swap]{\id_{V}}
&
V
 \arrow[Rightarrow,to path={(D) -- node[label=right:\footnotesize \scriptsize $\rho(z^{-1})$] {} (U)}]{}
  \arrow[bend left=25]{r}[name=U,below]{}{\rho(g)} 
  \arrow[bend right=25]{r}[name=D]{}[swap]{\rho(g)}
& 
V
  \arrow[bend left=25]{r}[name=U2,below]{}{\partial z} 
   \arrow[bend right=25]{r}[name=D2]{}[swap]{\id_{V}}
&
V
     \arrow[Rightarrow,to path={(D) -- node[label=right:\footnotesize \scriptsize $\id$] {} (U)}]{}
      \arrow[Rightarrow,to path={(D2) -- node[label=right:\footnotesize \scriptsize $\rho(z)$] {} (U2)}]{}
\end{tikzcd}
};
\end{tikzpicture}
\]
so that the diagram \eqref{diag:idenComp} can be rewritten as
\[
\begin{gathered}
\begin{tikzcd}[column sep=3.5cm,row sep=0.75cm,scale=0.8]
V \arrow[r,bend right=15, "\rho(\partial z)"{name=A, below}]{} \arrow[r,bend right=-15, "\id_{V}"{name=B, above}]{} \arrow[rrr,bend right=25, "\id_{V}"{name=X, below}]{} \arrow[rrr,bend right=-25, "\rho(g)"{name=Y, above}]{} \arrow[shorten <=3pt,shorten >=3pt,Rightarrow,to path={(A) -- node[label=right:\footnotesize \scriptsize $\rho(z^{-1})$] {} (B)}]{} &
V
\arrow[r,bend right=15, "\id_{V}"{name=C, below}]{}  \arrow[r,bend right=-15, "\rho(g)"{name=D, above}]{} &
V
\arrow[shorten <=3pt,shorten >=3pt,Rightarrow,to path={(C) -- node[label=right:\footnotesize \scriptsize $u$] {} (D)}]{}
  \arrow[r,bend right=15, "\rho(\partial z^{-1})"{name=E, below}]{} \arrow[r,bend right=-15, "\id_{V}"{name=F, above}]{} \arrow[shorten <=3pt,shorten >=3pt,Rightarrow,to path={(E) -- node[label=right:\footnotesize \scriptsize $\rho(z)$] {} (F)}]{} &
V.
\arrow[shorten <=3pt,shorten >=3pt,Rightarrow,to path={(X) -- node[label=right:\footnotesize \scriptsize $\rho^{-1}_{\partial z,e,\partial z^{-1}}$] {} (C)}]{}
\arrow[shorten <=3pt,shorten >=3pt,Rightarrow,to path={(D) -- node[label=right:\footnotesize \scriptsize $\rho_{e,g,e}$] {} (Y)}]{}
\end{tikzcd}
\end{gathered}
\]
Since the diagram
\[
\begin{gathered}
\begin{tikzcd}[column sep=3.5cm,row sep=1.25cm]
V
\arrow[r,bend right=15, "\rho(\partial z)"{name=A, below}]{} \arrow[r,bend right=-15, "\id_{V}"{name=B, above}]{} \arrow[rrr,bend right=20, "\id_{V}"{name=X, below}]{} \arrow[shorten <=3pt,shorten >=3pt,Rightarrow,to path={(A) -- node[label=right:\footnotesize \scriptsize $\rho(z^{-1})$] {} (B)}]{} &
V
\arrow[r, "\id_{V}"{name=D, above}]{} &
V
\arrow[r,bend right=15, "\rho(\partial z^{-1})"{name=E, below}]{} \arrow[r,bend right=-15, "\id_{V}"{name=F, above}]{} \arrow[shorten <=3pt,shorten >=3pt,Rightarrow,to path={(E) -- node[label=right:\footnotesize \scriptsize $\rho(z)$] {} (F)}]{} &
V
\arrow[shorten <=5pt,shorten >=5pt,Rightarrow,to path={(X) -- node[label=right:\footnotesize \scriptsize $\rho^{-1}_{\partial z,e,\partial z^{-1}}$] {} (D)}]{}
\end{tikzcd}
\end{gathered}
\]
is the identity $2$-morphism of $\id_{V}$, we conclude that \eqref{diag:idenComp} is equal to $u: \id_{V} \Rightarrow \rho(g)$.
\end{proof}

Working in this geometric set-up, we define the Real $2$-character
\gls{glChi}
of $\rho$ to be the holonomy of $\Tr_{\rho} : \mathcal{L}^{\refl} \cG \rightarrow \Vect_{\Fld}$. It is therefore a locally constant function $\chi_{\rho}: \mathcal{L}\mathcal{L}^{\refl} \cG \rightarrow {\Fld}$, or equivalently, $\chi_{\rho}:\Lambda \pi_{\leq 1}(\mathcal{L}^{\refl} \cG) \rightarrow {\Fld}$. Using either of these interpretations, we find that $\chi_{\rho}$ is function on the set
\[
\mathbb{G} = \{(g;f,x) \in G_0 \times G_1 \times G_2 \mid f \partial(x) (g^{\pi(f)}) f^{-1} = g \}
\]
which is invariant under the action of $\im(\partial_g)$, $g \in G_0$ and the conjugation action of morphisms in $\pi_{\leq 1} (\mathcal{L} \cG)$. When $G_1$ is trivially graded this gives a more precise description of the symmetries of $2$-characters than previously available. For example, it refines the $G_0$-conjugation invariance of $\chi_{\rho}$ proved in \cite[Proposition 5.1]{rumynin2018}.

The space of locally constant functions $\mathcal{L}\mathcal{L}^{\refl} \cG \rightarrow \Fld$, that is, Real $2$-class functions for $\cG$, is a direct sum of elliptic $2$-class functions, which are supported on $\mathbb{G} \cap (G_0 \times G_0 \times G_2)$ and Klein $2$-class functions, which are supported on $\mathbb{G} \cap (G_0 \times (G_1 \backslash G_0) \times G_2)$. The nomenclature stems from the appearance, modulo $\im(\partial)$, of the relations for the fundamental groups of the $2$-torus and Klein bottle in the definition of $\mathbb{G}$. This leads to an interpretation of $\mathcal{L}\mathcal{L}^{\refl} \cG$ in terms of moduli spaces of principal $\tcG$-bundles over the torus or Klein bottle, where in the latter case a compatibility condition with the orientation double cover is imposed, cf. \cite[Section 3.2]{mbyoung2019}. The $2$-character of the underlying $2$-representation of $\rho$ is in the elliptic sector. The Real nature of $\rho$ is reflected in the Klein sector of $\chi_{\rho}$ as well as the invariance of the elliptic sector of $\chi_{\rho}$ under the conjugation action of odd elements of $\pi_{\leq 1} (\mathcal{L} \cG)$.

\subsection{Induced categorical characters}
\label{sec:indChar}

In this section we determine the form of induced Real categorical and $2$-characters.

\begin{Thm}
\label{thm:inducedCatChar}
Let $\cG$ be a $\ZTwo$-graded crossed module with crossed submodule $\cH$. Assume that $G_1$ is finite and $H_2 = G_2$. Let $\rho$ be a $2$-representation of $\tcH$ on $2\Vect_{\Fld}$, which is Real if $\cH$ is non-trivially graded. Then there is a canonical isomorphism
\[
\Tr_{\Ind_{\tcH}^{\tcG} \rho} \simeq \Ind_{\mathcal{L}^{\refl} \cH}^{\mathcal{L}^{\refl} \cG} \Tr_{\rho}
\]
of vector bundles $\mathcal{L}^{\refl}\cG \rightarrow \Vect_{\Fld}$.
\end{Thm}

\begin{proof}
The proof is similar to those of \cite[Theorem 7.5]{ganter2008}, \cite[Theorems 7.3, 7.7]{mbyoung2018c}, which treat the cases in which $\cG$ is a group, possibly $\ZTwo$-graded. We focus on the differences.

Recall that $\Tr_{\rho}$ can be understood as a representation of (or vector bundle over) the groupoid $\pi_{\leq 1} (\mathcal{L}^{\refl} \cH)$; see the comments at the beginning proof of Proposition \ref{prop:RealCatChar}. Then $\Ind_{\mathcal{L}^{\refl} \cH}^{\mathcal{L}^{\refl} \cG} \Tr_{\rho}$ is the induced representation of $\pi_{\leq 1} (\mathcal{L}^{\refl} \cG)$. For required background on induction of groupoid representations, the reader is referred to \cite[\S 6.2]{ganter2008}.

Arguing as the beginning of the proof of Proposition \ref{prop:RealCatChar}, it suffices to prove the statement at the level of vector bundles $\mathcal{L}^{\refl}_{\leq 1} \cG \rightarrow \Vect_{\Fld}$. A choice of representatives of the connected components of $\mathcal{L}_{\leq 1}^{\refl} \cG$ induces an equivalence of groupoids
\[
\mathcal{L}_{\leq 1}^{\refl} \cG \simeq \bigsqcup_{g \in \pi_0(\mathcal{L}_{\leq 1}^{\refl} \cG)} BZ_{\cG}(g),
\]
where $Z_{\cG}(g) \leq G_2 \rtimes G_1$ is the stabilizer $\{ (f,x) \in G_1 \times G_2 \mid f \partial(x) (g^{\pi(f)}) f^{-1} = g\}$. Fix $g \in G_0$. We have a vector space isomorphism
\[
\Tr_{\Ind_{\cH}^{\cG} \rho}(g) \simeq \bigoplus_{\substack{t \in \sT \\ t^{-1} g t \in H_1}} \Tr_{\rho}(t^{-1} g t),
\]
where $\sT\subseteq G_1$ is a left transversal to $H_1$. 
We need to describe the action of $Z_{\cG}(g)$ on $\Tr_{\Ind_{\cH}^{\cG} \rho}(g)$. Write
\[
[g]_{\cG} \cap H_1 = \bigsqcup_{i=1}^n [h_i]_{\cH},
\]
where $[g]_{\cG}$ denotes the $G_2 \rtimes G_1$-orbit of $g$, and similarly for $[h_i]_{\cH}$. Let $\sT_i = \{t \in \sT \mid t^{-1} g t \in [h_i]_{\cH}\}$. For each $i \in \{1, \dots, n\}$, fix an element $t_i \in \sT_i$ and put $h_i = t_i^{-1} g t_i$. The assumption $H_2 = G_2$ implies that there is a canonical bijection
\[
Z_{\cG}(h_i) \slash Z_{\cH}(h_i) \simeq Z_{G_1}(h_i) \slash Z_{H_1}(h_i).
\]
We can therefore choose adapted representatives of $\sT_i \subset \sT$, as in \cite[Lemma 7.7]{ganter2008} (see also \cite[Lemma 7.4]{mbyoung2018c}). Consider the composition
\[
\phi_i: Z_{\cH}(h_i) \hookrightarrow Z_{\cG}(h_i) \xrightarrow[]{(f,x) \mapsto (t_i f t_i^{-1}, {^{t_i}}x)} Z_{\cH}(g).
\]
Then we can verify that $\Tr_{\Ind_{\cH}^{\cG} \rho}(g)$ is isomorphic to $\bigoplus_{i=1}^n \Ind_{\phi_i} \Tr_{\rho}(h_i)$, exactly as in the case $G_2 =\{e\}$.
\end{proof}

\begin{Cor}
\label{thm:induced2Char}
In the setting of Theorem \ref{thm:inducedCatChar}, the Real $2$-character of $\Ind_{\cH}^{\cG}\rho$ is
\[
\chi_{\Ind_{\tcH}^{\tcG} \rho} (g; f, x) = \frac{1}{c\vert H_1 \vert} \sum_{\substack{t \in G_1 \\ t^{-1}(g,f) t \in H_0 \times H_1}} \chi_{\rho}(t^{-1} g^{\pi(t)} t ; t^{-1} f t, {^{t^{-1}}}x),
\]
where $c=1$ if $\cH$ is trivially graded and $c=2$ otherwise.
\end{Cor}

\begin{proof}
By Theorem \ref{thm:inducedCatChar}, it suffices to compute the character of $\Ind_{\mathcal{L}^{\refl} \cH}^{\mathcal{L}^{\refl} \cG} \Tr_{\rho}$. This can be done using \cite[Proposition 6.11]{ganter2008}, giving the claimed result. 
\end{proof}

Considered as locally constant functions, Real $2$-characters are multiplicative:
\[
\chi_{\rho_1 \boxtimes \rho_2} = \chi_{\rho_1} \chi_{\rho_2}.
\]
This can be proved in the same way as \cite[Proposition 5.1]{rumynin2018}. See also \cite[Corollary 4.2]{ganter2015}. In particular, using the identification of Theorem \ref{thm:grothGroupDescr}, we find that each tuple $(g; f,x) \in \mathbb{G}$ defines a ring homomorphism
\[
\chi(g; f,x) : \mathbb{B}^{\Phi}_{\mathbb{Z}}(\cG) \rightarrow {\Fld}, \qquad \langle \rho, P \rangle \mapsto \chi_{\Ind_{\tcG_P}^{\tcG} \rho}(g; f,x).
\]
Motivated by \cite[Theorem 5.2]{rumynin2018}, we would like to understand this homomorphism in terms of marks. This can be done as follows.

Let $R$ be a ${\Fld}$-algebra. Assume that the cardinality of $\pi_1(\cG)$ is invertible in ${\Fld}$. Given $Q \in \mathcal{S}(G)$ and a semigroup homomorphism $\alpha: \Phi(Q) \rightarrow R^{\times}$, the associated mark homomorphism is the ${\Fld}$-linear map $f_Q^{\alpha} : \mathbb{B}_{{\Fld}}^{\Phi}(\cG) \rightarrow R$ defined by
\begin{equation}
\label{eq:markHomo}
f_Q^{\alpha} (\langle \rho, P \rangle) = \frac{1}{\vert P \vert} \sum_{\substack{ g \in G_1 \\ g Q g^{-1} \subset P}} \alpha(\Phi(\gamma_g: Q \rightarrow P)(\rho)).
\end{equation}
This is in fact a morphism of $R$-algebras \cite[\S 1]{gunnells2012}.

\begin{Cor}
\label{cor:Real2Char}
Assume that the cardinality of $\pi_1(\cG)$ is invertible in ${\Fld}$. Let $(g; f,x) \in \mathbb{G}$ and let $P$ be the subgroup of $\pi_1(\cG)$ generated by the images of $g$ and $f$. Then $\chi(g; f,x)$ is the mark homomorphism $\mathbb{B}^{\Phi}_{\mathbb{Z}}(\cG) \rightarrow {\Fld}$ determined via equation \eqref{eq:markHomo} by the restriction $\chi(g; f,x)_{\vert P} : \Phi(P) \rightarrow {\Fld}^{\times}$.
\end{Cor}

\begin{proof}
This follows by comparing equation \eqref{eq:markHomo} with the result of  Corollary~\ref{thm:induced2Char}.
\end{proof}

The analogy between $2$-character theory and Hopkins--Kuhn--Ravenel character theory \cite{hopkins2000} of Borel equivariant elliptic cohomology, as developed by Ganter and Kapranov \cite{ganter2008}, suggests the following problem.

\begin{Prob} \label{prob:HKR}
Interpret the results of Sections ~\ref{sec:RealCatChar} and ~\ref{sec:indChar} in terms of (transchromatic) Hopkins--Kuhn--Ravenel character theory \cite{hopkins2000, stapleton2013, lurie2019} of $2$-equivariant elliptic cohomology. In particular, relate Theorem~\ref{thm:inducedCatChar} and Corollary~\ref{thm:induced2Char} to the relevant transfer maps.
\end{Prob}

\begin{Rem}
A $3$-cocycle $\alpha \in Z^3(B \tcG, {\Fld}^{\times}_{\pi})$ on the classifying space determines a $\ZTwo$-graded central extension ${^{\alpha}}\tcG$ of $\tcG$ by ${\Fld}^{\times}$. An $\alpha$-twisted Real $2$-representation of $\tcG$ is then by definition a Real $2$-representation of ${^{\alpha}}\tcG$ whose restriction to ${\Fld}^{\times}$ is scalar multiplication. When $\cG$ is a finite group, the character theory of twisted representations is studied in \cite{ganter2016} and \cite{mbyoung2018c}. Together with the results of this section, this suggests an obvious candidate for the categorical character theory of twisted Real $2$-representations of finite $2$-groups in terms of vector bundles over $\mathcal{L}^{\refl} \cG$ which are twisted by the gerbe represented by the loop transgression of $\alpha$.
\end{Rem}


\setcounter{secnumdepth}{0}


\footnotesize

\bibliographystyle{plain}
\bibliography{mybib}

\printglossary[title={List of Symbols}]

\end{document}